\documentclass[11pt]{amsbook}
\usepackage{geometry}        
\usepackage[parfill]{parskip}  
\usepackage{amssymb}
\usepackage{epstopdf}

\usepackage[colorlinks=true, pdfstartview=FitV, linkcolor=blue, 
            citecolor=blue, urlcolor=blue]{hyperref}

\usepackage{graphicx}        

\makeindex
\usepackage[bottom]{footmisc}

\usepackage{newtxtext}       %
\usepackage{newtxmath}       

\counterwithin{figure}{chapter} 
\counterwithin{section}{chapter} 
\numberwithin{equation}{chapter}

	\usepackage[all,arc,import]{xy} 	
	\usepackage{bbm} 

    \usepackage{enumerate}
	\usepackage{amsmath,amsthm,amssymb, amsfonts} 
    
        \usepackage{tikz}
        \usepackage{tikz-cd}
    \usepackage{overpic}
    \usepackage{mathtools}
\newtheorem{thm}{thm}[chapter]

\theoremstyle{definition}

\newtheorem{proposition}[thm]{Proposition}
 \newtheorem{definition}[thm]{Definition} 
 \newtheorem{guess}[thm]{Guess} 
\newtheorem{theorem}[thm]{Theorem}
\newtheorem{example}[thm]{Example}
\newtheorem{exercise}[thm]{Exercise}
\newtheorem{example/exercise}[thm]{Example/Exercise}

\newtheorem{lemma}[thm]{Lemma}

\newtheorem{remark}[thm]{Remark}
\newtheorem{note[thm]}{Note}

\newtheorem{hypothesis}[thm]{Hypothesis}

\newtheorem{question-numbered}[thm]{Question}

\newtheorem{notation}[thm]{Notation}

\newtheorem{warning-numbered}[thm]{Warning}

\newtheorem{convention}[thm]{Convention}
\newtheorem{sadness}[thm]{Sadness}

\usepackage{hyperref} 

    \newcommand{\RR}{{\mathbb R}}

    \newcommand{\arw}{\ar}
    \newcommand{\conf}{\mathsf{Conf}}
    
    \newcommand{\diff}{\mathrm{Diff}}
    
    \newcommand{\bb}{\mathbb}
    \newcommand{\mrm}{\mathrm}
    \newcommand{\cal}{\mathcal}

    \newcommand{\rta}{\rightarrow}
    \newcommand{\xrta}{\xrightarrow}

    \DeclareMathOperator{\spec}{Spec}
    
    \newcommand{\del}{\partial}

    \newcommand{\diskn}{\disk_{n,\orr}}
    \newcommand{\disnn}{\diskn}

    \newcommand{\catf}[1]{{\mathsf{#1}}}

    \newcommand{\R}{\mathbb{R}}

    \newcommand{\K}{\mathbb{K}}

    \usepackage{needspace}

    \newcommand{\Aut}{\operatorname{Aut}}

    \newcommand{\End}{\operatorname{End}}

    \newcommand{\id}{\operatorname{id}}
    \newcommand{\Cat}{\catf{Cat}}
    
    \newcommand{\Set}{\catf{Set}}
    \let\to\undefined
    \newcommand{\to}{\longrightarrow}
    \let\mapsto\undefined
    \newcommand{\mapsto}{\longmapsto}

    \newcommand{\Vect}{\catf{Vect}}
    \newcommand{\Fun}{\catf{Fun}}
    
    \newcommand{\AlgBimod}{\catf{AlgBimod}}
    \newcommand{\ComAlg}{\catf{ComAlg}}

    \newcommand{\One}{\mathbbm{1}}
    
    \DeclareMathSymbol{\Phiit}{\mathalpha}{letters}{"08} 
    \DeclareMathSymbol{\Psiit}{\mathalpha}{letters}{"09}
    \DeclareMathSymbol{\Sigmait}{\mathalpha}{letters}{"06}
    \DeclareMathSymbol{\Xiit}{\mathalpha}{letters}{"04}
    \DeclareMathSymbol{\Piit}{\mathalpha}{letters}{"05}\let\Pi\undefined\newcommand{\Pi}{\Piit}
    \DeclareMathSymbol{\Gammait}{\mathalpha}{letters}{"00}
    \DeclareMathSymbol{\Omegait}{\mathalpha}{letters}{"0A}\let\Omega\undefined\newcommand{\Omega}{\Omegait}
    \DeclareMathSymbol{\Upsilonit}{\mathalpha}{letters}{"07}
    \DeclareMathSymbol{\Thetait}{\mathalpha}{letters}{"02}

    \let\Phi\undefined\newcommand{\Phi}{\Phiit}
    \let\Sigma\undefined\newcommand{\Sigma}{\Sigmait}
    \let\Psi\undefined\newcommand{\Psi}{\Psiit}
    \let\Gamma\undefined\newcommand{\Gamma}{\Gammait}
    
    \definecolor{Blue}  {rgb} {0.282352,0.239215,0.803921}
    \definecolor{Green} {rgb} {0.133333,0.545098,0.133333}
    \definecolor{Red}   {rgb} {0.803921,0.000000,0.000000}
    \definecolor{Violet}{rgb} {0.580392,0.000000,0.827450}

    \newcounter{jfc}

	\newcommand{\nc}{\newcommand}
	\nc{\enum}{\begin{enumerate}}
	\nc{\enumd}{\end{enumerate}}
    \nc{\eqn}{\begin{equation}}
    \nc{\eqnn}{\begin{equation}\nonumber}
    \nc{\eqnd}{\end{equation}}
    \nc{\orr}{\mathsf{or}}
    \nc{\fr}{\mathsf{fr}}
    \nc{\Emb}{\mathsf{Emb}}
    \nc{\morita}{\mathcal{M}\mathsf{orita}}
    \nc{\cdga}{\mathcal{C}\mathsf{dga}}
    \nc{\sing}{\mathsf{Sing}}
    \nc{\chain}{\mathcal{C}\mathsf{hain}}
    \nc{\Chain}{\mathcal{C}\mathsf{hain}}
    \nc{\Ass}{\mathcal{A}\mathsf{ss}}
    \nc{\Alg}{\mathcal{A}\mathsf{lg}}
    \nc{\Cob}{\mathcal{C}\mathsf{ob}}
    \nc{\cob}{\mathcal{C}\mathsf{ob}}
    \nc{\TTop}{\mathcal{T}\!\!\mathsf{op}}
	\nc{\spaces}{\mathsf{Spaces}}
	\nc{\kan}{\mathsf{Kan}}
    \nc{\mfld}{\mathcal{M}\mathsf{fld}}
    \nc{\mflddisc}{\mathsf{Mfld}}
    \nc{\disk}{\mathcal{D}\mathsf{isk}}
    \nc{\diskdisc}{\mathsf{Disk}}
    \nc{\kk}{\mathbbm{k}}
    \nc{\Cattwo}{\mathbbm{C}\mathsf{at}}
 	\nc{\DMO}{\DeclareMathOperator}
    \DMO{\op}{op}
    \DMO{\ob}{ob}
    \DMO{\colim}{colim}
    \DMO{\Sym}{Sym}
    \DMO{\fFun}{Fun}
    \DMO{\Conf}{Conf}
    \nc{\fg}{\mathfrak{g}}
    \nc{\EE}{\mathbb{E}}
    \nc{\CC}{\mathbb{C}}
    \nc{\cA}{\mathcal{A}}
    \nc{\cC}{\mathcal{C}}
    \nc{\cD}{\mathcal{D}}
    \nc{\cE}{\mathcal{E}}
    \nc{\LL}{\mathbb{L}}
    \nc{\ZZ}{\mathbb{Z}}
    \nc{\tensor}{\otimes}

\author{Hiro Lee Tanaka \\
\, \\
with contributions from  \\
\, \\
Araminta Amabel \\
Artem Kalmykov \\
Lukas M\"uller \\
}
\title{Lectures 	
	on 
	Factorization Homology, 
	$\infty$-Categories, 
	and 
	Topological Field Theories}

\begin{document}

\frontmatter
\maketitle


\chapter*{This is a preprint}
This is a preprint of the following work: Hiro Lee Tanaka (with contributions from Araminta Amabel, Artem Kalmykov, and Lukas M\"uller), {\em Lectures on Factorization Homology, $\infty$-Categories, and Topological Field Theories}, 2021, SpringerBriefs in Mathematical Physics, reproduced with permission of Springer International Publishing. The final authenticated version is available online at \url{http://dx.doi.org/10.1007/978-3-030-61163-7}.

Another link is here: \url{https://www.springer.com/gp/book/9783030611620}. 

While efforts have been made to retain enumerations of definitions, theorems, figures, et cetera, as close as possible to the published version, these enumerations may change through the editorial process, and this preprint may not be updated to reflect those changes due to contractual obligations. Changes to future editions of the book may also not be reflected in this pre-print.

\chapter*{Preface}

The aim of these lectures is to give an expository and informal introduction to the topics in the title. The target audience was imagined to be graduate students who are not homotopy theorists.

There will be two running themes in these lectures:

The first is the interaction between smooth topology and algebra. For example, we will see early on that the smooth topology of oriented, 1-dimensional disks (together with how they embed into each other) captures exactly the algebraic structure of being a unital associative algebra. Factorization homology allows us to take such algebraic structures and construct invariants of both manifolds and algebras through formal, categorical constructions. As an example, the category of representations of a quantum group $U_q(\fg)$ is braided monoidal (in fact, ribbon), and hence defines invariants of 2-dimensional manifolds equipped with additional structures. These invariants come equipped with mapping class group actions that recover well-known actions discovered in the representation theory of quantum groups~\cite{bzbj}.

Another place we will see this theme is in the cobordism hypothesis of Baez-Dolan~\cite{baezdolan} and Lurie~\cite{lurie-tft}. This articulates a precise way in which the smooth theory of framed cobordisms recovers the purely algebraically characterizable notion of objects having duals and morphisms having adjoints.

The second theme is the need to cohere various algebraic structures. A common question during the lectures was ``so is this actually associative or not?'' Coherence is a subtle beast, and we indicate some of the places where the language of $\infty$-categories is essential in articulating coherences, and in concluding the existence of continuous symmetries (thus allowing us to observe the existence of, say, mapping class group actions). The goal was to give an audience a feel for why this language is useful; we certainly did not give a user's guide, but we hoped to make the ideas less foreign.

It must be said that there are many topics that could have been covered, but certainly were not. Though this list is intended as an apology, the reader may also take it to be a list of further reading, or at least a list of exciting storylines to follow. 
\begin{enumerate}
	\item Most notably, there have been numerous works by Ayala and Francis; they have shown how factorization homology can be generalized in ways exhibiting the explicit coherence between Poincar\'e Duality for smooth manifolds and Koszul duality for $E_n$-algebras~\cite{ayala-francis-zero-pointed, ayala-francis-topological}. They have also developed a program for proving the cobordism hypothesis that does not involve intricate Cerf theory~\cite{ayala-francis-tft}. These are important developments that articulate how factorization homology and its generalizations may be central to smooth topology and higher algebra.
	\item To those interested in the interactions between representation theory and low-dimensional topology, let us mention another disservice. We do not touch on the strategy for producing quantum link invariants by computing invariants of chain maps
	\eqnn
	\int_{S^1 \times D^2} A \to
	\int_{S^3} A
	\eqnd
induced by framed embeddings of knots in $S^3$. (Here, $A$ is an $E_3$-algebra one can construct out of a semisimple Lie algebra.) As far as HLT understands, this is work in progress of Costello, Francis, and Gwilliam. We also do not touch upon the work of Ben-Zvi-Brochier-Jordan~\cite{bzbj}.
	\item Another perspective lacking in these notes is the work of Costello and Costello-Gwilliam~\cite{costello-gwilliam} 
	on factorization algebras, which makes more manifest the connections to perturbative quantum field theories and deformation theory. Aside from the physical considerations, these works also give deformation-theoretic sources of locally constant factorization algebras and exhibit fruitful connections to shifted Lie and Poisson algebras in characteristic 0. Indeed, we make no mention of Kontsevich formality theorems or of $P_n$-algebras in these notes. 
	\item We do not touch the algebro-geometric avatars expressed through Ran spaces, as developed by Francis-Gaitsgory~\cite{francis-gaitsgory} and as utilized by Gaitsgory-Lurie~\cite{gaitsgory-lurie} to prove Weil's conjecture on Tamagawa numbers. 
	\item Lurie's topological account in~\cite{higher-algebra} also uses the Ran space perspective; there he refers to (a slightly more general version of) factorization homology as topological chiral homology.
	\item We do not discuss applications to configuration spaces (though this is a manifestly natural topic to consider) and relations to Lie algebra homology. Representative works include those of Knudsen~\cite{knudsen-1,knudsen-2} and Kupers-Miller~\cite{kupers-miller}. The algebro-geometric techniques from the previous bullet point have also been used to prove homological stability for configuration spaces arising in positive characteristic by Ho~\cite{ho-free-factorization-algebras}. 
	\item There is no mention in these notes of the intersection with derived algebraic geometry. For example, when $A$ is a commutative cdga in non-positive degree, its Hochschild chain complex is at once (functions on) the derived loop space of $A$ (in the algebro-geometric sense) and the factorization homology of the circle with coefficients in $A$. Put another way, Hochschild chains also have an interpretation as the sheaf cohomology of the structure sheaf of the mapping stack $Map(S^1,\spec A)$ where $S^1$ is the constant stack. On the other hand, when $\spec A$ is replaced by a stack which is non-affine, $\tensor$-excision fails. Moreover, when one further considers stacks not on the site of (commutative) affine schemes, but on a site of $\EE_n$-stacks (where ``affine'' objects are now Spec of $\EE_n$-algebras), these mapping stacks are expected to be far more sensitive manifold invariants. Many are expected to arise as the result of deformation quantizations of shifted symplectic structures as discussed in the next bullet point. (For more, see~\cite{benzvi-francis-nadler} and~\cite{benzvi-nadler-loops-1} and~\cite{ayala-francis-primer}.) 
	\item Another open problem relating to derived algebraic geometry is the problem of quantizing shifted symplectic stacks. For example, the AKSZ formalism---as reinterpreted in~\cite{ptvv}---suggests we should be able to quantize (the category of sheaves on) certain mapping stacks with shifted Poisson or shifted symplectic structures. This was carried out in some examples in~\cite{cptvv}, and for example, one obtains a braided monoidal quantization of the category of finite-dimensional representations of a Lie algebra $\fg$. These quantizations are expected to lead to interesting examples of TFTs, in contrast to the non-quantized TFTs.  (The naive TFT $W \mapsto Map(W,BG)$ is only sensitive to the homotopy type of manifolds $W$.) The intuition that certain symplectic structures should give rise to quantizations is of course an old story with many modern narratives, but it should be mentioned that a hugely influential starting point of algebraically deformation-quantizing Poisson structures is Kontsevich's formality theorem for Poisson manifolds~\cite{kontsevich-poisson}. 
\end{enumerate}

For those seeking more, we refer to the excellent notes of Teleman~\cite{teleman-lectures}, who also takes a far more representation-theoretic basepoint for their exposition. Other resources include~\cite{ayala-francis-topological},~\cite{ayala-francis-primer},~\cite{ginot}, and~5.5 of~\cite{higher-algebra}.

Finally, we note that the contents of these notes expand significantly on the contents of the delivered lectures. This was done in the hopes of having a somewhat more complete written account of the story. We warn future speakers that they should not attempt to fit a chapter of these notes into a single lecture. We also remark that most of Chapter~\ref{chapter. infty cats} did not appear in the original lectures; as such, other chapters do not assume deep knowledge of higher categorical tools. 

\begin{convention}\label{convention:manifolds-small}
Unless noted explicitly, every manifold in this work is smooth and paracompact. For simplicity the reader may assume that every manifold may be obtained by a {\em finite} sequence of open handle attachments---for example, every manifold arises as the interior of some compact manifold with boundary. One can treat larger manifolds (for example, countably infinite-genus surfaces), but we would have to say a few words about preserving filtered colimits. See Remark~\ref{remark.open-exhaustion-colimits}.
\end{convention}

\vspace{\baselineskip}
\begin{flushright}\noindent
San Marcos, Texas \hfill {\it Hiro Lee Tanaka}\\
July 2020 \hfill {\it \, }\\
\end{flushright}

\chapter*{Acknowledgements}

These notes were based on lectures given by HLT, and notes taken by AA, AK, and LM, at the 2019 Summer School on Geometric Representation Theory and Low-Dimensional Topology, hosted at the International Centre for Mathematical Sciences in Edinburgh, Scotland. We would like to thank ICMS for providing facilities and administrative support for the summer school. We also thank the organizers---Dan Freed, David Jordan, Peter Samuelson, and Olivier Schiffmann---and David Jordan especially for facilitating the organizational mechanisms that led to a timely production of these notes. HLT would like to thank the students at the school, whose questions gave rise to many of the clarifying remarks in these notes; his coauthors, who gave feedback and drew all the nice pictures (HLT drew the not-nice ones); those who gave great feedback on a draft of these notes (D. Ayala, J. Francis, G. Ginot); and Pavel Safronov, who in so many words remarked that a version of these notes would have been useful five years ago. To say the least, the idea of not hearing the same statement five years from now was motivating.

The summer school and this work were supported under NSF Grant \#1856643 and European
Research Council grant no. 637618. AA was supported by an NSF GRFP (NSF Grant No. 1122374). AK was supported by an NCCR SwissMAP grant of the Swiss National Science Foundation.

\tableofcontents

\mainmatter

\chapter{The one-dimensional case}
\label{chapter. dim 1}

The goal of our first lecture is to see that the geometry of oriented embeddings of open, 1-dimensional disks gives rise to the algebra of associativity. That is, local building blocks of oriented 1-manifolds encode algebraic structure. We then explore how these local structures could be extended to define invariants of all oriented 1-dimensional manifolds. We study one extension in particular, given by factorization homology.

Of course, we know how to classify (oriented) 1-manifolds, so this invariant gains little for us as a manifold invariant. But we will see that an interesting algebraic invariant (the Hochschild chain complex) appears. Moreover, by exploiting the continuity of factorization homology, we will discover that the Hochschild complex naturally admits a circle action. 

We will move to higher-dimensional manifolds in the Chapter~\ref{chapter. fact hom in higher dims}.

\section{The algebra of disks} 
\label{sec: 1-disks}

We begin with the following definition.

\begin{definition}
We let $\diskdisc_{1,\orr}$ 
\index{$\diskdisc_{1,\orr}$}
denote the category of oriented, 1-dimensional disks. Objects are finite disjoint unions of oriented 1-dimensional open disks. The collection of morphisms 
	$
	\hom(X, Y)
	$
consists of all smooth embeddings $j: X \hookrightarrow Y$ respecting orientations.
\end{definition}

\begin{remark}
Up to isomorphism, we can enumerate the objects as follows:
	\eqnn
	\emptyset, \qquad \RR, \qquad \RR \coprod \RR, \qquad \ldots \qquad \RR^{\coprod k} , \qquad \ldots.
	\eqnd
Note that we allow for the empty disjoint union. 
\end{remark}

This category comes with a natural operation of taking disjoint union: If $X$ is a union of $k$ oriented disks and $Y$ is a union of $l$ oriented disks, then $X \coprod Y$ is a union of $k+l$ oriented disks. There are natural isomorphisms
	\eqnn
	X \coprod Y \cong Y \coprod X
	\eqnd
and $\coprod$ renders $\diskdisc_{1,\orr}$ a symmetric monoidal category. The empty set $\emptyset$ is the monoidal unit of this category, as we can supply natural isomorphisms $X \coprod \emptyset \cong X \cong \emptyset \coprod X$.

\begin{notation}
Fix a field $\kk$. We denote by $\Vect_{\kk}^{\otimes_{\kk}}$ the category of $\kk$-vector spaces. The superscript $\tensor_{\kk}$ indicates that we view $\Vect_{\kk}$ as a symmetric monoidal category equipped with the usual tensor product of vector spaces: 
	\eqnn
	\Vect_{\kk} \times \Vect_{\kk} \to \Vect_{\kk},
	\qquad
	(V,W) 
	\mapsto V\otimes_{\kk} W.
	\eqnd
\end{notation}

\newenvironment{functor-props}{
	  \renewcommand*{\theenumi}{(F\arabic{enumi})}
	  \renewcommand*{\labelenumi}{(F\arabic{enumi})}
	  \enumerate
	}{
	  \endenumerate
}

\index{$\diskdisc_{1,\orr}^{\coprod}$}
That's the set-up. For no good reason whatsoever, I want to now study functors 
	\eqnn
	F: \diskdisc_{1,\orr}^{\coprod} \rightarrow \Vect_{\kk}^{\otimes_{\kk}}
	\eqnd
where:
\begin{functor-props}
\item $F$ is symmetric monoidal,
\item \label{item.isotopies} Isotopic embeddings are sent to the same linear map.
\end{functor-props}

\begin{example}
Thanks to the second requirement, such functors become quite tractable. For instance, consider the orientation-preserving smooth embeddings of $\RR$ into itself. This is the space of all smooth, strictly increasing functions $\RR \rightarrow \RR$. These embeddings are all isotopic to a very simple one -- the identity map $id_{\RR}: \RR \rightarrow \RR$. You can check this fact as an exercise.

Therefore, if the functor $F$ satisfies \ref{item.isotopies}, $F$ sends the whole collection of morphisms 
	\eqnn
	\hom_{\diskdisc_{1,\orr}}(\RR,\RR)
	\eqnd
to the identity element 
	\eqnn
	id_{F(\RR)}\in \hom_{\Vect_{\kk}}(F(\RR),F(\RR)).
	\eqnd 
\end{example}

\begin{warning-numbered}
If you are not used to monoidal categories, when you hear that $F$ must be symmetric monoidal, you probably imagine that $F$ ``respects'' the symmetric monoidal structures on both sides. But I put ``respect'' in quotes because being symmetric monoidal is not merely a property of $F$; the {\em data} of $F$ being symmetric monoidal means we also {\em supply} natural isomorphisms
	\eqnn
	F(X\coprod Y) \xrightarrow{\simeq} F(X)\otimes_{\kk} F(Y).
	\eqnd
We also demand that the swap isomorphism $X \coprod Y \cong Y \coprod X$ is sent to the swap isomorphism $F(X) \tensor F(Y) \cong F(Y) \tensor F(X)$ in a way compatible with the above natural isomorphism.

\end{warning-numbered}

Here is the main result I'd like us to understand. It is fundamental to everything that follows.

\begin{theorem}\label{theorem.associativity}
Suppose $F: \diskdisc_{1,\orr}^{\coprod} \to \Vect_{\kk}^{\tensor_\kk}$ is a symmetric monoidal functor such that if $j$ and $j'$ are isotopic embeddings, then $F(j) = F(j')$.
Then the data of $F$ is equivalent to the data of a unital associative $\kk$-algebra.
\end{theorem}

\begin{proof}[Sketch]
Let us set the notation
	\eqnn
	A = F(\RR).
	\eqnd
Since $F$ is symmetric monoidal, the empty set $\emptyset$ is sent to the monoidal unit of $\Vect_{\kk}^{\otimes_{\kk}}$, which is the base field $\kk$. More generally, any object of $\disk_{1,\orr}$ is a disjoint union of $l$ disks, and $F$ sends $\RR^{\coprod l}$ to the tensor power $A^{\otimes l}$. 

We now study $F$'s effect on morphisms. Consider the morphism set 
	$
	\hom_{\diskdisc_{1,\orr}^{\coprod}}(\emptyset,\RR).
	$ 
This set consists of one element, given by the embedding $\eta: \emptyset \rightarrow \RR$. $F$ determines a map
	\eqnn
	u=F(\eta): \kk \rightarrow A
	\eqnd 
which we will call the unit. 

Now consider $\hom_{\diskdisc_{1,\orr}^{\coprod}}(\RR^{\coprod 2},\RR)$. Up to isotopy, there are exactly two embeddings $\RR \coprod \RR \to \RR$. To see this, fix a smooth orientation-preserving embedding $j: \RR \coprod \RR \to \RR$.  Let us order (non-canonically) the connected components of the domain $\RR \coprod \RR$, and consider the image of each component under $j$; the orientation of the codomain $\RR$ respects this ordering, or does not. This distinguishes the two isotopy classes to which $j$ can belong.

Let us denote an order-respecting embedding by $j_1$, and a non-respecting embedding by $j_2$. 

Take
	\eqnn
	m = F(j_1): A\otimes_{\kk} A \rightarrow A.
	\eqnd
This is our multiplication. We claim that $m$ and $u$ determine a unital associative algebra structure on $A$.

\begin{itemize}
    \item Unit: we need to verify that the diagram of vector spaces
    	\eqn\label{eqn:unitality}
        \begin{tikzcd}
        A \arrow[r,"\simeq"] \arrow[drr,"id_A"'] & A \otimes_{\kk} \kk \arrow[r,"id_A\otimes_{\kk} u"] &  A\otimes_{\kk} A \arrow[d,"m"] \\
        & & A
        \end{tikzcd}
        \eqnd
    commutes. Consider the following diagram in $\diskdisc_{1,\orr}$:
        \begin{center}
            \begin{tikzcd}\label{assoc-alg}
            \RR \arrow[r,"\simeq"] \arrow[drr,"id_{\RR}"'] & \RR \coprod \emptyset \arrow[r,"id_{\RR}\coprod \eta"] &  \RR \coprod \RR \arrow[d,"j_1"] \\
            & & \RR.
            \end{tikzcd}
        \end{center}
    This diagram is not commutative---the embeddings of $\RR$ to itself, given by $j_1 \circ (\id_\RR \coprod \eta)$ and $\id_\RR$, need not be equal---but there does exist an isotopy between the two embeddings. That is, the diagram commutes up to isotopy. Thus, by \ref{item.isotopies}, the induced diagram \eqref{eqn:unitality} is commutative.
    
    \item Associativity: we need to check
        \begin{equation}\label{eqn:associativity}
            \begin{tikzcd}
            A \otimes_{\kk} A \otimes_{\kk} A \arrow[r,"m\otimes_{\kk} id_A"] \arrow[d,"id_A\otimes_{\kk} m"] & A\otimes_{\kk} A \arrow[d,"m"] \\
            A\otimes_{\kk} A \arrow[r,"m"] & A
            \end{tikzcd}
        \end{equation}
        commutes. 
        Consider the embeddings $j_1 \circ (j_1 \coprod \id_\RR)$ and $j_1 \circ (\id_\RR \coprod j_1)$, which we may draw as follows:
        \begin{center}
        \begin{tikzcd}
            \begin{tikzpicture}
            \node at (0.5,0.2) {1};
            \node at (2.5,0.2) {2};
            \node at (4.5,0.2) {3};
            \draw [brown,->] (0,0) -> (1,0);
            \draw [brown,->] (2,0) -> (3,0);
            \draw [brown,->] (4,0) -> (5,0);
            \end{tikzpicture} 
            \arrow[r] \arrow[d]& 
            \begin{tikzpicture}
            \node at (0.5,0.2) {1};
            \node at (2.5,0.2) {2};
            \node at (4.5,0.2) {3};
            \draw [brown,->] (0,0) -> (1,0);
            \draw [brown,->] (2,0) -> (5,0);
            \draw [brown,domain=-20:20]  plot ({cos(\x)+2}, {sin(\x)});
            \draw [brown,domain=160:200]  plot ({cos(\x)+3}, {sin(\x)});
            \draw [brown,domain=-20:20]  plot ({cos(\x)+4}, {sin(\x)});
            \draw [brown,domain=160:200]  plot ({cos(\x)+5}, {sin(\x)});
            \end{tikzpicture} 
            \arrow[d]
            \\
            \begin{tikzpicture}
            \node at (0.5,0.2) {1};
            \node at (2.5,0.2) {2};
            \node at (4.5,0.2) {3};
            \draw [brown,->] (0,0) -> (3,0);
            \draw [brown,->] (4,0) -> (5,0);
            \draw [brown,domain=-20:20]  plot ({cos(\x)}, {sin(\x)});
            \draw [brown,domain=160:200]  plot ({cos(\x)+1}, {sin(\x)});
            \draw [brown,domain=-20:20]  plot ({cos(\x)+2}, {sin(\x)});
            \draw [brown,domain=160:200]  plot ({cos(\x)+3}, {sin(\x)});
            \end{tikzpicture} 
            \arrow[r] &
            \begin{tikzpicture}
            \node at (0.5,0.2) {1};
            \node at (2.5,0.2) {2};
            \node at (4.5,0.2) {3};
            \draw [brown,->] (0,0) -> (5,0);
            \draw [brown,domain=-20:20]  plot ({cos(\x)}, {sin(\x)});
            \draw [brown,domain=160:200]  plot ({cos(\x)+1}, {sin(\x)});
            \draw [brown,domain=-20:20]  plot ({cos(\x)+2}, {sin(\x)});
            \draw [brown,domain=160:200]  plot ({cos(\x)+3}, {sin(\x)});
            \draw [brown,domain=-20:20]  plot ({cos(\x)+4}, {sin(\x)});
            \draw [brown,domain=160:200]  plot ({cos(\x)+5}, {sin(\x)});
            \end{tikzpicture} .
        \end{tikzcd}
        \end{center}
    The compositions $j_1 \circ (j_1 \coprod \id_\RR)$ and $j_1 \circ (\id_\RR \coprod j_1)$ are non-equal; but there is an isotopy between the two embeddings. Because $F$ collapses isotopic embeddings to the same linear map, the associativity axiom is satisfied.
        
    For the reverse implication, let $A$ be a unital associative algebra, i.e. we are given maps $m: A\otimes_{\kk} A \rightarrow A$ and $u: \kk \rightarrow A$ satisfying the unit and associativity axioms. We can reverse-engineer a functor $F$ by the same rules as above. We leave it to the reader to check that all embeddings, up to isotopy, can be factored as a composition of (i) component-permutations $\RR^{\coprod l} \to \RR^{\coprod l}$, and (ii) disjoint unions of $\eta$  and of $j_1$. Therefore, the property of $A$ being associative and unital is enough to determine the whole functor $F$.
\end{itemize}
\end{proof}

\begin{remark}
What happened to $j_2$, otherwise known as the ``other'' embedding $\RR \coprod \RR \to \RR$? We may express $j_2$ as a composition $j_1 \circ \sigma$ where $\sigma$ is the swap map $\RR \coprod \RR \cong \RR \coprod \RR$ switching the components. Because $F$ was demanded to be symmetric monoidal, we have
	\eqnn
	F(j_2) 
		= F(j_1 \circ \sigma)
		= m \circ \sigma_{\Vect_{\kk}^{\tensor_\kk}}.
	\eqnd
(Here, $\sigma_{\Vect_{\kk}^{\tensor_\kk}}$ is the swap map $V \tensor W \cong W \tensor V$.)
Hence $F(j_2)$ is no new data; it encodes the canonical ``opposite'' multiplication given by any associative algebra. 

Importantly, note that there is no reason to prefer $j_2$ over $j_1$, and in particular, we break symmetry when we decide whether a functor $F$ ought to determine the algebra $A$ with multiplication $m=F(j_1)$, or with multiplication $m=F(j_2)$. This symmetry is due to the automorphism of the category of associative algebras (which sends an algebra to its opposite), and the automorphism of the category of symmetric monoidal functors $F$ (which precomposes $F$ with the ``orientation-reversal'' automorphism of the category $\diskdisc_{1,\orr}$).
\end{remark}

\begin{remark}
As stated, Theorem~\ref{theorem.associativity} is a bit vague. One can in fact upgrade the statement to induce an equivalence of categories (and in particular, map natural transformations of functors to maps of associative algebras).
\end{remark}

In a sense, a miracle happened: an algebraic operations turns out to be ``encoded'' in something geometric and rather simple. 

\begin{remark}\label{remark: dissatisfaction}
However, some things leave us unsatisfied. 

\enum
\item First, why restrict ourselves just to open disks? Could we extend $F$ to other 1-manifolds---for example, to a circle? Then $F(S^1)$ would be an invariant of a circle, and this may even be an interesting invariant of the algebra $A$.

\item \label{item:dissatisfied-isotopy}Second, the isotopy business. It seems unnatural to require isotopy invariance from the functors. Let me be careful about this point. 

Imagine that we have managed to assign an invariant to a circle, say $F(S^1)$. Of course, up to isotopy, any two orientation-preserving diffeomorphisms of the circle are equivalent. But there are {\em interesting} isotopies---for example, the identity map $\id_{S^1}: S^1 \to S^1$ admits self-isotopies given by full rotations. Wouldn't our invariant $F(S^1)$ be more interesting if it detected these self-isotopies? 
\enumd
\end{remark}

We will come back to this last point in a moment. Let us first fantasize what an invariant of the circle might look like.

\section{The co-center: A first stab at a circle invariant}\label{section.circle-guess}

\begin{definition}
Let $\mflddisc_{1,\orr}^{\coprod}$
\index{$\mflddisc_{1,\orr}^{\coprod}$}
be the category of oriented, 1-dimensional manifolds with finitely many connected components. Morphisms are smooth embeddings that respect orientation. This category is also symmetric monoidal with respect to the disjoint union $\coprod$. 
\end{definition}

\begin{remark}
Note the (symmetric monoidal) inclusion of the category $\diskdisc_{1,\orr}^{\coprod}$ into the category $\mflddisc_{1,\orr}^{\coprod}$.
\end{remark}

Suppose we are given an associative algebra $A$, which we identify with a symmetric monoidal functor
	\eqnn
	F: \diskdisc_{1,\orr}^{\coprod} \to \Vect_{\kk}^{\tensor_\kk}
	\eqnd
(using Theorem~\ref{theorem.associativity}). 
Can we extend $F$ to a dashed functor as below?
	\eqnn
	\xymatrix{
	\diskdisc_{1,\orr}^{\coprod} \ar[d] \ar[r]^F & \Vect_{\kk}^{\tensor_\kk} \\
	\mflddisc_{1,\orr}^{\coprod} \ar@{-->}[ur].
	}
	\eqnd
Let us imagine what such an extension would ``want'' to apply to $S^1$. We will denote by $F(S^1)$ the vector space assigned by this extension.

From embeddings $\RR^{\coprod k} \rightarrow S^1$ we get maps $\phi_k: A^{\otimes k} \rightarrow F(S^1)$ for any $k\geq 0$. In fact, any such embedding factors through a single connected interval of $S^1$, so these linear maps must factor through the multiplication of $A$:
	\eqnn
	\phi_k: A^{\tensor k} \to A \to F(S^1).
	\eqnd
Moreover, since we are on a circle, one can isotope a configuration of $k$ intervals cyclically. 
For instance, for $k=2$, fixing an embedding $h: \RR \coprod \RR \to S^1$,  we witness an isotopy
	\eqnn
	h \circ \sigma \sim h.
	\eqnd
See Figure~\ref{figure. circle interval swap}.

\begin{figure}[h]
\caption{
An image of an isotopy realizing a cyclic permutation of embedded intervals in a circle. The vertical direction is the ``time'' direction along which the isotopy evolves.}\label{figure. circle interval swap}
	\eqnn
    \begin{overpic}[scale=1]{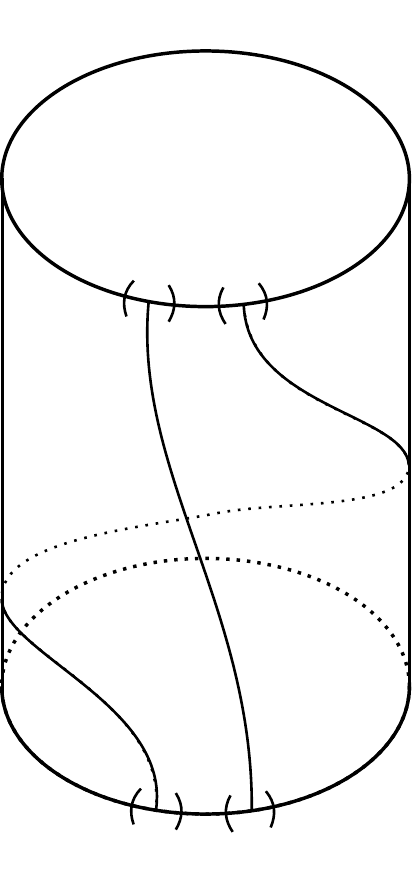}
    \put(16,67){$1$}
    \put(16,4){$2$}
    \put(26,67){$2$}
    \put(27,4){$1$}
    \end{overpic}
    \eqnd
\end{figure}
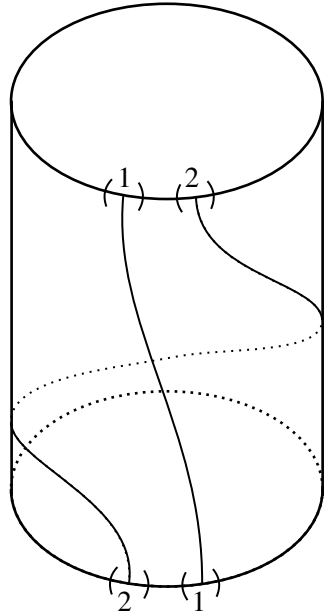

In particular, the linear map $\phi_2$ satisfies the property that $\phi_2 \circ \sigma = \phi_2$ where $\sigma$ is the swap map $A \tensor A \cong A \tensor A$. For example, for elements $x_1,x_2 \in A$, we have
	\eqnn
	\phi_1(x_2 x_1) = \phi_2(x_2 \tensor x_1) = \phi_2(x_1 \tensor x_2) = \phi_1(x_1 x_2).
	\eqnd
Let $[A,A] \subset A$ denote the vector subspace generated by commutators $x_1x_2 - x_2 x_1$. We see that $\phi_1$ must factor through the quotient
	\eqnn
	A/[A,A].
	\eqnd
There is at this point a naive:

\begin{guess}
The invariant $F(S^1)$ of the circle is isomorphic to $A/[A,A]$.
\end{guess}

\begin{remark}
Sometimes, $A/[A,A]$ is called the {\em cocenter} of $A$. If $A$ is commutative, the projection map from $A$ to the cocenter is an isomorphism. 
\end{remark}

The vector space $A/[A,A]$ has another presentation:

\begin{proposition}\label{prop:cocenter}
Let $A$ be an associative $\kk$-algebra. There exists a natural isomorphism
	\eqnn
	A/[A,A] \cong A\otimes_{A\otimes_{\kk} A^{op}}A
	\eqnd
of $\kk$-vector spaces. 
\end{proposition}

\begin{remark}\label{remark. A is bimodule over itself}
To make sense of the tensor product, note that the algebra $A$ is naturally an $(A,A)$-bimodule by left and right multiplication. This means $A$ can be considered both a right- and left-module over the ring $A\otimes_{\kk} A^{op}$, where $A^{op}$ is the algebra with the opposite multiplication.
\index{$A^{\op}$}
\index{opposite algebra}
By {\em naturality}, we mean that any unital $\kk$-algebra map $A \to B$ fits into a commutative diagram
	\eqnn
	\xymatrix{
	A/[A,A] \ar[r] \ar[d] &A\otimes_{A\otimes_{\kk} A^{op}}A\ar[d]\\
	B/[B,B] \ar[r]  & B\otimes_{B\otimes_{\kk} B^{op}}B.
	}
	\eqnd
\end{remark}

\begin{remark}
The proof is left as Exercise~\ref{exercise:cocenter}.
If you have never proven this on your own, I highly encourage you to do it.
\end{remark}

We will come back to the expression in Proposition~\ref{prop:cocenter} by the end of this lecture.

\section{$(\infty,1)$-categories, a first pass} 

\index{$(\infty,1)$-category}

Let us return to~\eqref{item:dissatisfied-isotopy} of Remark~\ref{remark: dissatisfaction}. What to do with isotopies?

The frustration of~\eqref{item:dissatisfied-isotopy} might inspire us to contemplate a category whose collection of morphisms has a topology. If our invariants are functors that respect these topologies, we have a hope of seeing the topology of the space of embeddings from $S^1$ to $S^1$, and hence these non-trivial rotational isotopies.

We will talk more about this next lecture, but let's use this motivation to talk a bit about $(\infty,1)$-categories.

On a first pass, you can think of an $(\infty,1)$-category as having a collection of objects and, for each pair of objects $(X_0,X_1)$, a \emph{topological space} of morphisms $\hom(X_0,X_1)$ (instead of a mere set of morphisms). We demand that the composition maps
	\eqnn
	\hom(X_1,X_2) \times
	\hom(X_0,X_1) \to
	\hom(X_0,X_2) ,
	\qquad
	(f_{12},f_{01}) \mapsto f_{12} \circ f_{01}
	\eqnd
be continuous.  (See Remark~\ref{remark. infinity 1 terminology} for why one might call this an $(\infty,1)$-category.)

For example, one can endow the collection of smooth, orientation-preserving embeddings
	\eqnn
	\{j: X \to Y\}
	\eqnd
with a topology---e.g., the weak Whitney $C^\infty$ topology. Composition of embeddings is a continuous operation. In this way, we can define

\begin{definition}
We let
	\eqnn
	\disk_{1,\orr}
	\eqnd 
\index{$\disk_{1,\orr}$}
denote the $(\infty,1)$-category whose objects are (disjoint unions of) oriented 1-dimensional open disks, and whose morphism spaces are given by the space of smooth, orientation-preserving embeddings.
\end{definition}

\begin{warning-numbered}
There is a font difference between $\diskdisc$ (the category from before) and $\disk$ (the $(\infty,1)$-category, which sees the topology of embedding spaces).
\end{warning-numbered}

\begin{definition}[Homotopy]
\index{homotopy in an $(\infty,1)$-category}
Given an $(\infty,1)$-category $\cC$ and two objects $X,Y$, fix two morphisms $f_0,f_1\in \hom_{\cC}(X,Y)$. A continuous path from $f_0$ to $f_1$ in the space $\hom_{\cC}(X,Y)$ is called a {\em homotopy} from $f_0$ to $f_1$. 
\end{definition}

\begin{warning-numbered}
Consider the example of $\disk_{1,\orr}$. Confusingly, a homotopy is not the same thing as a smooth isotopy, as one might choose a continuous path in $\hom_{\disk_{1,\orr}}$ which does not give rise to a \emph{smooth} isotopy; regardless, by smooth approximation, two morphisms are homotopic if and only if they are smoothly isotopic as embeddings. 
\end{warning-numbered}

\begin{warning-numbered}
You may recognize that what I have called an ``$(\infty,1)$-category'' in this section is really a ``topologically enriched category.'' This is one model for an $(\infty,1)$-category. However, Joyal's model of $(\infty,1)$-categories (often called quasi-categories, weak Kan complexes, or simply $\infty$-categories) is often much better behaved, and allows the construction of far more sophisticated examples of $(\infty,1)$-categories. See~\ref{section. quasi categories}.
\end{warning-numbered}

For now, you can think of a functor between $(\infty,1)$-categories as a functor in the usual sense, and for which the maps between morphism spaces are \emph{continuous}. 

What are other examples of $(\infty,1)$-categories? 

\begin{example}
Every ordinary category is an $(\infty,1)$-category---we simply treat the set of morphisms as a discrete topological space. 
\end{example}

\begin{example}
As a sub-example, the category of vector spaces $\Vect_{\kk}$ is an $(\infty,1)$-category with discrete morphism spaces. That is, two linear maps $f: V \to W$ are homotopic if and only if they are equal.
\end{example}

\begin{example}[Chain complexes]\label{example:Dold-Kan}
\index{Dold-Kan}
We are now going to sketch the idea of an $(\infty,1)$-category
	\eqnn
	\chain_{\kk}
	\eqnd
\index{$\chain_{\kk}$}
of cochain complexes over a field $\kk$. Its objects are cochain complexes, and for now we will content ourselves with only sketching the space of morphisms. To this end, fix two cochain complexes $V$ and $W$. 
One can construct a space $\hom(V,W)$ which is combinatorially defined---i.e., built of simplices:
\begin{itemize}
\item Vertices of $\hom(V,W)$ are usual morphisms of chain complexes---that is, maps $f: V \to W$ such that $df = fd$.
\item An edge is the data of a triplet $(f_0,f_1,H)$ where $f_0,f_1$ are vertices and $H$ is a chain homotopy, i.e. a degree -1 map $H: V \to W$ satisfying $dH + Hd=f_1-f_0$.
\item Simplices of dimension $k$ are degree $-k$ maps exhibiting homotopies between homotopies. For example, a triangle is the data 
	\eqnn
	(f_0,f_1,f_2, H_{01},H_{02},H_{12},G)
	\eqnd
where the $H_{ij}$ are homotopies from $f_i$ to $f_j$, and $G$ is a degree -2 map $G: V \to W$ exhibiting a homotopy between $H_{02}$ and $H_{12}+H_{01}$.
\end{itemize}
This space is called the \emph{Dold-Kan space} 
\index{Dold-Kan}
of the usual hom cochain complex $Hom^{\bullet}(V,W)$. We won't talk much about it, though we will talk a little more in the next lecture about the general philosophy of what algebraically motivated $(\infty,1)$-categories look like. 

For more on Dold-Kan, the interested reader may consult III.2 of~\cite{goerss-jardine}. Original sources are~\cite{dold-1958} and~\cite{kan-1958-css}.
\end{example}

\begin{remark}
For now, I'd like to say that this combinatorially defined space has well-understood homotopy groups:
	\eqnn
	\pi_i(\hom_{\chain_\kk}(V,W))) \cong H^{-i}(Hom^\bullet(V,W)),
	\qquad
	i \geq 0.
	\eqnd
That is, there is a natural isomorphism between the non-positive cohomology groups of the cochain complex $Hom^\bullet$ and the homotopy groups of the space $\hom$.
\end{remark} 

The upshot being---in a way I'll elaborate on next lecture---one should think of a functor of $(\infty,1)$-categories
	\eqnn
	\disk_{1,\orr} \to \chain_{\kk}
	\eqnd
as assigning a chain complex to each object of $\disk_{1,\orr}$, a chain map to every embedding, and a {\em chain homotopy} to every isotopy of embeddings, higher homotopies of chain complexes to higher homotopies of isotopies, and so forth.

\begin{remark}
You may note the {\em topological} definition of hom-spaces in $\disk_{1,\orr}$, and the {\em combinatorial} definition of hom-spaces in $\chain_{\kk}$. For more on these flavors, see Remark~\ref{remark.oo-cat algebraic and combinatorial}
\end{remark}

One can also speak of symmetric monoidal $(\infty,1)$-categories; we don't define these here, but you should stick with the intuition that these are symmetric monoidal categories whose symmetric monoidal structures are continuous with respect to the morphisms. 

\begin{example}
$\disk_{1,\orr}^{\coprod}$ is a symmetric monoidal $(\infty,1)$-category with symmetric monoidal structure given by disjoint union.
\end{example}

\section{Algebra of disks, revisited} 

\begin{definition}
\index{$\disk_{1,\orr}$-algebra}
Fix a symmetric monoidal $(\infty,1)$-category $\cC^\tensor$.
A {\em $\disk_{1,\orr}$-algebra} in $\cC^\tensor$ is a symmetric monoidal functor
	\eqnn
	F: \disk_{1,\orr}^{\coprod} \to \cC^\tensor.
	\eqnd
\end{definition}

\begin{example}
Take $\cC^\tensor = \Vect_{\kk}^{\otimes}$ as the target category. Because this is a discrete category, the conclusion of Theorem~\ref{theorem.associativity} holds verbatim: The data of a symmetric monoidal functor $F$ is equivalent to the data of a unital associative $\kk$-algebra. Indeed, the condition that ``isotopic embeddings must be sent to the same linear maps'' from before is re-expressed as our functor $F$ being a functor of $(\infty,1)$-categories. 
\end{example}

Let us next consider the $(\infty,1)$-category $\chain_{\kk}$ with symmetric monoidal structure given by the usual tensor product $\tensor_{\kk}$. \footnote{It is important here that we are working over a field $\kk$; otherwise we would take the derived tensor product, but this would lead us afield, pun intended.} As before, let us set $A = F(\RR)$.

Then for {\em every} embedding $j: \RR^{\coprod 2} \rightarrow \RR$ we have a multiplication $F(j) =: m_j: A \tensor A \to A$. The choice of $j$ is by no means canonical, and specifying all these $m_j$ is indeed an enormous amount of data. But the space of order-preserving\footnote{In the sense we used in the proof of Theorem~\ref{theorem.associativity}} $j$ is contractible. So $m_j$ is determined up to contractible choice. We will come back to this in the next lecture; for now, you should imagine that this data is manageable---in fact, applying a contraction to a point, you should imagine that specifying all the $\{m_j\}_j$ continuously is tantamount to just producing a single $m$.

But what to make of associativity of $m_j$? Tracing through the same proof as before, we find that the square \eqref{eqn:associativity}  is commutative only up to chain homotopy. That is, the data of the functor $F$ is not supplying a multiplication which is associative on the nose, but only associative up to homotopy, and these homotopies are specified (by the images of isotopies of disks). 

You can now imagine that the data of the spaces of embeddings
	\eqnn
	\RR \coprod \ldots \coprod \RR \to \RR
	\eqnd
supplies even further complicated data. But in fact, the space of these embeddings is also contractible up to swapping components of the domain. All told, it turns out that the immense amount of data specifying an algebra ``associative up to specified homotopies'' has a name already:

\begin{theorem}
\index{$A_\infty$-algebra}
The data of a functor between $(\infty,1)$-categories $F: \disk_{1,\orr}^{\coprod} \rightarrow \Chain_{\kk}^{\otimes_{\kk}}$ is equivalent to the data of a unital $A_{\infty}$-algebra.

More precisely, there is an equivalence of $(\infty,1)$-categories  between the $(\infty,1)$-category of symmetric monoidal functors $F$, and the $(\infty,1)$-category of unital $A_\infty$-algebras.
\end{theorem}

\begin{remark}
\index{$A_\infty$-algebra}
The word ``$A_\infty$'' is not important at all for what we will speak of next. I just wanted to emphasize that a $\disk_{1,\orr}$-algebra is simply an algebra whose multiplication is associative up to certain specified homotopies and higher homotopies; and if this is too much, for now, you will lose little intuition imagining these algebras to be actually associative. 
\end{remark}

\begin{example}
\index{$A_\infty$-algebra}
It is a classical fact that any unital $A_\infty$-algebra in $\Chain_{\kk}^{\tensor_\kk}$ is equivalent to an associative, unital dg-algebra.\footnote{This follows, for example, by embedding an $A_\infty$-algebra into its category of modules through the Yoneda embedding. See for example Section~(2g) of~\cite{seidel-book}.} In particular, unital dg-algebras are examples of $\disk_{1,\orr}$-algebras. Examples of such algebras include the de Rham cochains of a manifold (which in fact form a commutative dg algebra) and the endomorphism hom-complex of a cochain complex. 
\end{example}

\begin{example}[Maps of $A_\infty$-algebras.]\label{example:Aoo-maps}
\index{$A_\infty$-algebra}
However, you will {\em gain} something if you become accustomed to the fact that maps between these algebras need not respect multiplication and associativity on the nose. (This is also visible if one contemplates what a natural transformation between two symmetric monoidal functors $\disk_{1,\orr}^{\coprod} \to \chain_{\kk}^{\tensor_\kk}$ looks like.) For example, an $A_\infty$-algebra map between two dg-algebras is {\em not} the same thing as a map of dg algebras.

Just to give you a feel for what an $A_\infty$-algebra map $f:A \to B$ between two dg-algebras might look like, let us say that an $A_\infty$-algebra map is not simply a map 
	\eqnn
	f_1: A \to B
	\eqnd
satisfying the equation
	\eqnn
	f_1(da) = d f_1(a)
	\qquad
	\text{for all $a \in A$}.
	\eqnd
An $A_\infty$-algebra map also contains the data of a map
	\eqnn
	f_2: A \tensor A \to B[-1]
	\eqnd
satisfying
	\eqn\label{eqn:Aoo-map-quadratic}
	f_1(m(a_2,a_1))
	-
	m(f_1 a_2, f_1 a_1)
	=
	\pm
	f_2(da_2,a_1)
	+ \pm f_2 (a_2,da_1)
	+ \pm d f_2(a_2,a_1).	
	\eqnd
Note that~\eqref{eqn:Aoo-map-quadratic} states that $f_1$ does not necessarily respect multiplication on the nose, but the failure to do so is controlled by an exact element in the hom cochain complex $Hom^\bullet(A \tensor A, B)$ (as expressed on the righthand side). By definition, an $A_\infty$-algebra map also comes equipped with maps $f_k: A^{\tensor k} \to B[k-1]$ satisfying higher analogues of~\eqref{eqn:Aoo-map-quadratic} that cohere associativity properties up to homotopy and higher homotopies.

For more on $A_\infty$-algebras, we refer the interested reader to sources such as~\cite{keller-A-infinity-algebras-modules}.
\end{example}

\begin{remark}\label{remark:oo-category-functors-need-not-respect-composition}
The above is also a useful principle to keep in mind for $(\infty,1)$-categories. While a single $(\infty,1)$-category may be thought of as a category enriched in topological spaces, it is of course the maps and equivalences between them that make $(\infty,1)$-categories interesting. In particular, while any functor between $(\infty,1)$-categories can be modeled appropriately by an actual functor between topologically enriched categories, it is highly helpful to think instead of functors as lacking a ``strict'' respect for compositions, but equipped with higher coherences that ``make up'' for the lack of strict respect. 
\end{remark}

\section{Factorization homology of the circle} 
Now we are ready to define an invariant of $S^1$. 

\begin{definition}
\index{$\mfld_{1,\orr}^{\coprod}$}
Denote by $\mfld_{1,\orr}^{\coprod}$ the $(\infty,1)$-category of oriented, one-dimensional manifolds with finitely many connected components. The morphism space
	\eqnn
	\hom_{\mfld_{1,\orr}^{\coprod}}(X,Y)
	\eqnd
is the space of smooth, orientation-preserving embeddings from $X$ to $Y$.
\end{definition}

\begin{remark}
There is an obvious symmetric monoidal inclusion functor $\disk_{1,\orr}^{\coprod} \rightarrow \mfld_{1,\orr}^{\coprod}$. The superscript $\coprod$ in both notations is meant to remind the reader of that both $(\infty,1)$-categories are equipped with a symmetric monoidal structure.
\end{remark}

\begin{definition}
\index{factorization homology}
\index{left Kan extension}
Let $\cC^{\otimes}$ be given by $\Vect_{\kk}^{\tensor_{\kk}}$ or by $\Chain_{\kk}^{\otimes_{\kk}}$.
Let $A$ be a $\disk_{1,\orr}$-algebra in  $\cC^{\otimes}$. The \emph{factorization homology} of $A$ is the left Kan extension    
    \begin{center}
        \begin{tikzcd}
        \disk_{1,\orr}^{\coprod} \arrow[r,"A"] \arrow[d,hook] & \cC^{\otimes}\\
        \mfld_{1,\orr}^{\coprod}\arrow[ur,dashed,"\int A"'] & 
        \end{tikzcd}
    \end{center} 
and we denote this functor by $\int A$ as in the diagram. Given a smooth, oriented 1-dimensional manifold $X \in \mfld_{1,\orr}$, we denote the value of the left Kan extension by
	\eqnn
	\int_X A
	\eqnd
and we call this the factorization homology of $X$ with coefficients in $A$.
\end{definition}

Let me not tell you what left Kan extension exactly is, for the time being. (See Section~\ref{section: left kan extension}.) But let me tell you one theorem we can prove about this left Kan extension:

\begin{theorem}[$\tensor$-excision for $S^1$]\label{theorem: circle excision}
Fix a $\disk_{1,\orr}$-algebra $A$ in $\cC^\tensor$ and fix an orientation on $S^1$. Then factorization homology of the circle admits an equivalence
	\eqn\label{eqn:circle-excision}
	\int_{S^1} A
	\simeq
	A \bigotimes_{A \tensor A^{\op}} A.
	\eqnd
\end{theorem}

\begin{remark}
Let $A$ be an associative $\kk$-algebra. 
As we saw in Remark~\ref{remark. A is bimodule over itself}, $A$ is a bimodule over itself, and in particular, an $A \tensor A^{\op}$-module (on the right, or on the left). This explains the righthand side of~\eqref{eqn:circle-excision} when $\cC^{\tensor} = \Vect_{\kk}^{\tensor_\kk}$.
\end{remark}

\begin{example}
If $A$ is a unital associative algebra in $\cC^\tensor = \Vect_{\kk}^{\otimes_{\kk}}$, we already guessed that $\int_{S^1} A \simeq A/[A,A]$ in Section~\ref{section.circle-guess}.
\end{example}

\begin{example}\label{example:hochschild-complex}
Let $\cC^\tensor = \Chain_{\kk}^{\otimes_{\kk}}$. 
To make sense of~\eqref{eqn:circle-excision} when $\cC^{\tensor} = \chain_{\kk}^{\tensor_{\kk}}$, let us simply state that the notion of (bi)modules makes sense for $A_\infty$-algebras as well, and the notion of tensoring modules over an algebra also makes sense, for instance by articulating a model for the bar construction. 

However, let us caution the uninitiated that the bar construction models the derived tensor product:
	\eqnn
	\int_{S^1}A \simeq A \bigotimes^{\mathbb{L}}_{A\otimes_{\kk} A^{op}} A.
	\eqnd
This tensor product already has a name: It is the {\em Hochschild chain complex} of $A$. (See~\cite{hochschild-1945} and~\cite{cartan-eilenberg-56}.)
\index{Hochschild chain complex}
\end{example}

\begin{remark}
Let $A$ be an ordinary unital associative $\kk$-algebra, concentrated in degree 0. Then one may consider $A$ to be a $\disk_{1,\orr}$-algebra in $\cC^{\tensor} = \Vect_{\kk}^{\tensor_{\kk}}$ and in $\cC^{\tensor} = \chain_\kk^{\tensor_{\kk}}$. Then the bar construction $A \tensor_{A \tensor A^{\op}} A$ constructed in $\cC = \Vect_{\kk}$ yields a vector space given by the 0th cohomology of the Hochschild complex, otherwise known as $A/[A,A]$. On the other hand, the bar construction constructed in $\cC = \chain_\kk$ encodes more homotopically rich information, giving rise to the entire Hochschild chain complex of $A$.
\end{remark}

We mentioned at some point that the circle invariant naturally possesses an $S^1$-action (more correctly, an action of the orientation-preserving diffeomorphism group $Diff^+(S^1)$; and this group is homotopy equivalent to $S^1$). Assume that $A$ is a smooth algebra over a perfect field $\kk$. Then there is a \emph{Hochschild-Kostant-Rosenberg isomorphism}~\cite{HKR}
\index{Hochschild-Kostant-Rosenberg}
	\eqnn
	H^{-i}(\mathrm{Hochschild\ complex}) \cong \Omega^i(A),
	\qquad
	i \geq 0
	\eqnd
where $\Omega^i(A)$ is the space of algebraic de Rham $i$-forms. The latter can be equipped with the de Rham differential, and this is precisely the circle action in this case. We will elaborate on this in the start of next lecture.

\section{Exercises}

\begin{exercise}
Improve upon Theorem~\ref{theorem.associativity} by exhibiting an equivalence of categories
	\eqnn
	\Fun^{\tensor}(\diskdisc_{1,\orr}^{\coprod} , \Vect_{\kk}^{\tensor_\kk}) \simeq
	\Ass\Alg(\Vect_{\kk}^{\tensor_{\kk}}).
	\eqnd
between the category of symmetric monoidal functors $\diskdisc_{1,\orr}^{\coprod} \to \Vect_{\kk}^{\tensor_\kk}$ and the category of unital associative $\kk$-algebras.
\end{exercise}

\begin{exercise}\label{exercise:cocenter}
Verify Proposition~\ref{prop:cocenter}.
\end{exercise}

\begin{exercise}
Let $\cC^\otimes = \Cat^\times$ be the $(\infty,1)$-category of categories. Its objects are categories, and given two objects $D, E$ we define the hom-space combinatorially as follows. We let $\hom_\Cat(D,E)$ have vertices given by functors $D\rta E$, edges natural isomorphisms, triangles commutative triangles of natural isomorphisms, and $k$-simplices commutative diagrams of natural isomorphisms in the shape of $k$-simplices. (Though we call $\Cat^\times$ an $(\infty,1)$-category for the purposes of this problem, it is most naturally a 2-category with invertible 2-morphisms.)

Show that a $\disk_{1,\orr}$-algebra in $\Cat^\times$ is equivalent to a (unital) monoidal category.
\end{exercise}

\begin{exercise}\label{exercise: no orientations}
Let $\disk_1$ be the category of 1-disks without any orientation condition on morphisms. Then the space of embeddings
	\eqnn
		\Emb(\bb{R},\bb{R})\simeq O(1)\simeq S^0
	\eqnd  
is no longer contractible. Given a symmetric monoidal functor
	\eqnn
		F:\disk_1^{\coprod} \rta\Vect_{\kk}^{\tensor_{\kk}}
	\eqnd 
let $A=F(\bb{R})$. We get a map $\mrm{Id}:A\rta A$ and  a map $\tau:A\rta A$ from orientation reversal. 

How does $\tau$ interact with the multiplication map?

Note that now there are two distinct isotopy classes of maps $S^1 \to S^1$; one of them contains an orientation-reversing diffeomorphism. Can you describe the induced map on cocenters?
\end{exercise}

\chapter{Interlude on $(\infty,1)$-Categories}
\label{chapter. infty cats}
\index{$(\infty,1)$-category}

If the reader does not care for this chapter, they may soon find out, and they may just as soon skip to the next section. Here we introduce various models of $(\infty,1)$-categories. When discussing the model of quasi-categories (otherwise known as weak Kan complexes, or $(\infty,1)$-categories), we emphasize the importance of either {\em not having} a notion of composition inside a category, or of a functor {\em not ``respecting'' on the nose} some notion of composition between a domain and codomain $(\infty,1)$-category.

We recommend~\cite{goerss-jardine, bergner-survey, lurie-htt, joyal} for more details on the contents of this chapter.

\begin{remark}\label{remark. infinity 1 terminology}
Let us remark on where the term $(\infty,1)$ arises. In general, higher category theory may prompt us to study $(m,n)$-categories---these are categories who have higher morphisms up to order $m$, and for whom all morphisms of order $>n$ are invertible. When $m= \infty$ and $n=1$, what we obtain is a category where one may ask for maps between (1-)morphisms, maps between these, and so forth, but wherein all of these maps are actually invertible. It is customary, and useful, to think of these maps and higher maps as homotopies and higher homotopies.

When one is given a category enriched in topological spaces, one can think of the points of morphisms spaces to be the 1-morphisms, paths between these to encode the (invertible) 2-morphisms (invertible because any path admits an inverse up to homotopy), and so forth.
\end{remark}

\section{Homotopy equivalences, and equivalences of $(\infty,1)$-categories}
Recall that last time, we told you to think of an $(\infty,1)$-category $\cC$ as  a category such that for every pair  of objects $x,y\in\cC$, the collection of maps $\hom_{\cC}(x,y)$ is a topological space and composition is continuous. Such data is usually called a \emph{topologically enriched category}, and this is one way you can think about what an $(\infty,1)$-category is.
\index{topologically enriched category}

In this model, a \emph{functor} $F:\cC\rta\cal{D}$ of $(\infty,1)$-categories is a functor in the usual sense with the property that the induced maps
	\eqnn
		\hom_\cC(x,y)\rta\hom_{\cal{D}}(Fx,Fy)
	\eqnd 
are continuous.

In an $(\infty,1)$-category $\cC$, one can talk about equivalences of objects.

\begin{definition}\label{defn. equivalence in oo-cat}
An \emph{equivalence} $x\rta y$ between objects $x,y$ in an $(\infty,1)$-category $\cC$ is the data of a map $f:x\rta y$ such that there exists a map $g:y\rta x$ and homotopies $f\circ g\simeq \mrm{Id}$ and $g\circ f\simeq \mrm{Id}$.
\index{equivalence!in an $(\infty,1)$-category}
\end{definition}

\begin{definition}
A functor $F:\cC\rta\cal{D}$ of $(\infty,1)$-categories is called \emph{essentially surjective} if, up to equivalence of objects (Definition~\ref{defn. equivalence in oo-cat}), every object of $\cal{D}$ is in the image of $F$. 
\end{definition}

For reasons that are not always obvious at your first rodeo, it turns out that the question ``what are $(\infty,1)$-categories?'' is just as important as the question ``when are two $(\infty,1)$-categories equivalent?''

\begin{definition}\label{remark:equivalence-of-infty-cats}

A functor $F:\cC\rta\cal{D}$ of $(\infty,1)$-categories is an \emph{equivalence} if $F$ is essentially surjective, and the induced maps
	\eqnn
		\hom_\cC(x,y)\rta\hom_{\cal{D}}(Fx,Fy)
	\eqnd 
are weak homotopy equivalences.
\index{equivalence!of $(\infty,1)$-categories}
\end{definition}

\begin{remark}
Recall that a map is a weak homotopy equivalence if it induces a bijection of connected components and isomorphisms on all homotopy groups. That is, a continuous map of topological spaces $g: X \to Y$ induces functions
	\eqnn
	\pi_0(X) \to \pi_0(Y),
	\qquad
	\pi_1(X) \to \pi_1(Y),
	\qquad
	\pi_2(X) \to \pi_2(Y), \qquad
	\ldots
	\eqnd
and we say $g$ is a weak homotopy equivalence if these maps are bijections $\pi_i(X) \to \pi_i(Y)$ for all $i \geq 0$. Of course, for $i \geq 1$, we demand that these are bijections for any choice of connected component of $X$.

As an example, if $Y$ is a contractible space, then any choice of map $\ast \to Y$ is a weak homotopy equivalence, and the unique map $Y \to \ast$ is also a weak homotopy equivalence.
\end{remark}

\begin{remark}
One should think of the notion of a weak homotopy equivalence of spaces as analogous to the definition of  quasi-isomorphisms for chain complexes. It is not always true that weak homotopy equivalences can be inverted, even up to homotopy (nor can quasi-isomorphisms). Regardless, weak homotopy equivalences (and quasi-isomorphisms) induce isomorphisms on the most tractable algebraic invariants we have: homotopy groups of spaces (and cohomology groups of chain complexes). 

And just as for chain complexes, one must have some technology to really consider equivalent objects to behave as though they are equivalent---for example, if $f: \cC \to \cD$ is an equivalence of $(\infty,1)$-categories, we might demand a functor $g: \cD \to \cC$ exhibiting some notion of invertibility of $f$. In homological algebra, this was classically dealt with via derived categories. For $(\infty,1)$-categories (and in more recent approaches to homological algebra), this can be dealt with through the language of model categories and various localization techniques. To this end, Bergner~\cite{bergner, bergner-survey}, Rezk~\cite{rezk-model-for-homotopy-theory}, Lurie~\cite{lurie-htt} and Joyal~\cite{joyal} construct model categories of $(\infty,1)$-categories.
\end{remark}

\section{Contractibility}

\begin{remark}\label{remark.contractibility is uniqueness}
We make a remark on contractibility. When you hear ``There is a contractible space of BLAH,'' you should think ``There is a unique BLAH.'' This is because given $BLAH$, there is a canonical map $\{ BLAH\} \rta *$ to the point, and contractibility means that this map is a weak homotopy equivalence.
\end{remark}

\begin{example}
The space of oriented embeddings from $\RR$ to itself is contractible, and contains the identity morphism.

For this reason, when we are given a functor of $(\infty,1)$-categories $\disk_{1,\orr} \to \cC$, the induced continuous map
	\eqnn
	\hom_{\disk_{1,\orr}}(\RR,\RR)
	\to
	\hom_\cC(F(\RR),F(\RR))
	\eqnd
can fruitfully be thought of as sending the identity of $\RR$ to the identity of $F(\RR)$, and as little more information. (However, as $F$ is a functor of $(\infty,1)$-categories, $F$ also exhibits a compatibility between this homotopy equivalence and composition.)
\end{example}

\section{Combinatorial models and $(\infty,1)$-categories: simplicial sets}\label{section. simplicial categories}

The comments so far are meant to make $(\infty,1)$-categories seem less foreign; if you know what categories, spaces, and weak homotopy equivalences are, you can more or less follow a conversation or discussion.

But that's just our first pass. Let's now take a second look.

\begin{remark}[$(\infty,1)$-categorical practice in algebraic settings]\label{remark.oo-cat algebraic and combinatorial}
It was ``obvious'' that the discrete category $\diskdisc_{1,\orr}$ admitted a topology on its morphism set, so we were naturally led to consider the $(\infty,1)$-category $\disk_{1,\orr}$. But how do we construct $(\infty,1)$-categories in algebra? For example, it seems like a non-trivial task to put a topology on a set of chain maps.

In algebraic examples, the morphism space $\hom_\cC(x,y)$ is almost always combinatorially defined. We saw this yesterday in the Dold-Kan space: We already had a set called the chain maps, but rather than try to topologize the collection of chain maps, we added on edges (for every homotopy) and higher simplices (for homotopies between homotopies).

And in such cases, we do not construct functors by constructing truly flimsy continuous maps of topological spaces; instead, we usually maps combinatorially, by first mapping vertices to vertices, then edges to edges, and so on.   Of course, when spaces are combinatorially defined (e.g., built out of vertices, edges, and so forth) any combinatorially well-behaved function automatically induces a continuous map.
\end{remark}

\begin{example}\label{example:cdga-oo-category}
At this summer school, 
Pavel Safronov is giving lectures on Poisson structures in derived algebraic geometry. When he speaks of the $(\infty,1)$-category of cdgas, one can likewise construct a combinatorial space of maps. Given $A, B$ two cdgas, $\hom(A,B)$ has vertices given by honest cdga maps, edges given by homotopies between these (which is not just the data of a homotopy of the underlying chain maps), and so forth.
\end{example}

\begin{remark}\label{remark:singular-complex}
One of the apparently hard things about constructing $\disk_{1,\orr}$-algebras is that $\disk_{1,\orr}$ and $\Chain_\kk$ are $(\infty,1)$-categories in different ways. The former is geometric and the latter is algebraic---for example, the domain is topologized using continuous techniques while the target has combinatorially defined spaces.

Normally, we overcome this by making the more topological thing behave more combinatorially. For example, out of any space $X$, one can construct a combinatorial gadget whose vertices are points of $X$, whose edges are paths in $X$, and whose higher simplices are continuous maps from higher simplices into $X$. (The combinatorial output is called the {\em singular complex} of $X$; see for example~\cite{goerss-jardine} and Example~\ref{example. sing X}.) As in the previous remark, constructing a functor of $(\infty,1)$-categories then usually boils down to combinatorially assigning simplices to simplices.
\end{remark}

Let us make these appeals to combinatorial arguments precise. References for simplicial sets include~\cite{goerss-jardine}, while resources for $(\infty,1)$-categories include~\cite{lurie-htt} and the Appendix of~\cite{nadler-tanaka}.

\begin{definition}
Let $\Delta$ denote the category of finite, linearly ordered, non-empty sets. A morphism is a map of posets. $\Delta$ is often called the {\emph simplex category}, and an object of $\Delta$ is sometimes referred to as a \emph{combinatorial simplex}.
\index{simplices!combinatorial}
\index{simplices!category of}
\index{$\Delta$}
\index{$[n]$}
\end{definition}

\begin{remark}
Any object of $\Delta$ is isomorphic to the standard linearly ordered set $[n] = \{0< \ldots < n\}$ for some $n \geq 0$. 
\end{remark}

\begin{definition}\label{defn. simplicial set}
A \emph{simplicial set} is a functor $\Delta^{\op} \to \Set$ to the category of sets. A map of simplicial sets is a natural transformation.
\index{simplicial set}
\end{definition}

\begin{example}\label{example. sing X}
\index{$\sing(X)$}
\index{singular complex of $X$}
Let $X$ be a topological space, and $\Delta^n \subset \RR^{n+1}$ the standard $n$-simplex. That is, $\Delta^n$ is the space of points $(x_0,x_1,\ldots,x_n)$ for which $\sum x_i = 1$ and $0 \leq x_i \leq 1$.

Then $\sing(X)$, the singular complex of $X$,  is the simplicial set sending $[n]$ to the set of continuous maps $\gamma: \Delta^n \to X$. Any map of posets $f: [n] \to [m]$ induces a linear map $\Delta^n \to \Delta^m$ (determined by sending the $i$th vertex of the domain to the $f(i)$th vertex of the codomain, and extending linearly), and $\gamma \mapsto f^*\gamma$ defines the effect of $\sing(X)$ on morphisms. 
\end{example}

\begin{example}[A combinatorial model for spaces of smooth embeddings]
Fix two smooth manifolds $X$ and $Y$.
The reader will soon realize that a {\em continuous} path in the topological space of smooth embeddings $\Emb(X,Y)$ need not represent a {\em smooth} isotopy. So the previous remark's example of the singular complex construction doesn't fit the idea that an edge (in a combinatorial model for the space of smooth embeddings) should represent a smooth isotopy.

So here is yet another technique that can be used to model $\disk_{1,\orr}$ as an $(\infty,1)$-category (see for example ~\cite{aft-1,aft-2}, though this idea goes back much further). One declares a vertex to be a smooth embedding $j: X \to Y$. One declares an edge to be the data of a smooth embedding $X \times \Delta^1 \to Y\times \Delta^1$ which respects the projection to $\Delta^1$. More generally, a $k$-simplex is a smooth embedding $X \times \Delta^k \to Y \times \Delta^k$ respecting the projections to $\Delta^k$. 
\end{example}

\begin{example}
For every $n$, we let $\Delta^n = \hom(-,[n])$ denote the simplicial set represented by $[n]$. This notation conflicts with our use of $\Delta^n \subset \RR^{n+1}$ as the standard $n$-simplex, but we shall proceed as the distinction should be clear from context.
\end{example}

\begin{notation}
Let $A$ be a simplicial set. we denote the set $A([n])$ by $A_n$.

The category of simplicial sets has a symmetric monoidal structures given by direct product, where $(A\times B)_n = A_n \times B_n$. The unit is the constant simplicial set sending every $[n]$ to the terminal (i.e., one-point) set.

Also, given two simplicial sets $A$ and $B$, we can define a new simplicial set $\hom(A,B)$ by declaring $\hom(A,B)_k$ to be the collection of simplicial set maps $A \times \Delta^k \to B$. One can arrange for the category of simplicial sets to be enriched over itself.
\end{notation}

\begin{remark}[Simplicial sets model spaces]\label{remark. simplicial sets are spaces}
Given a simplicial set $A$, one may define a topological space $|A|$ built by quotienting $\coprod_{k \geq 0} A_k \times \Delta^k$ along relations given by the morphisms $[k] \to [k']$. This is called the {\em geometric realization} of $A$. In the opposite direction, we saw above the passage from a space $X$ to $\sing(X)$ in Example~\ref{example. sing X}. It is a classical fact that $\sing$ and $|-|$ define an equivalence between a homotopy theory of simplicial sets, and a homotopy theory for topological spaces. Informally, one may think of simplicial sets themselves as giving combinatorial models for topological spaces, at least up to the notion of weak homotopy equivalence. We refer the interested reader to~\cite{goerss-jardine}. 
\end{remark}

Then, just as categories enriched in topological spaces are a model for $(\infty,1)$-categories, categories enriched in simplicial sets also form a model for $(\infty,1)$-categories.

\section{Combinatorial models and $(\infty,1)$-categories: weak Kan complexes}\label{section. quasi categories}

In this chapter we will talk about a model of $(\infty,1)$-categories that opens a completely different intuition: the model of $\infty$-categories (otherwise known as weak Kan complexes or quasi-categories).

\begin{definition}
Let $n \geq 0$, and choose $0 \leq k \leq n$. Then the {\em $k$th $n$-horn} $\Lambda^n_k$ is defined to be the simplicial set sending $[i]$ to the set of all morphisms $[i] \to [n]$ that do {\em not} surject onto the set $\{0,1,\ldots,k-1,k+1,\ldots,n\}$.
\index{horns}
\index{$\Lambda^n_k$}
\end{definition}

Note the natural transformation $\Lambda^n_k \to \Delta^n$. One may informally think of $\Lambda^n_k$ as obtained from $\Delta^n$ by deleting the interior, and the $k$th face (i.e., the face opposite the $k$th vertex), of $\Delta^n$.

\begin{definition}
\index{weak Kan complex}
\index{quasi-category}
\index{$\infty$-category}
Let $\cC$ be a simplicial set. We say that $\cC$ is a {\em weak Kan complex}, a {\em quasi-category}, or an {\em $\infty$-category} if the following holds: For every $n \geq 2$ and for every $0 < k < n$, every map $\Lambda^n_k \to \cC$ has a filler as follows:
	\eqn\label{eqn. weak kan condition} 
	\xymatrix{
	\Lambda^n_k \ar[r] \ar[d] & \cC \\
	\Delta^n \ar@{-->}[ur]^{\exists}
	}	
	\eqnd
That is, one can extent the map $\Lambda^n_k \to \cC$ to a map from the $n$-simplex $\Delta^n$. A {\em functor} between two $\infty$-categories is a map of simplicial sets.
\end{definition}

It is not a priori obvious that such a simplicial set captures any notion of a category. You will see in Exercise~\ref{exercise. unique filling is category} that, if the horn-filling above is unique, then $\cC$ captures exactly the data of a small category.

\begin{remark}[$(\infty,1)$-categories as weak Kan complexes]\label{remark:weak-kan-model}
It is healthy to think of any category (in the classical sense) as a bunch of vertices (for objects) and edges (for morphisms) and triangles (for commutative triangles). $\infty$-categories generalize this intuition to $(\infty,1)$-categories. This model for $(\infty,1)$-categories goes back to Boardman-Vogt~\cite{boardman-vogt-weak-kan} and is the model developed by Joyal~\cite{joyal} and Lurie~\cite{lurie-htt}.
\end{remark}

\begin{remark}
All models of $(\infty,1)$-categories are equivalent in a precise sense. (See for example the work of Julie Bergner \cite{bergner, bergner-survey}.) As I have said before: If you do not work with this stuff, then for this lecture series, it is probably healthiest (and least technical) to think of an $(\infty,1)$-category as a topologically enriched category.
\end{remark}

Given an $\infty$-category $\cC$, we refer to the set $\cC_0$ of vertices be the set of {\em objects} of $\cC$. We refer to the set $\cC_1$ of edges as the set of {\em morphisms} of $\cC$. Note that every edge begins and ends at some vertex; these give the source and target of the morphism. More precisely, note that there are exactly two injections $d_0, d_1: [0] \to [1]$, where the map $d_i$ does not contain $i \in [1]$ in its image. Then, for any $f \in \cC_1$, we let $d_0^* f$ be the target of $f$, and $d_1^* f$ the source of $f$.  One should think of a $k$-simplex (i.e., an element of $\cC_k$) as a homotopy-commuting diagram in the shape of a $k$-simplex.

\begin{example}\label{example:composition-in-weak-kan}
Note that an $\infty$-category $\cC$ does {\em not} specify a composition law for an $(\infty,1)$-category. Regardless, if one fixes three objects $x_0,x_1,x_2 \in \cC_0$ and two composable morphisms $f_{01},f_{12} \in \cC_1$, one may study the space of all 2-simplices in $\cC_2$ of the form
	\eqnn
	\xymatrix{
	x_0 \ar[r]^{f_{01}} \ar[dr]_g & x_1 \ar[d]^{f_{12}} \\
	& x_2.
	}
	\eqnd
Informally, the space of such triangles may be thought of as the space of all $g$ equipped with a homotopy between $g$ and a putative composition $f_{12} \circ f_{01}$. For a weak Kan complex, this space is always contractible. (See 1.6 of~\cite{joyal}, or Corollary~2.3.2.2 of~\cite{lurie-htt}.) That is, by Remark~\ref{remark.contractibility is uniqueness}, there is a homotopically {\em unique} way to compose $f_{12}$ with $f_{01}$.

This illustrates one of the most useful operating principles of $\infty$-categories: One can often avoid defining a specific operation (such as multiplication in an algebra, or composition in a category). Instead, it is often easier to construct a gigantic gadget containing a large space of possible choices for an operation, and to prove the contractibility of such choice-spaces. The insight here is that it is often difficult to finagle coherences of operations defined in particular ways; it is easier to describe properties about the spaces of possible operations.
\end{example}

\begin{remark}\label{remark. spaces as groupoids}
\index{$\infty$-groupoid}
Pavel has also spoken of $\infty$-groupoids; these are $(\infty,1)$-categories in which every morphism is an equivalence.

A general philosophy going back to Grothendieck is that any $\infty$-groupoid is equivalent to a space. That is, given a space, one can obtain an $(\infty,1)$-category whose objects are points of the space, and whose morphisms are paths in the space; and any $\infty$-groupoid is equivalent as an $(\infty,1)$-category to such a thing. This is called the Homotopy Hypothesis; it is provable in any model of $(\infty,1)$-categories. 

Thus you will often hear of the space of objects of an $(\infty,1)$-category (obtained by throwing out all non-equivalences), or of people treating a space as an $(\infty,1)$-category. 

Concretely for us, assume that $\cC$ is a simplicial set for which the horn-filler~\eqref{eqn. weak kan condition} exists for {\em all} $0 \leq k \leq n$. Such a simplicial set is called a {\em Kan complex}, and is clearly an example of an $\infty$-category. Also easily verified is that for any topological space $X$, $\sing(X)$ is a Kan complex. Indeed, Kan complexes are precisely models for $\infty$-groupoids. See Exercise~\ref{exercise. groupoids}.
\end{remark}

\begin{warning-numbered}\label{warning.infinity-category-warnings}
Now that we have come this far, let us point out some intuitions that can be {\em misleading} if one only thinks of $(\infty,1)$-categories as topologically enriched categories. (This is why it's worth taking a second look.)
\enum
	\item A notion of composition need not always be defined for an $(\infty,1)$-category; instead one may provide a contractible space of ways in which composition can be interpreted. (Example~\ref{example:composition-in-weak-kan}.) This is not some impossible amount of data; that one can construct such a contractible space is often a consequence of the model one is using, such as the weak Kan complex model, and constructing an $(\infty,1)$-category is often either a combinatorial or formal task.
	\item Likewise, a functor need not ``respect'' composition in the classical sense---especially when composition may not even be strictly defined! (Remark~\ref{remark:oo-category-functors-need-not-respect-composition}.)
	\item Given an $(\infty,1)$-category $\cC$, there need not be an ``underlying category'' that we have topologized to obtain $\cC$. (In this sense, both $\Chain$ and $\mfld_{n}$ are somewhat misleading examples, as both had natural starting points that we sought to topologize.) One reason that there is no ``underlying category'' to be topologized is because ``underlying set'' is not a notion preserved under weak homotopy equivalences; hence ``underlying category'' is not a notion invariant under equivalences of $(\infty,1)$-categories. (Definition~\ref{remark:equivalence-of-infty-cats}.) An invariant construction is that of the homotopy category, obtained by taking $\pi_0$ of every morphism space.
\enumd
\end{warning-numbered}

\subsection{Nerves}\label{section. nerves}

We have claimed that all models of $(\infty,1)$-categories are equivalent, so let us at least show how to pass from a category enriched in simplicial sets to an $\infty$-category.

Let $A$ be a category in the classical sense. We assume $A$ is small, meaning that the collection of objects $\ob A$ is a set, and for any two objects $x, y \in \ob A$, the set of morphisms $\hom_A(x,y)$ is also a set. (We remind the reader that set-theoretic issues arise at the very beginnings of category theory, and the usual work-around consists of choosing a particular Grothendieck universe of sets one can speak of; a set is called small if it is inside this Grothendieck universe.) Then one can construct a simplicial set out of $A$ as follows. First, consider the poset $[n] = \{0< 1 < \ldots < n\}$ consisting of $n+1$ elements. Then $[n]$ can also be treated as a category, where 
	\eqnn
	\hom_{[n]}(i,j) = \begin{cases}
	\ast & i \leq j \\
	\emptyset & i > j
	\end{cases}.
	\eqnd
	
\begin{definition}\label{defn. nerve}
\index{nerve}
The {\em nerve} of a small category $A$ is the simplicial set $N(A)$ whose set of $n$-simplices is defined to be the set of functors
	\eqnn
	N(A)_n = \fFun([n], A).
	\eqnd
Given any morphism $j: [m] \to [n]$ in $\Delta$, the map $j^*: N(A)_n \to N(A)_m$ is given by precomposition by $j$. 
\end{definition}

Let $\cA$ be a category enriched in simplicial sets. We will assume that for every pair of objects $x, y \in \ob \cA$, that $\hom(x,y)$ is a Kan complex. Then one can construct a weak Kan complex, $N(\cA)$, called the {\em homotopy coherent nerve of $\cA$}. For details, we refer the reader to Definition~1.1.5.5 of~\cite{lurie-htt}. Here, we simply give examples of simplices of $N(\cA)$:
\index{nerve!homotopy coherent}
\begin{itemize}
	\item A 0-simplex of $N(\cA)$ is an object of $\cA$.
	\item An edge $f$ from $x_0$ to $x_1$ in $N(\cA)$ is a 0-simplex of $\hom_{\cA}(x_0,x_1)$.
	\item A triangle
		\eqnn
		\xymatrix{
		x_0 \ar[r]^{f_{01}} \ar[dr]_{f_{02}} & x_1 \ar[d]^{f_{12}} \\
		& x_2
		} \in N(\cA)_2
		\eqnd
	is the data of three 0-simplices $f_{ij} \in \hom_{\cA}(x_i,x_j)_0$, along with the data of a 1-simplex $H \in \hom_{\cA}(x_0,x_2)_1$ from the composition $f_{12} \circ f_{01}$ to $f_{02}$. 
\end{itemize}
For a drawing of what a 3-simplex of $N(\cA)$ represents, we refer the reader to the Appendix of~\cite{nadler-tanaka}. 

\begin{example}
Let $\kan$ be the full subcategory of the category of simplicial sets consisting only of the Kan complexes. Note that $\kan$ can be enriched over itself. (It is an exercise to see that if $A$ and $B$ are Kan complexes, then the simplicial set $\hom(A,B)$ is also a Kan complex.) 

Then the {\em $\infty$-category of spaces} is often modeled by the homotopy coherent nerve associated to (the self-enrichment of) $\kan$. Note that we are using Remark~\ref{remark. simplicial sets are spaces} to think of the theory of Kan complexes as equivalent to the theory of spaces.
\end{example}

\begin{notation}[The $\infty$-category $\TTop$ of spaces]
We will let $\TTop$ denote the $\infty$-category of spaces.
\index{$\TTop$}
\end{notation}

\subsection{Some useful constructions}
Let $A, B, C$ be $\infty$-categories, and fix functors $A \to C$ and $B \to C$. Then one can define the {\em fiber product} $\infty$-category $A \times_C B$ by declaring the set of $n$-simplices to be
	\eqnn
	(A \times_C B) = A_n \times_{C_n} B_n.
	\eqnd

\index{fiber product! of $\infty$-categories}

We end with the construction of slice categories for $\infty$-categories. This will come in handy if you want to understand Section~\ref{section.framings} later on. 

Given two simplicial sets $A$ and $B$, we let the {\em join} $A \star B$ be the simplicial set whose set of $n$-simplices is given by
	\eqnn
	(A \star B)_n = A_n \coprod B_n \coprod \left(\coprod_{i + j = n-1} A_i \times B_j\right).
	\eqnd
\index{$A \star B$}
\index{join}

\begin{definition}[Slice categories]\label{defn. slice categories}
Let $F: \cC \to \cD$ be a functor of $\infty$-categories. We define the {\em over-category} $\cD_{/F}$ to be the simplicial set whose set of $n$-simplices is given by the set of simplicial set maps $j: \Delta^n \star \cC \to \cD$ for which $j$ restricts to $F$ along $\cC \subset \Delta^n \star \cC$. Likewise, the {\em under-category} $\cD_{F/}$ has $n$-simplices given by maps $\cC \star \Delta^n \to \cD$ restricting to $F$ along $\cC$.
\index{slice categories}
\end{definition}

\section{Exercises}

\begin{exercise}\label{exercise. unique filling is category}
\index{nerve}
Suppose that $\cC$ is a weak Kan complex, and that for every $\Lambda^n_k \to \cC$, the filler $\Delta^n \to \cC$ in \eqref{eqn. weak kan condition} is {\em unique}. Prove that the data of $\cC$ is equivalent to the data of a small category whose set of objects consists of $\cC_0$. (More precisely, show that $\cC$ is isomorphic to the nerve of some small category $A$.)

Now suppose that $\cD$ is another such simplicial set. Using your previous solution, exhibit a natural bijection between the collection of simplicial set maps $\cC \to \cD$, and the collection of functors between the associated categories.

Now show that a map $\cC \times \Delta^1 \to \cD$ of simplicial sets is exactly the data of a natural transformation between two functors.

What data does a map $\cC \times \Delta^2 \to \cD$ of simplicial sets encode?
\end{exercise}

\begin{exercise}\label{exercise. groupoids}
Note that there is a unique map $s_0: [1] \to [0]$. Given an $\infty$-category $\cC$ and an object $X \in \cC_0$, we let the {\em identity morphism $\id_X$ of $X$} be the edge $s_0^*X \in \cC_1$.

We say that an edge $f$ from $X$ to $Y$ is an {\em equivalence} in $\cC_1$ if there exist an edge $g$ from $Y$ to $X$, and two triangles
	\eqnn
	\xymatrix{
	X \ar[r]^f \ar[dr]_{\id_X} & Y \ar[d]^g \\
	& X
	},
	\qquad
	\xymatrix{
	Y \ar[r]^g \ar[dr]_{\id_Y} & X \ar[d]^f \\
	& Y
	}
	\qquad
	\in \cC_2
	\eqnd
with the indicated boundary edges. (For example, for the lefthand triangle---call it $T$---we demand that $d_0^* T = g, d_1^* T = \id_X, d_2^* T = f$. Here, $d_i: [1] \to [2]$ is the unique injection that does not contain $i \in [2]$ in its image.
Show that if $\cC$ is an $\infty$-groupoid, then every edge is an equivalence.

Conversely, suppose that $\cC$ is an $\infty$-category for which every edge is an equivalence. Prove that $\cC$ is an $\infty$-groupoid (i.e., show that $\cC$ is a Kan complex).
\end{exercise}

\begin{exercise}
\index{nerve!homotopy coherent}
Let $\spaces$ be the category of topological spaces. It is enriched over itself. Render it a category enriched over simplicial sets by applying $\sing$ to each morphism space, and let $N(\spaces)$ denote the homotopy coherent nerve. What is a 2-simplex in this $\infty$-category?

Look up the definition of homotopy coherent nerve in~\cite{lurie-htt}, or follow the Appendix of~\cite{nadler-tanaka}, to write out what a 3-simplex of $N(\spaces)$ is.
\end{exercise}

\begin{exercise}
\index{nerve!dg-nerve}
\index{nerve!$A_\infty$-nerve}
Look up the notion of the dg-nerve of a dg-category (for example, Construction~1.3.1.6 of~\cite{higher-algebra}). Verify that the dg-nerve of a dg-category is a weak Kan complex, and that any functor of dg-categories induces a functor of $\infty$-categories. Produce an example of a functor between dg-nerves that does {\em not} arise as a functor between dg-categories.  

Repeat for the $A_\infty$-nerve of an $A_\infty$-category. (References include~\cite{faonte, tanaka-thesis, tanaka-pairing}.)

\end{exercise}

\begin{exercise}\label{exercise. slice}
Fix an object $B \in \TTop$. Note that this induced a functor $j: \ast \to \TTop$ from the terminal simplicial set, sending the unique object of $\ast$ to $B$. We will denote the associated slice category $\TTop_{/j}$ by $\TTop_{/B}$.

Write out what a 1-simplex and 2-simplex of $\TTop_{/B}$ are. To know what a 2-simplex is, you will need to know what 3-simplices of homotopy coherent nerves are.
\end{exercise}

\chapter{Factorization homology in higher dimensions}
\label{chapter. fact hom in higher dims}

In this chapter we introduce higher-dimensional versions of associative algebras. The simplest of these are the $\EE_n$-algebras. To many of you, these will contain new kinds of algebraic structure. Informally, $\EE_n$-algebras have more commutativity than associative algebras, but they do not quite have all the commutativity one could wish for. This is a feature, not a bug; the lack of higher commutativity is in some sense what makes these algebras appropriate for detecting building blocks of manifolds of fixed dimension.

We will also define factorization homology. This is a local-to-global invariant satisfying a generalization of the $\tensor$-excision we saw last time for the circle. 

The last section, Section~\ref{section. leftovers}, contains various commentary on the notions we did not touch on in-depth during the spoken lecture.

\clearpage
\section{Review of last talk (Chapter~\ref{chapter. dim 1})}

\begin{itemize}
    \item We defined $\disk_{1,\orr}^{\coprod}$ as a symmetric monoidal $(\infty,1)$-category under disjoint union. We defined also $\mfld_{1,\orr}^{\coprod}$.
    \item We considered symmetric monoidal functors
    \eqnn
    	F\colon\disk_{1,\orr}^{\coprod}
    	\rta \Vect_{\kk}^{\tensor_{\kk}}
	\eqnd
    and saw that the data of such a functor was the same as the data of an associative algebra over $\kk$.
    \item We defined factorization homology with coefficients in an algebra $A:=F(\RR)$ as the left Kan extension along the inclusion $\disk_{1,\orr}\hookrightarrow \mfld_{1,\orr}$. This was opaque and unexplained.
    \item We stated that factorization homology for the circle satisfies $\otimes$-excision.
\end{itemize}

\begin{convention}\label{convention. oo-cats}
From hereon, we will chiefly use the model of $\infty$-categories, rather than utilize the undefined term ``$(\infty,1)$''-category, however useful. In particular, any category that was enriched in simplicial sets in previous chapters will be treated as an $\infty$-category by taking the homotopy coherent nerve (Section~\ref{section. nerves}).
\end{convention}

\clearpage
\section{The example of Hochschild chains}
Let us elaborate on our last example from last lecture.
\index{Hochschild chain complex}

Let $F\colon\disk_{1,\orr}\rta \chain_{\kk}^{\tensor_{\kk}}$ be a symmetric monoidal functor and set $A:=F(\RR)$. I claimed that factorization homology
	\eqnn
		\int_{(-)}A\colon\mfld_{1,\orr}\rta \Chain_\kk
	\eqnd 
allows us to see an action of the group of orientation-preserving diffeomorphisms $\diff^+(S^1)$ on $\int_{S^1}A$. Let's discuss this some more.

Since factorization homology is a functor, we get a map
	\eqnn
		\hom_{\mfld_{1,\orr}}(S^1,S^1)\rta\hom_{\Chain_\kk}(\int_{S^1}A,\int_{S^1}A)
	\eqnd 
and this map is continuous. The circle acts on itself, so we have the inclusion
	\eqnn
		S^1\hookrightarrow\diff^+(S^1)=\hom_{\mfld_{1,\orr}}(S^1,S^1)
	\eqnd 
into the space of orientation-preserving diffeomorphisms.
(This inclusion is a weak homotopy equivalence, though this will not matter for us.)
So we get a continuous map 
	\eqnn
		S^1\rta \hom_{\Chain_\kk}(\int_{S^1}A,\int_{S^1}A).
	\eqnd 
Because this map is continuous, we may study its effect on homotopy groups. The effect on $\pi_0$
	\eqnn
		\pi_0 S^1\rta\pi_0\left(\hom_{\Chain_\kk}(\int_{S^1}A,\int_{S^1}A)\right)\simeq H^{0}\hom_{\Chain_\kk}(\int_{S^1}A,\int_{S^1}A)
	\eqnd 
is not interesting; by definition, a functor must send the identity component $[\mrm{id}]\in\pi_0S^1$ to  the homology class of the identity chain map of $\int_{S^1}A$.\footnote{
\index{Dold-Kan}Here we are using the fact from last time that the homotopy groups of the Dold-Kan space recover the cohomology groups of the $Hom$ cochain complex. See Example~\ref{example:Dold-Kan}.} However, what can we say about the map on fundamental groups?

\begin{theorem}\label{theorem. derham diff is circle action}
Assume that $A$ is a smooth commutative $\kk$-algebra and $\kk$ is perfect. The map 
	\eqnn
		\pi_1S^1\rta\pi_1\left(\hom_{\Chain_\kk}(\int_{S^1}A,\int_{S^1}A)\right)\simeq H^{-1}\hom_{\Chain_\kk}(\int_{S^1}A,\int_{S^1}A)
	\eqnd 
induced by factorization homology sends a generator $1\in\pi_1S^1=\bb{Z}$ to the de Rham differential $[d_{dR}]$.
\end{theorem}

We will not prove this theorem here. (The reader may consult sources such as~\cite{husemoller-cyclic-notes}, \cite{loday-cyclic-book}, Section~1.4 of~\cite{benzvi-nadler-loops-1}, Example 5.5.3.14 of~\cite{higher-algebra}, and Proposition~2.2 of~\cite{loday-quillen}.) Regardless, let us explain its content.

Recall that by excision (Theorem~\ref{theorem: circle excision}
) we have
	\eqnn
		\int_{S^1}A\simeq A\bigotimes_{A\otimes A^\mrm{op}}^\bb{L}A
	\eqnd 
where the righthand side is a well-known chain complex, called the Hochschild chain complex of $A$ (Example~\ref{example:hochschild-complex}). By the Hochschild-Kostant-Rosenberg theorem~\cite{HKR}, 
\index{Hochschild-Kostant-Rosenberg}
when $A$ is a smooth commutative algebra and $\kk$ is a perfect field, there is an isomorphism
	\eqnn
		H^{-l}\left(\text{Hochschild chains on A}\right)
		\cong
		\{\text{degree  $l$ algebraic de Rham forms}\}=\Omega_{A/\kk}^l.
	\eqnd
Yes, you read that correctly. The {\em cohomology} of something recovers {\em forms}. Moreover, things seem to be in wonky degrees: For example, the 1-forms are concentrated in cohomological degree {\em minus} 1. And of course, when we write down de Rham forms, we can also usually write down a de Rham differential. Where is the de Rham differential here? Theorem~\ref{theorem. derham diff is circle action} states that the degree -1 element picked out by the generator of $\pi_1 S^1$ is precisely the de Rham differential.

\begin{remark}
Historically, people {\em combinatorially} exhibited an $S^1$ action on the Hochschild chain complex. This story goes at least as far back as Connes~\cite{connes83}.  Factorization homology exhibits this action more geometrically.

In higher dimensions, and in particular for manifolds other than $S^1$, you might appreciate that combinatorially modeling actions of diffeomorphism groups is not an easy task. Finding such a combinatorial model may even be an ad hoc task one must do one manifold at a time. Factorization homology, for free, exhibits actions of diffeomorphism groups on our invariants. 
\end{remark}

\begin{remark}\label{remark: coherent-group-actions}
An astute reader may have heeded the repeated warnings that functors of $\infty$-categories need not respect composition on the nose (Remark~\ref{remark:oo-category-functors-need-not-respect-composition} and Warning~\ref{warning.infinity-category-warnings}). Now, because of the equivalence of the different models of $\infty$-categories, if one has a functor $f: \cC \to \cD$ of $\infty$-categories, it is true that one can write an honest continuous, group homomorphism
	\eqnn
	\Aut_\cC(X) \to \Aut_{\cD}(fX)
	\eqnd
for any object $X \in \cC$. However, one can only do this after possibly passing to very particular representatives of the homotopy equivalence classes of the automorphism spaces in both the domain and target. So even if $\cC = \mfld_{n,\fr}$ or $\mfld_{n,\orr}$ (which are $\infty$-categories of manifolds that we will introduce below), it is not always feasible, nor advisable, to write a strict and continuous group action of $\diff_{\fr}(X)$ or  $\diff_{\orr}(X)$ on its factorization homology. Instead, the natural output of factorization homology is a homotopy coherent (not strict) action of $\diff$ on the target.

Indeed, at this point, it is the opinion of this lecturer that one ought to get used to the idea of a functor---and of an action---that need not respect composition on the nose, and embrace models like the weak Kan complex model that gets one used to thinking about maps and functors that simply supply extra homotopies. 
\end{remark}

\clearpage
\section{The algebra of disks in higher dimensions}

We will now discuss the algebra of disks in higher dimensions. 

\subsection{Framings and orientations, a first glance} 
\index{$G$-structure!tangential}
Before we get on with it, let us recall Exercise~\ref{exercise: no orientations}. We learned that by considering disks with different kinds of tangential structure (an orientation, or no orientation at all)---and altering our embedding spaces accordingly---we discover different kinds of algebraic structures. The same is true in higher dimensions. 

While there is a different kind of $n$-dimensional disk algebra for any choice of group $G$ equipped with a continuous homomorphism $G \to GL_n(\RR)$, we will mainly consider the cases of $G = SO_n(\RR)$ (the oriented case), and of $G = \{e\}$ (the case of framed manifolds). As we will see, the framed case will have the simplest algebraic description.

\begin{definition}
Let $\disnn$ be the $\infty$-category whose objects are finite disjoint unions of oriented, $n$-dimensional open disks. Morphisms are orientation-preserving smooth embeddings.
\index{$\disnn$}
\end{definition}

\begin{remark}
As before, any object of $\disnn$ is equivalent to $(\RR^n)^{\coprod k}$ for some $k \geq 0$. Disjoint union renders $\disnn$ a symmetric monoidal $\infty$-category.
\end{remark}

\begin{definition}
Fix $\cC^{\tensor}$ a symmetric monoidal $\infty$-category. A {\em $\disk_{n,\orr}$-algebra} in $\cC^{\tensor}$ is a symmetric monoidal functor
	\eqnn
	F: \disk_{n,\orr}^{\coprod} \to \cC^\tensor.
	\eqnd
\index{$\disk_{n,\orr}$-algebra}
\end{definition}

Here is the framed variant:

\begin{definition}
Fix $X$ a smooth $n$-manifold. A {\em framing} on $X$ is a choice of trivialization $\phi:TX\xrta{\cong}X\times \RR^n$ of vector bundles.
\index{framing}
\end{definition}

\begin{remark}
Not every smooth manifold $X$ admits a framing. For example, if $X$ is a compact, orientable, boundary-less 2-manifold, $X$ admits a framing only when $X$ is a torus. Note, however, that a given manifold may have many inequivalent framings. 
\end{remark}

\begin{example}
Let $X$ be a framed $n$-dimensional manifold. The thickening $X \times \RR^k$ may admit many framings that do not decompose as the ``direct product'' of a framing on $X$ with a framing on $\RR^k$.
\end{example}

\begin{definition}[Informal.]
\index{$\disk_{n,\fr}$}
Let $\disk_{n,\fr}$ be the $\infty$-category with objects finite disjoint unions of framed $n$-disks and morphisms smooth embeddings {\em equipped} with a compatibility of framings. 

\index{$\disk_{n,\fr}$-algebra}
A {\em $\disk_{n,\fr}$-algebra} in $\cC^{\tensor}$ is a symmetric monoidal functor 
	\eqnn
	F: \disk_{n,\fr}^{\coprod} \to \cC^{\tensor}.
	\eqnd
\end{definition}

\begin{remark}
This ``compatibility'' of framings doesn't have an expression that you learn in a typical differential geometry class. Let me just say for now that a morphism isn't simply a smooth embedding $j: X \to Y$ satisfying a property, but a morphism is equipped with additional data on $j$. See Section~\ref{section.framings} and Definition~\ref{defn:mfld_fr} at the end of this chapter for details. We will see there that the most natural way to think of framings (and maps of framed manifolds) is by constructing pullbacks of certain natural $\infty$-categories.
\end{remark}

To see the difference between $\disk_{n,\orr}$- and $\disk_{n,\fr}$-algebras, 
let us begin to unpack the definitions. Fix a symmetric monoidal functor $F$ out of $\disk_{n,\orr}$. The reader may benefit from setting $n=2$ for ease of drawing, though we will work with arbitrary $n$. 

As in the one-dimensional case, we denote by $A:=F(\RR^n)$ the value of $F$ on a single disk. On objects, $F$ sends the empty manifold to $\kk$, a single disk $\RR^n$ to $A$, and a disjoint union $(\RR^n)^{\coprod_k}$ to $A^{\otimes k}$.  $F$ also induces a map
	\eqnn
		\hom_{\disk_{n,\orr}}(\RR^n,\RR^n) =: \Emb^{\orr}(\RR^n,\RR^n)\rta \hom_\cC(A,A)
	\eqnd 
So let us first understand $\Emb^{\orr}(\RR^n,\RR^n)$, the space of oriented embeddings of $\RR^n$ to itself.

\begin{lemma}\label{lemma: SOn space of embeddings}
For any $n$, the inclusion of special orthogonal transformations 
	\eqnn
		SO_n(\RR)\rta\Emb^{\orr}(\RR^n,\RR^n)
	\eqnd 
is a homotopy equivalence.
\end{lemma}

\begin{proof}
Let $\Emb^{\orr}_{0}(\RR^n,\RR^n)$ denote the subspace consisting of those embeddings that send the origin to the origin. By translating, the inclusion
	\eqnn
		\Emb^{\orr}_{0}(\RR^n,\RR^n)\rta \Emb^{\orr}(\RR^n,\RR^n)
	\eqnd 
is a homotopy equivalence. (Any embedding $j$ can be translated to the sum $j - tj(0)$; running this from time $t=0$ to $t=1$ gives the retraction.)

Moreover, let us send an origin-preserving embedding $f$ to its difference quotient at the origin:
	\eqnn
		\frac{f(x_0 + t \vec v) - f(x_0)}{t} 
		=
		\frac{f(t\vec v)}{t}
		\qquad
		\text{(at $x_0 = 0$).}
	\eqnd 
Running this difference quotient from time $t=1$ to $t=0$ exhibits a deformation retraction of $\Emb^{\orr}_{0}(\RR^n,\RR^n)$ to $Gl_n^+(\RR)$ (invertible matrices with positive determinant). Finally, the Gram-Schmidt process---whose formulas are all continuous in its parameters---defines a deformation retraction of $Gl_n^+(\RR)$ onto $SO_n(\RR)$.
\end{proof}

\begin{remark}
So an oriented $n$-disk-algebra specifies an object $A = F(\RR^n)$ with some homotopically coherent action of $SO_n(\RR)$, as articulated by the data of $F$ on morphism spaces:
	\eqnn
	SO_n(\RR) \simeq \hom_{\disk_{n,\orr}}(\RR^n,\RR^n)
	\to
	\hom_{\chain_{\kk}}(A,A).
	\eqnd	
This is already a lot of data; actions by (special) orthogonal groups do not grow on trees. 
\end{remark}

Demanding that our embeddings $j$ be equipped with framing compatibilities simplifies the situation considerably. Informally, one can run the deformation retractions in the proof above so that the derivative of $j$ at the origin is equipped with a homotopy to the $n \times n$ identity matrix; so in fact, the inclusion of a point,
	\eqnn
	\ast = \{I_{n \times n}\} \to \hom_{\disk_{n,\fr}}(\RR^n,\RR^n),
	\eqnd
is a homotopy equivalence. That is, the framed embedding space is contractible.We state this as a Lemma for the record:

\begin{lemma}\label{lemma: framed space of embeddings}
For any $n$, the inclusion of the identity
	\eqnn
		\ast \rta \Emb^{\fr}(\RR^n,\RR^n)
	\eqnd 
is a homotopy equivalence.
\end{lemma}

For a proof, see Remark~\ref{remark. framed embs of R^n contractible}.

In summary so far, a $\disk_{n,\orr}$-algebra assigns to $\RR^n$ an object $A$ with an $SO_n(\RR)$ action. The framed version---a $\disk_{n,\fr}$-algebra---simply sends $\RR^n$ to an object (with no specified $SO_n(\RR)$ action).

\subsection{$\EE_n$-algebras}
\index{$\EE_n$-algebra}
What of the multiplication maps? As before, setting $A:= F(\RR^n)$, we would like to interpret morphisms
	\eqnn
	j: \RR^n \coprod \RR^n \to \RR^n
	\eqnd
as inducing multiplication maps $F(j): A \tensor A \to A$. 
When $n \geq 2$, up to isotopy, there is one oriented embedding $j: \RR^n \coprod \RR^n \to \RR^n$ (this is in contrast to the $n=1$ case). So up to chain homotopy, we have a chain map $m: A \tensor A \to A$ specified by $F$; but let us choose a particular $j$ and let $m = F(j)$. 

How does $m$ interact with the swap map?

Let's begin with the two-dimensional case and let $\sigma:\RR^2\sqcup\RR^2\rta \RR^2\sqcup\RR^2 $ be the swap map. Note that, unlike in the 1-dimensional case, one can exhibit an isotopy 
	\eqnn
	j\simeq j\circ \sigma
	\eqnd
(see for example Figure~\ref{image. E2 swap isotopy}).
For every such isotopy, $F$ specifies a homotopy $m \circ \sigma \sim m$. In particular, $m$ is commutative up to homotopy.

\begin{figure}
\caption{\label{image. E2 swap isotopy}
	A homotopy swapping the placement of two embedded disks.
}
	\eqnn
    	\xy
    	\xyimport(8,8)(0,0){\includegraphics[width=2in]{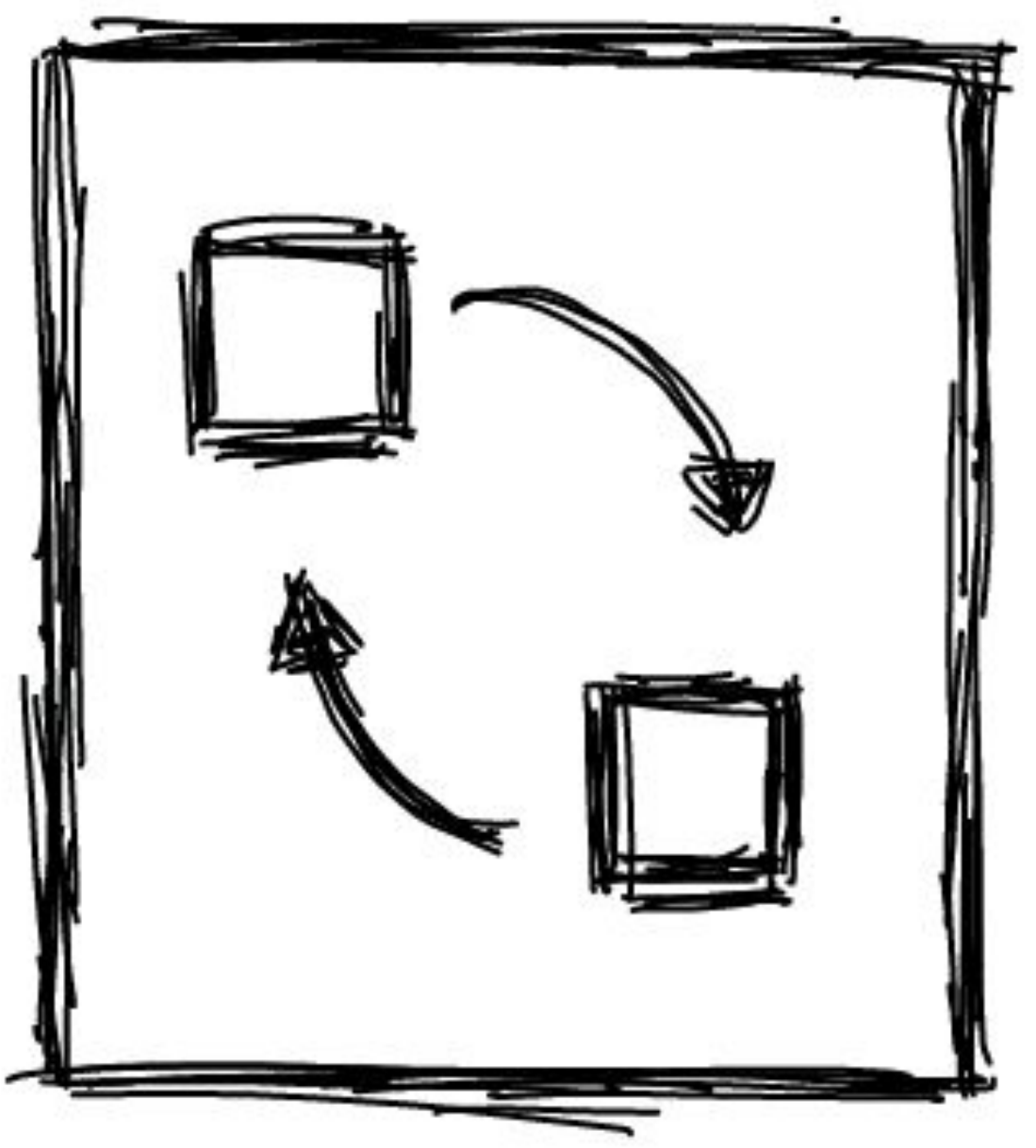}}
   		\endxy
	\eqnd
\end{figure}

\begin{remark}
You have most likely seen this picture before; it is the same picture one draws when proving that the homotopy groups $\pi_n$ for $n \geq 2$ are all commutative. This argument is an example of the so-called Eckmann-Hilton argument.
\end{remark}

However, the above isotopy of disks is far from unique---one could wind a pair of disks around each other one more time. 

\begin{figure}
\caption{An isotopy from a multiplication to itself.}\label{image. winding 2-disks}
	\eqnn
    			\xy
    			\xyimport(8,8)(0,0){\includegraphics[width=2in]{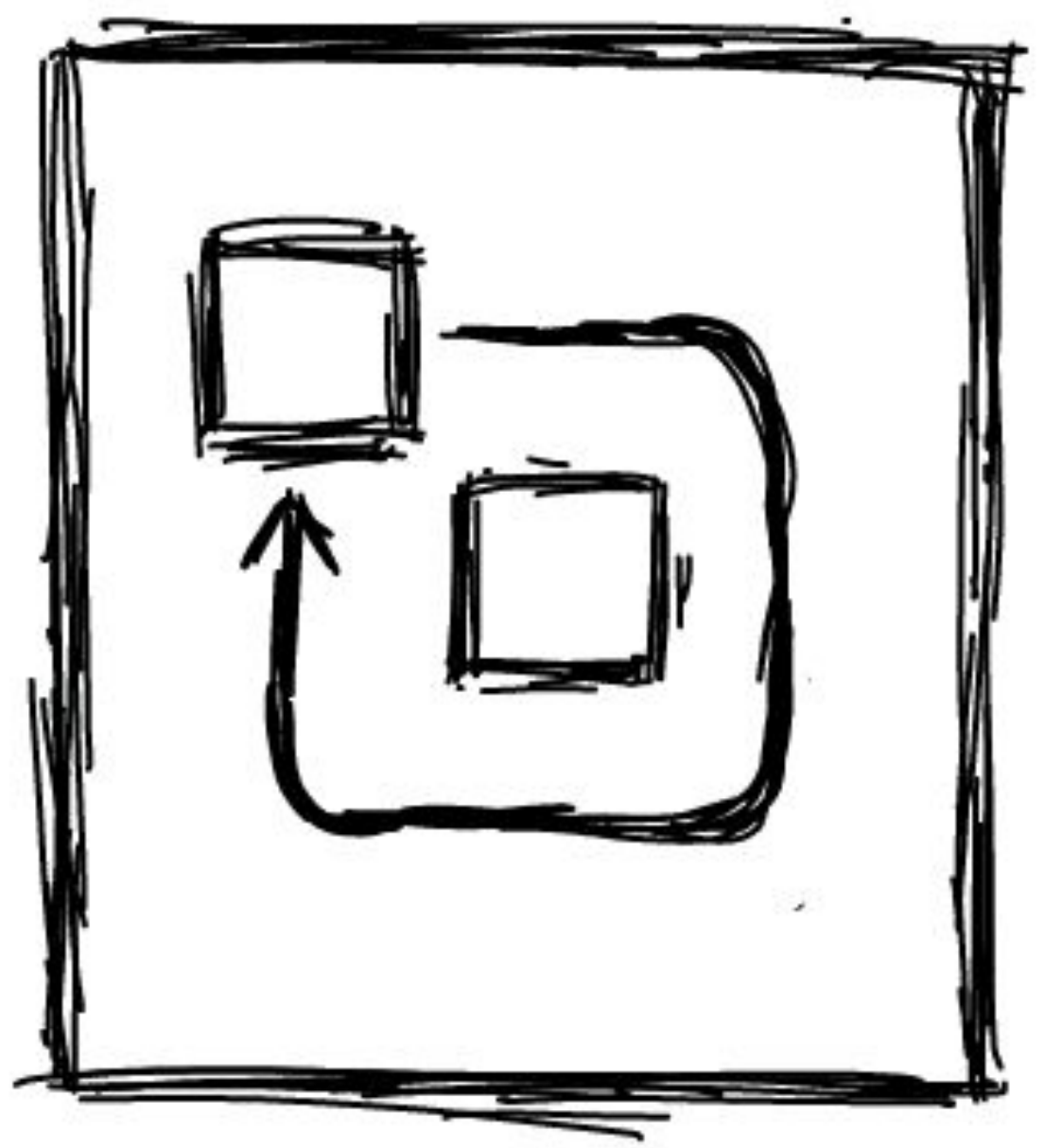}}
    			\endxy
	\eqnd
\end{figure}

So the space of isotopies from $j$ to $j \circ \sigma$ is not contractible (in fact, we see a winding number obstruction). 

So we see that $m=F(j)$ is commutative up to homotopy; but the {\em space} of ways in which this commutativity is exhibited is non-trivial, and $F$ specifies $\ZZ$ many ways to do so (given by the winding number of disks moving past each other). 

\begin{figure}
\caption{A disk exhibiting a null homotopy of the loop from Figure~\ref{image. winding 2-disks}}\label{image. winding 3-disks}
	\eqnn
    			\xy
    			\xyimport(8,8)(0,0){\includegraphics[width=2in]{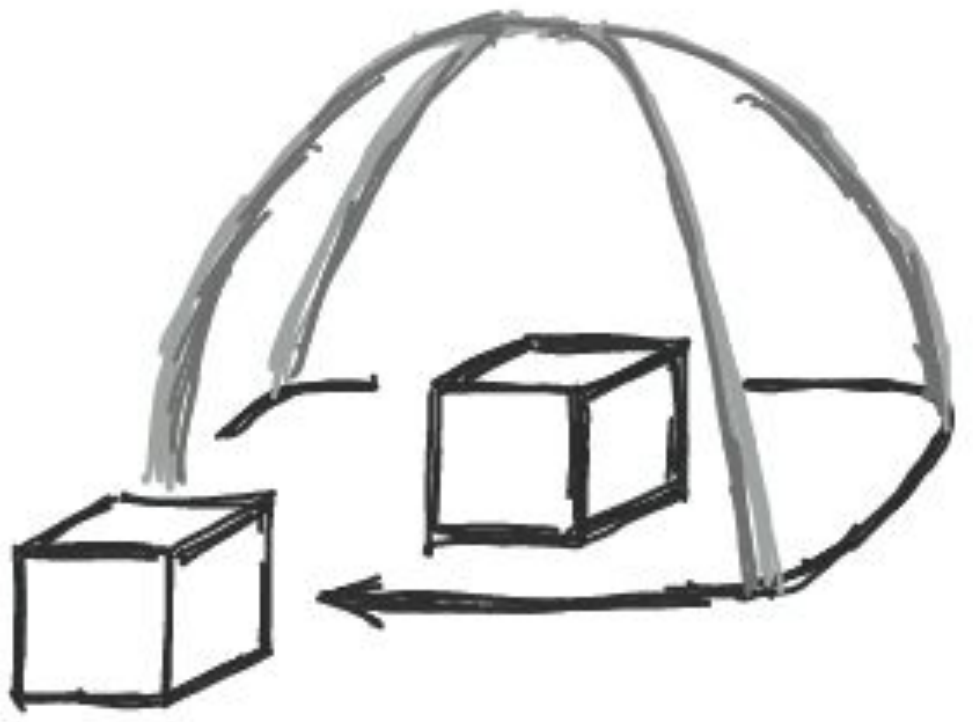}}
    			\endxy
	\eqnd
\end{figure}

In higher dimensions, this winding number obstruction can be trivialized, but another obstruction arises.

In Figure~\ref{image. winding 3-disks}, we have trivialized the winding by choosing a disk whose boundary is the winding isotopy. But we note that another mathematician could have chosen a disk presenting the {\em lower} hemisphere of a sphere, rather than the upper. That is, the space of exhibiting commutativity now has a non-trivial two-sphere appearing. 

In general, in dimension $n$, a $\disk_{n,\fr}$-algebra yields an algebra with a multiplication which is commutative up to homotopy, and the canonicity of the commutativity is obstructed by  an $(n-1)$-dimensional sphere. In a precise sense, the limit as $n \to \infty$ gives rise to the notion of a ``fully'' commutative algebra for homotopy theorists, but we won't get into the details here. The interested reader may look up the term {\em $E_\infty$-algebra}.

\begin{definition} 
\index{$\EE_n$-algebra}
Fix a symmetric monoidal $\infty$-category $\cC^{\tensor}$.
For any $n \geq 1$,
an $\EE_n$-algebra in $\cC^{\tensor}$ is a symmetric monoidal functor $\disk_{n,\fr}^{\coprod} \rta\cC^{\tensor}$.
\end{definition}

\begin{remark}
The notion of an $\EE_n$-algebra was defined (differently) far before being expressed the way we have expressed it. Historically, $\EE_n$-algebras were defined by considering not the space of framed smooth embeddings, but by considering spaces of rectilinear embeddings of bounded cubes---that is, embeddings that component-wise can be written as a composition of scalings and translations. This definition goes back at least to the work of Boardman and Vogt~\cite{boardman-vogt-68}, and was famously studied by May in~\cite{may-iterated-loops-1972}. 

We will proceed using our definition. For now the reader may simply have the intuition that an $\EE_n$-algebra for $n \geq 2$ is an algebra equipped with some commutativity data, but this data is not quite canonical.
\end{remark}

\begin{remark}
We have already seen that a $\disk_{n,\orr}$-algebra (i.e., the oriented setting) specifies at the very least an object $A$ with an action of $SO_n(\RR)$. In a way we do not articulate here, one can informally think of a $\disk_{n,\orr}$-algebra as an $E_n$-algebra equipped with an $SO_n(\RR)$-action which is compatible with the multiplication and commutativity. (It is this compatibility that we do not articulate.)
\end{remark}

\subsection{Examples of $\EE_n$-algebras}\label{section. examples of algebras}
\index{$\EE_n$-algebra!examples}
See also the Exercises at the end of this chapter.

\begin{example}\label{example.commutative-is-En}
Fix $R$ a base ring. We let $\chain_R^{\tensor_R^{\LL}}$ be the symmetric monoidal $\infty$-category whose objects are chain complexes of $R$-modules, and whose symmetric monoidal structure is given by derived tensor product over $R$. (If one takes $R=\kk$ to be a field, one need not derive the tensor product as all chain complexes are flat.)

Then a commutative $R$-algebra is an $\EE_n$-algebra in $\chain_R^{\tensor_R^\LL}$ for any $n$. More generally, a cdga (commutative dg algebra) is an $\EE_n$-algebra for any $n$.
\end{example}

\begin{remark}
In Exercise~\ref{exercise. framing in 1 dim is orientation}, you will show that a framing on a 1-manifold is the same thing as an orientation.
Thus, an $\EE_1$-algebra is the same thing as a $\disk_{1,\orr}$-algebra, which we grew to love in the last lecture. In fact, $\EE_1$ and $A_\infty$ are synonyms to a homotopy-theorist.
\end{remark}

\begin{example}[Dunn additivity]\label{example.dunn-additivity}
\index{Dunn additivity}
Let $\cC^{\tensor}$ be a symmetric monoidal $\infty$-category. One can then show that the $\infty$-category $\EE_1\Alg(\cC^{\tensor})$ of $\EE_1$-algebras in $\cC^{\tensor}$
 is again a symmetric monoidal $\infty$-category under $\tensor$. Hence one can iterate: What are the $\EE_1$-algebras in $\EE_1$-algebras?

The {\em Dunn additivity theorem} states that there exists an equivalence
	\eqnn
	\EE_n\Alg(\cC^{\tensor})
	\simeq
	\EE_1\Alg(\EE_1\Alg(\ldots(\EE_1\Alg(\cC^{\tensor}))))
	\eqnd
between the $\infty$-category of $\EE_n$-algebras in $\cC^{\tensor}$, and the $\infty$-category of $\EE_1$-algebras in $\ldots$ in $\EE_1$-algebras in $\cC^{\tensor}$. For example, an $\EE_2$-algebra is the same data as an $\EE_1$-algebra structure on an $\EE_1$-algebra. By induction, an $\EE_n$-algebra is the same data as an $\EE_1$-algebra structure on an $\EE_{n-1}$-algebra.

Informally, one may thus think of an $\EE_n$-algebra as an object of $\cC^{\tensor}$ equipped with $n$ mutually compatible multiplications, each of which is associative up to coherent homotopy. 
For example, an associative monoid in the category of associative monoids is simply a commutative monoid (Exercise~\ref{exercise. associative associative monoid is commutative}). This is the most common incarnation of the Eckmann-Hilton argument.

See~\cite{dunn-additivity} and a modern account in Section~5.1.2 of~\cite{higher-algebra}. 
\end{example}

\begin{example}
\index{Hochschild cochain complex}
The Hochschild {\em cochain} complex (not the chain complex) of any associative algebra (or dg-algebra, or $A_\infty$-algebra) is an example of an $\EE_2$-algebra in $\cC^{\tensor} = \chain_\kk^{\tensor_k}$. For definitions of the complex, see for example~\cite{hochschild-1945} or Chapter 9 of~\cite{weibel-book}.

It was historically observed that the cohomology of the Hochschild cochain complex (i.e., the Hochschild cohomology) of an associative algebra had an action by the homology groups of the framed embedding spaces of 2-dimensional disks. Deligne conjectured that this action lifts to the chain level---i.e., that the Hochschild cochain complex admitted an $\EE_2$-algebra structure---in 1993. Since then many proofs have been given of this conjecture, the first being Tamarkin's~\cite{tamarkin-1998-deligne-conjecture} as far as we know.
\end{example}

\begin{remark}
A reader may be familiar with the fact that the Hochschild cochain complex of an $A_\infty$-category $\cA$ may be computed as the natural transformations of the identity functor. (When $\cA$ has a single object, one recovers the Hochschild cochain complex of $\cA$ considered as an $A_\infty$-algebra.) That the Hochschild cochain complex has an $\EE_2$-algebra structure is compatible with Dunn additivity (Example~\ref{example.dunn-additivity}) in the following sense: The collection of self-natural transformations of the identity functor has two natural composition maps---one by composing natural transformations, and the other by composing the identity functor with itself (which induces an a priori different---but homotopic---multiplication operation from composition on the collection of natural transformations). 
\end{remark}

\begin{example}[$n$-fold loop spaces]\label{example:loop-spaces}
Let $\TTop^{\times}$ denote the $\infty$-category of topological spaces under direct product. 
Let $X$ be a topological space and choose a basepoint $x_0 \in X$. Let $\Omega X$ denote the space of continuous maps $\gamma:[0,1] \to X$ such that $\gamma(0) = \gamma(1) = x_0$. (This is what one normally calls the {\em based loop space} of $X$.) Then $\Omega X$ is an $\EE_1$-algebra in $\TTop^{\times}$.

This is a ``homotopical lift'' of the well-known statement that the fundamental group $\pi_1(X,x_0)$ of $X$ at $x_0$ is associative under path concatenation.
\index{$\Omega^n X$}
\index{based loop space}

More generally, let $\Omega^n X = \Omega(\Omega(\ldots (\Omega X)))$ denote the $n$-fold based loop space. (An element of $\Omega^n X$ is given by a continuous map $\gamma: [0,1]^n \to X$ for which the boundary of $[0,1]^n$ is sent to $x_0$.) Then $\Omega^n X$ is an $\EE_n$-algebra in $\TTop^{\times}$. 
\end{example}

\begin{example}[Configuration spaces as free algebras]\label{example:free-conf}
\index{configuration spaces}
The free $\EE_n$-algebra on one generator (in $\TTop^\times$) is the topological space
	\eqnn
	\coprod_{l \geq 0} \Conf_l(\RR^n)
	\eqnd
whose $l$th component is the space of unordered configurations of $l$ disjoint points in $\RR^n$. 
\end{example}

\begin{example}
Given any $\EE_n$-algebra $X$ in $\TTop^\times$, the singular chain complex $C_*(X;\kk)$ is an $\EE_n$-algebra in $\Chain_{\kk}^{\tensor_\kk}$. (This is true even when $\kk$ is not a field.) This follows from the fact that the assignment of a space to its singular chain complex is a lax symmetric monoidal functor of $\infty$-categories.

In particular, chains on based loop spaces are examples of $A_\infty$-algebras in chain complexes. 
\end{example}

\begin{remark}
$\EE_n$-algebras are often constructed as deformations of commutative or cocommutative objects. For example, the braided monoidal category of representations of $U_q(\fg)$ is constructed from deforming a cocommutative Hopf algebra (the universal enveloping algebra of $\fg$)~\cite{drinfeld-quantum-groups, jimbo-difference-analogue}.

Another source using combinatorially articulable deformation problems is found in~\cite{kapranov-kontsevich-soibelman}.
\end{remark}

\section{Framings and tangential $G$-structures}\label{section.framings}
Let's address this framing business.
\index{framing}


Suppose you naively define a framed embedding to simply be one that respects framings, in the sense that
	\eqn\label{eqn:framing-naive}
	\xymatrix{
	TX \ar[r]^{\phi_X} \ar[d]^{Dj} & X \times \RR^n \ar[d]^{j \times \id_{\RR^n}}\\
	TY \ar[r]^{\phi_Y} & Y \times \RR^n
	}
	\eqnd
commutes. Here, $j: X\to Y$ is the smooth embedding, $Dj$ is its derivative, and the $\phi$ are the framings on $X$ and $Y$. The commutativity of~\eqref{eqn:framing-naive} is a property of $j$, and merits no extra data on $j$.

Already when $X = Y = \RR^n$, it becomes nearly impossible to find a $j$ making the above diagram commute for arbitrary choices of $\phi_X$ and $\phi_Y$. For example, endow $X$ with the canonical framing one constructs from the $\RR$-vector space structure on $X$. Then any smooth function $\alpha: \RR^n \to GL_n(\RR)$ defines a framing $\phi_Y$ on $Y=\RR^n$. It is a highly difficult differential geometry problem to find a map $j: X \to Y$ such that the derivative $Dj$ recovers exactly the matrices $\alpha$. (Indeed, this integrability problem almost always lacks a solution.) The upshot is that the category of framed disks would have many non-equivalent objects whose underlying manifold is $\RR^n$. 

On the other hand, to have an ``algebra'' structure, we need to be able to embed copies of $\RR^n \coprod \RR^n$ into $\RR^n$. To do this in a way satisfying~\eqref{eqn:framing-naive} is impossible if each component is given the vector space framing.

So we see that the naive notion of ``embeddings that respect framings'' will at least lead us down a complicated path; we can't enumerate objects easily, and articulating the multiplications $\RR^n \coprod \RR^n \to \RR^n$ seems to require us to grapple with this difficult enumeration.

So let us instead articulate some technology to simplify things. Any manifold $X$ comes equipped with a canonical continuous map classifying the tangent bundle:
	\eqnn
	\tau_X: X \to BGL_n(\RR).
	\eqnd
Here, $n = \dim X$ and $BGL_n(\RR)$ is the classifying space of principle $GL_n(\RR)$-bundles.

\begin{definition}\label{defn.G-structure}
\index{$G$-structure}
\index{$G$-structure!tangential}
Fix a topological group $G$ with a map $G \to GL_n(\RR)$. Note this induces a continuous map of classifying spaces $p: BG \to BGL_n(\RR)$. A {\em tangential $G$-structure} on a smooth manifold $X$ is the data of a map $X \to BG$, equipped with a homotopy for the composite $X \to BG \to BGL_n(\RR)$ to the map $X \to BGL_n(\RR)$ classifying the tangent bundle.

We draw this data as a triangle
	\eqnn
	\xymatrix{
	& BG \ar[d]^{p} \\
	X \ar[r]_-{\tau_X} \ar[ur]^{\phi} & BGL_n(\RR)
	}
	\eqnd
where the data of the homotopy is implicit.
\end{definition}

\begin{example}
The data of a homotopy $p \circ \phi \sim \tau_X$ reduces the structure group of $TX$ from $GL_n(\RR)$ to $G$. 
When $G = SO_n(\RR)$, the triangle above is a choice of orientation on $X$. When $G = \ast$ is trivial, the above is the data of a framing on $X$. 
\end{example}

In fact, manifolds with a framing arise as objects of a very natural fiber product of $\infty$-categories:

\begin{definition}[$\mfld_{n,\fr}$. See Definition~2.7 of~\cite{ayala-francis-topological} and Definition~5.0.2 of~\cite{aft-1}]\label{defn:mfld_fr}
\index{$\mfld_n$}
\index{$\mfld_{n,\fr}$}
\index{$\disk_{n,\fr}$}
Consider the $\infty$-category $\mfld_n$ whose objects are smooth $n$-dimensional manifolds and whose morphisms consist of all smooth embeddings; we note that the tangent bundle construction defines a functor
	\eqnn
	\mfld_n \to \TTop_{/BGL_n(\RR)},
	\qquad
	X \mapsto (\tau_X: X \to BGL_n(\RR))
	\eqnd
to the $\infty$-category of all topological spaces equipped with a map to $BGL_n(\RR)$. On the other hand, if one fixes a map $G \to GL_n(\RR)$ (and hence a map $BG \to BGL_n(\RR)$) we have an induced functor $\TTop_{/BG} \to \TTop_{/BGL_n(\RR)}$. When $G = \ast$ is trivial, we define $\mfld_{n,\fr}$ to be the pullback (i.e., fiber product)
	\eqnn
	\xymatrix{
	\mfld_{n,\fr} \ar[r] \ar[d] & \TTop \simeq \TTop_{/\ast} \ar[d] \\
	\mfld_{n} \ar[r]^{\tau} & \TTop_{/BGL_n(\RR)}.
	}
	\eqnd
of $\infty$-categories. Likewise, $\disk_{n,\fr}$ is the full subcategory of $\mfld_{n,\fr}$ obtained by pulling back as below:
	\eqnn
	\xymatrix{
	\disk_{n,\fr} \ar[r] \ar[d] & \TTop \simeq \TTop_{/\ast} \ar[d] \\
	\disk_{n} \ar[r]^{\tau} & \TTop_{/BGL_n(\RR)}.
	}
	\eqnd
\end{definition}

\begin{remark}
For slice categories in the $\infty$-categorical setting, see Definition~\ref{defn. slice categories}, Exercise~\ref{exercise. slice}, and Section~1.2.9 of~\cite{lurie-htt}. Let us informally state the following: Given two objects $(X \to B)$ and $(Y \to B)$ in $\cC_{/B}$, a morphism is given by a homotopy-commutative triangle
	\eqnn
	\xymatrix{
	X \ar[rr] \ar[dr] && Y \ar[dl] \\
	 & B
	}
	\eqnd
where both the map $X \to Y$ and the homotopy rendering the triangle homotopy-commutative are the data of the morphism. 
\end{remark}

\begin{remark}
Because we don't want to get bogged down in the homotopy-theoretic details, let me just point out that the notion of ``pullback'' of spaces is usually not invariant under homotopy equivalences. A common way to get around this issue is, when trying to form a pullback $A \times_C B$, to replace either the map $A \to C$ or $B \to C$ with a ``fibration,'' and then take the honest pullback. Such pullbacks are often called homotopy pullbacks.

A similar issue can arise in any homotopy-theoretic framework; let me just state that, in our examples, we can always assume that the maps between slice categories are appropriate fibrations for $\infty$-categories.

Then, morphism spaces of pullback $\infty$-categories can be computed as homotopy pullbacks of the original morphism spaces. 
\end{remark}

\begin{remark}\label{remark. framed embs of R^n contractible}
Fix $\RR^n \in \mfld_{n,\fr}$ and let us compute its endomorphism space. Because $\RR^n$ is contractible (and the $\infty$-category of spaces identifies weakly homotopy equivalent spaces), we have that
	\eqn\label{eqn: loops BGLn is GL_n}
	\hom_{\TTop_{/BGL_n(\RR)}}(\RR^n, \RR^n)
	\simeq
	\hom_{\TTop_{/BGL_n(\RR)}}(\ast, \ast)
	\simeq
	\Omega BGL_n(\RR)
	\simeq
	GL_n(\RR)
	\eqnd
where the $\Omega$ denotes the based loop space. We have also already seen in the proof of Lemma~\ref{lemma: SOn space of embeddings}---using the trick of translation and derivative-taking---that the inclusion
	\eqn\label{eqn: embeddings GLn}
	GL_n(\RR) \to \hom_{\mfld_n}(\RR^n,\RR^n)
	\eqnd
is a homotopy equivalence. Consider the homotopy pullback diagram of mapping spaces
	\eqnn
	\xymatrix{
	 \hom_{\mfld_{n,\fr}}(\RR^n,\RR^n) \ar[r] \ar[d] &
	 	\ast = \hom_{\TTop}(\ast,\ast) \ar[d] \\
	\hom_{\mfld_n}(\RR^n,\RR^n) \ar[r]^-{\sim}
	 	& 
	\hom_{\TTop_{/BGL_n(\RR)}}(\RR^n, \RR^n).
	}
	\eqnd
The bottom arrow is an equivalence by tracing through~\eqref{eqn: loops BGLn is GL_n} and~\eqref{eqn: embeddings GLn}. Since pullbacks of equivalences are equivalences, 
we conclude that $\hom_{\mfld_{n,\fr}}(\RR^n,\RR^n) \to \ast$ is an equivalence. That is, as promised before, the space of endomorphisms of framed $\RR^n$ is contractible.
\end{remark}

\begin{remark}
More generally, after fixing a continuous group homomorphism $G \to GL_n(\RR)$, one can define the $\infty$-category of $G$-structured manifolds by the pullback of $\mfld_n$ along $\TTop_{/BG} \to \TTop_{/BGL_n(\RR)}$. The same computation as above shows that the endomorphism space of $\RR^n$ in this category is weakly homotopy equivalent to the topological group $G$.
\end{remark}

\begin{remark}\label{remark:disk-G-algebras}
\index{$G$-structure!tangential}
Thus for any continuous group homomorphism $G \to GL_n(\RR)$, one can define the $\infty$-category of $\disk_{n,G}$ as a fiber product $\disk_n \times_{\TTop_{/BGL_n(\RR)}} \TTop_{/BG}$. Informally, an object is a disjoint union of $n$-dimensional Euclidean spaces, each equipped with a tangential $G$-structure. This still has a symmetric monoidal structure given by disjoint union, and one can study $\disk_{n,G}$-algebras, which are symmetric monoidal functors
	\eqnn
	\disk_{n,G}^{\coprod} \to \cC^{\tensor}.
	\eqnd
(For example, when $G=SO(n)$, we have the oriented disk-algebras.) One can define factorization homology as a left Kan extension as before (thereby obtaining invariants of $G$-structured $n$-dimensional manifolds) and---if $\tensor$ preserves sifted colimits in each variable---one can promote factorization homology to a symmetric monoidal functor, and one still has the $\tensor$-excision theorem. This is proven for example in~\cite{ayala-francis-topological,aft-2}.
\end{remark}

\begin{remark}
Tracing through the definition of $\mfld_{n,\fr}$ (Definition~\ref{defn:mfld_fr}), we may finally understand what we mean by equipping embeddings with compatibilities of framings:
Fix framings on manifolds $X$ and $Y$, and fix a smooth embedding $j: X \to Y$. Then a compatibility of $j$ with the framings is the data of higher homotopies rendering the following tetrahedral diagram homotopy-coherent:
	\eqnn
	\xymatrix{
		&BG \ar[dd] \\
		&	& Y \ar[ul]_{\phi_Y} \ar[dl]^{\tau_Y}\\
	X \ar[r]_-{\tau_X} \ar[uur]^{\phi_X} \ar@{->}[urr]
		& BGL_n(\RR)
	}
	\eqnd
This involves the data of the faces containing both $X$ and $Y$, and the three-cell defining the interior of the tetrahedron. The term ``homotopy coherent'' is imprecise; concretely, the above is a 3-simplex in the $\infty$-category $\TTop$. (See Exercise~\ref{exercise. slice}.)
\end{remark}

\clearpage
\section{Factorization homology}

A $\disk_{n,\fr}$-algebra is the local input data of an invariant of framed $n$-manifolds. As in the 1-dimensional case, we want to investigate whether we can extend a functor out of $\disk_{n,\fr}$ to a functor out of $\mfld_{n,\fr}$.  Factorization homology will be this global invariant.

\begin{remark}
In the previous talk, we took $\cC$ to be $\Vect_{\kk}^{\tensor_{\kk}}$ or $\Chain_\kk$. We could have taken $\cC$ to be any symmetric monoidal $\infty$-category such that $\cC$ has all sifted colimits, and so that $\otimes$ preserves sifted colimits in each variable. (See Section~\ref{section: sifted colimits} below.) As such---and to accommodate the examples of $E_n$-algebras from Section~\ref{section. examples of algebras}, we will henceforth consider factorization homology for any such choice of $\cC^{\tensor}$. You can pretend $\cC^{\tensor} = \Chain_\kk^{\tensor_{\kk}}$ if you prefer.
\end{remark}

\begin{definition}[\cite{ayala-francis-topological,aft-2}]
\index{factorization homology}
\index{factorization homology!with coefficients in $A$}
Let $\cC^{\tensor}$ be a symmetric monoidal $\infty$-category admitting all sifted colimits.
Fix an $\EE_n$-algebra in $\cC^{\tensor}$, which we will denote $A$ by abusing notation.
\emph{Factorization homology with coefficients in $A$} is the left Kan extension,
	\eqnn
		\xymatrix{
\disk_{n,\fr}\arw[r]^A\arw[d] & \cC\\
\mfld_{n,\fr}\arw@{-->}[ur]_{\int A} & 
}.
	\eqnd 
Given a framed manifold $X \in \mfld_{n,\fr}$, we let
	\eqnn
	\int_X A
	\eqnd
\index{$\int_X A$}
denote the value of factorization homology, and we call it {\em factorization homology of $X$ with coefficients in $A$}.
\end{definition}

\begin{remark}
\index{$G$-structure!tangential}
We may in general define factorization homology for manifolds with tangential $G$-structures as the left Kan extension of a $\disk_{n,G}^{\coprod}$-algebra along the inclusion $\disk_{n,G} \to \mfld_{n,G}$.
\end{remark}

\begin{remark}
The interested reader may consult Section~\ref{section: left kan extension} below for details on what a left Kan extension is.
\end{remark}

\begin{remark}
Factorization homology is always {\em pointed}, meaning that for any framed manifold $X$, the factorization homology $\int_X A$ is always equipped with a map from the monoidal unit of $\cC^{\tensor}$. This is because the empty set admits an embedding into any manifold, and uniquely so. 

As an example, if $\cC^{\tensor} = \Cat^{\times}$ is the $\infty$-category of categories, the monoidal unit is the trivial category with one object, and the pointing $\ast \to \int_X A$ picks out an object of the category $\int_X A$.

If $\cC^{\tensor} = \chain_{\kk}^{\tensor_\kk}$, then the pointing picks out a degree 0 cohomology class of $\int_X A$, specified by a map $\kk \to \int_X A$. 
\end{remark}

\begin{remark}\label{remark. factorization homology is symmetric monoidal}
Assume $\tensor$ commutes with sifted colimits in each variable. Then factorization homology can be made symmetric monoidal (Lemma~2.16 of~\cite{aft-2}). That is, one can supply natural equivalences
	\eqnn
	\int_{X \coprod Y} A
	\simeq
	\int_X A \bigotimes \int_Y A.
	\eqnd
Put another way, the left Kan extension can be upgraded to a symmetric monoidal left Kan extension:
	\eqnn
		\xymatrix{
\disk_{n,\fr}^{\coprod}\arw[r]^A\arw[d] & \cC^{\tensor}\\
\mfld_{n,\fr}^{\coprod}\arw@{-->}[ur]_{\int A} & 
}.
	\eqnd 
For such a target $\cC^{\tensor}$, it suffices to compute factorization homology on connected manifolds.
\end{remark}

When factorization homology can be made symmetric monoidal, here is how one in principle computes factorization homology:

\begin{theorem}[$\otimes$-Excision, \cite{ayala-francis-topological,aft-2}]\label{theorem: excision}
\index{excision!$\tensor$-excision}
Fix a framed $n$-manifold $X$ and a decomposition 
	\eqn\label{eqn.collar-gluing}
	X=X_0\bigcup_{W\times \RR}X_1
	\eqnd
where each of $X_0, X_1$ is an open subset of $X$, and we have given their intersection $X_0 \cap X_1 \cong W \times \RR$ a direct product decomposition as a smooth manifold.

Let $\cC^{\tensor}$ be a symmetric monoidal $\infty$-category admitting all sifted colimits and such that $\tensor$ preserves sifted colimits in each variable. 
Then there is an equivalence
	\eqn\label{eqn.excision}
		\int_XA\simeq \int_{X_0}A\bigotimes_{\int_{W\times\RR}A}\int_{X_1}A.
	\eqnd 
\end{theorem}

Both the decomposition~\eqref{eqn.collar-gluing} and the tensor product on the righthand side of~\eqref{eqn.excision} warrant an explanation. First, the decomposition:

\begin{remark}\label{remark. framing on W x R}
Note that $X_0$ and $X_1$ form an open cover of $X$, and the crucial part of the decomposition is the choice of direct product decomposition $X_0 \cap X_1 \cong W \times \RR$, where $W$ is an $(n-1)$-dimensional manifold. Of course, being open subsets of $X$, each of $X_0, X_1, W \times \RR$ inherits a framing from $X$. Moreover, since $\RR$ is contractible, the framing map from $W \times \RR$ factors, up to homotopy, through the projection $W \times \RR \to W$. The direct product decomposition thus allows us to interpret the inherited framing on $W \times \RR$ as a framing on the direct sum bundle $TW \oplus \underline{\RR}$ over $W$. Note in particular that $W$ itself need not admit a framing. See also Definition~\ref{defn.n-dim framing}.

This discussion also holds if one replaces the notion of framing by any other $G$-structure (Definition~\ref{defn.G-structure}).
\end{remark}

\begin{remark}
The decomposition~\eqref{eqn.collar-gluing} is an example of a {\em collar-gluing} in the sense of~\cite{aft-1, aft-2}. 
\index{collar-gluing}
\end{remark}

Now, let us explain the tensor product in~\eqref{eqn.collar-gluing}:

\begin{proposition}
$\int_{W\times\RR}A$ is an $\EE_1$-algebra. Moreover, each $\int_{X_i}A$ for $i=0,1$ is a module over $\int_{W\times\RR}A$. 
\end{proposition}

\begin{proof}[Sketch.]
Given $W$, one has an induced symmetric monoidal functor
	\eqnn
	\disk_{1,\orr}^{\coprod} \to \mfld_{n,\fr}^{\coprod}
	\qquad
	\RR^{\coprod k} \mapsto W \times \RR^{\coprod k}.
	\eqnd
(It follows from Remark~\ref{remark. framing on W x R} that the embeddings $ \RR^{\coprod k} \to \RR$ indeed induce maps $W \times \RR^{\coprod k} \to W \times \RR$ with framing compatibilities.)  So we see that $W \times \RR$ is an $\EE_1$-algebra (in $\mfld_{n,\fr}^{\coprod}$). Because factorization homology can be made symmetric monoidal by Remark~\ref{remark. factorization homology is symmetric monoidal}, the composition
	\eqnn
	\disk_{1,\orr}^{\coprod} \xrightarrow{W \times -} \mfld_{n,\fr}^{\coprod} \xrightarrow{\int A} \cC^{\tensor}
	\eqnd
	 exhibits the $\EE_1$-algebra structure on $\int_{W \times \RR} A$. 

(If you don't like the above paragraph for whatever reason, 
just consider those embeddings
	\eqnn
	j:
	(W \times \RR)
	\coprod
	\ldots
	\coprod (W \times \RR)
	\to
	W \times \RR
	\eqnd
which are the identity on the $W$ direction. That is, they are given as a direct product
	$
	\id_W \times h
	$
where $h: \RR \coprod \ldots \coprod \RR \to \RR$ is an orientation-respecting embedding. The images of these embeddings under $\int A$ exhibit the ``associative'' algebra structure on $\int_{W \times \RR} A$. 

On the other hand, because $W \times \RR$ is a collar for $X_1$, we have embeddings
	\eqnn
	\rho:
	X_1 \coprod (W \times \RR)
	\to
	X_1
	\eqnd
given by ``squeezing'' $X_1$ into itself along the collar, then inserting a copy of $W \times \RR$ in the available collar-space.

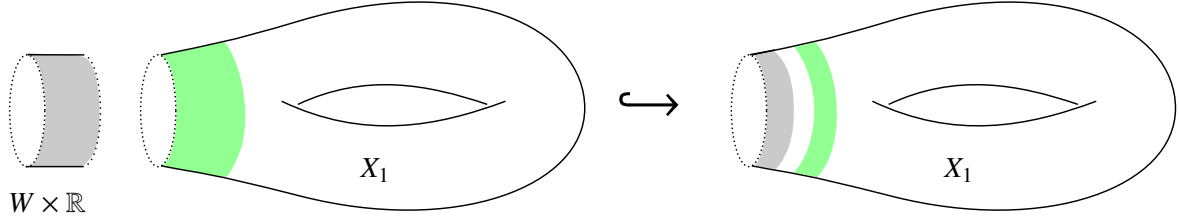
\begin{figure}
\caption{
The left module structure on $X_1$.
On the right, all of $X_1$ has been ``squeezed'' into $X_1$, but we are showing in green only what happens to the indicated collaring region of $X_1$. The grey cylinder is the manifold $W \times \RR$, and on the right we have indicated its image under $\rho$.}\label{image. module action of collars}
\eqnn
    \begin{overpic}[scale=0.55]{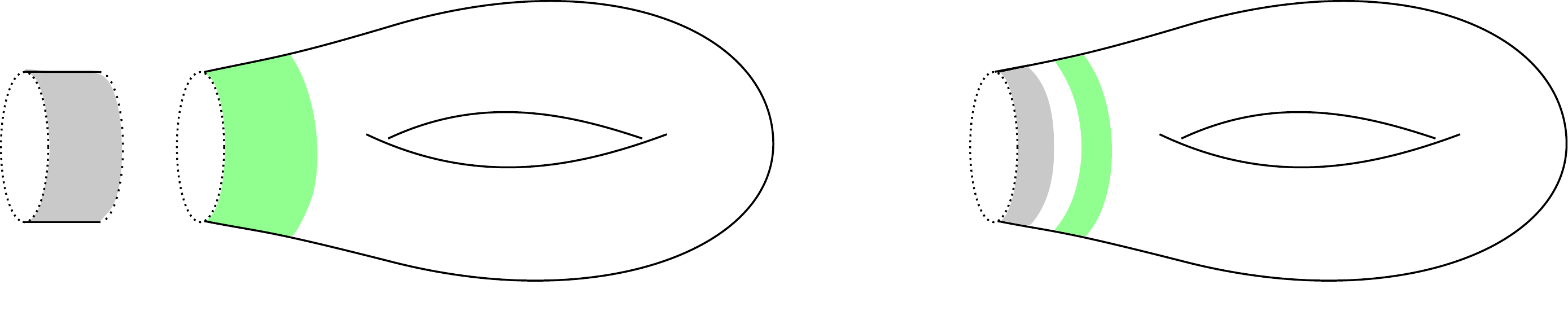}
    \put(52,10){\huge $\hookrightarrow$} 
    \put(0,2){$W\times \R$}
    \put(30,5){$X_1$}
    \put(80,5){$X_1$}
    \end{overpic}
\eqnd
\end{figure}

Then the functor $\int A$ exhibits maps
	\eqnn
	\int_{W \times \RR} A \tensor \int_{X_1} A 
	\simeq 
	\int_{(W \times \RR) \coprod X_1} A 
	\to
	\int_{X_1}A
	\eqnd
which gives $\int_{X_1} A$ a module structure. The same argument shows the module structure on $\int_{X_0} A$.

Note that we have not been precise about how to cohere these embeddings and their framing data; nor have we defined what it precisely means to be a module over an algebra in a symmetric monoidal $\infty$-category (see Section~4.2 of~\cite{higher-algebra}). For these reasons, this proof remains a sketch; but we hope the reader sees the rough idea. 
\end{proof}

\begin{remark}
Let us explain the tensor product in~\eqref{eqn.excision}.
The word ``tensor'' is typically used for usual linear objects over some base ring $R$. However, there are notions of tensor product for modules in arbitrary monoidal categories and arbitrary monoidal $\infty$-categories admitting sifted colimits; given a left module and a right module over an algebra, one writes down a simplicial object called the {\em bar construction}. When sifted colimits (and in particular colimits of simplicial objects) exist, the colimit, or geometric realization, of the bar construction is what one usually calls the tensor product of two modules over an algebra. Most readers will not lose any intuition by imagining that the tensor product in Theorem~\ref{theorem: excision} is simply a derived tensor product.
\end{remark}

\subsection{Examples}

Factorization homology exhibits mapping group actions:

\begin{example}
Fix a $\disk_{2,\orr}$-algebra $\cal{R}$ in $\cC^{\tensor}$. For any oriented genus $g$ surface $\Sigma_g$, let $\diff^+(\Sigma^g)$ denote the space of orientation-preserving diffeomorphisms of $\Sigma_g$. Then factorization homology induces a map
	\eqnn
		\hom_{\mfld_{2,\orr}}(\Sigma_g,\Sigma_g) = \diff^+(\Sigma_g)\rta \hom_\cC(\int_{\Sigma_g}\cal{R},\int_{\Sigma_g}\cal{R}).
	\eqnd 
Taking connected components, we get an action of the mapping class group on $\pi_0$ of the right hand side.

When $\cC^{\tensor}$ is the $\infty$-category of $k$-linear categories, ~\cite{bzbj} recovers well-known mapping class group actions on certain invariants of quantum groups.
\end{example}

Factorization homology can also be computed iteratively for product manifolds:

\begin{example}
Fix the product framing of the torus $S^1 \times S^1 = T^2$. For an $\EE_2$-algebra $A$, let $B = \int_{S^1 \times \RR} A$ be Hochschild chains, which now has an $E_1$-algebra structure by virtue of the $\RR$ factor. We have 
	\eqnn
		\int_{T^2}A
		\simeq  
			\int_{S^1\times \RR}A
				\bigotimes_{\int_{(S^1\times \RR)^{\sqcup_2}}A}
			\int_{S^1\times \RR}A
		\simeq 
			B
				\bigotimes_{
					B \tensor B^{\op}
					}
			B
	\eqnd 
is Hochschild chains of Hochschild chains. 
\end{example}

This example can also be exhibited using the following:

\begin{theorem}[Fubini Theorem]
\index{Fubini Theorem}
There is an equivalence
	\eqnn
		\int_{X\times Y}A\simeq\int_X\int_{Y\times{\RR}^{\dim X}}A
	\eqnd 
\end{theorem}

See Proposition~3.23 of~\cite{ayala-francis-topological} and Corollary~2.29 of~\cite{aft-2}.

\begin{remark}
To make sense of the Fubini theorem, note that factorization homology for $Y \times \RR^{\dim X}$ has the structure of an $\EE_d$-algebra for $d = \dim X$.
\end{remark}

Factorization homology also computes mapping spaces: (Of course, it is an old problem of topology to be able to compute invariants of mapping spaces.)

\begin{theorem}[Non-abelian Poincar\'e Duality]
\index{Non-abelian Poincar\'e Duality}
Let $\cC^{\tensor} = \TTop^{\times}$. Fix a topological space $X$ which has trivial homotopy groups in dimensions $\leq n-1$, choose a basepoint $x_0 \in X$, and consider the $\EE_n$-algebra $\Omega^n X$ (Example~\ref{example:loop-spaces}). Then for any framed $n$-manifold $M$, we have that
	\eqnn
	\int_M \Omega^n X \simeq Map_c(M,X).
	\eqnd
That is, factorization homology is homotopy equivalent to the spaces of compactly supported maps from $M$ to $X$. (Here, compact support of $f: M \to X$ means that outside some compact subset of $M$, $f$ is constant with value given by the basepoint of $X$.)
\end{theorem}

\begin{remark}
The above equivalence is called {\em non-abelian Poincar\'e duality} and is a theorem due in several guises to Salvatore~\cite{salvatore}, Lurie~\cite{higher-algebra}, Ayala-Francis~\cite{ayala-francis-topological}, and Ayala-Francis-T~\cite{aft-2}.

In the case $n=1$ with $M=S^1$, it states that Hochschild homology of a based loop space $\Omega X$ is homotopy equivalent to the free loop space of $X$ (for $X$ connected). This recovers a theorem of Burghelea-Fiedorowicz~\cite{burghelea-fiedorowicz} and Goodwillie~\cite{goodwillie} after applying the singular chains construction. 
\end{remark}

Factorization homology is intimately tied to configuration spaces, which are also recurring characters in topology:
\index{configuration spaces}

\begin{theorem}\label{theorem:conf-as-fact-hom}
Let $\cC^{\tensor} = \TTop^{\times}$ once more, and let $A$ be the free $E_n$-algebra on one generator (Example~\ref{example:free-conf}). Then
	\eqnn
	\int_M A
	\simeq
	\coprod_{l \geq 0} \conf_l(M).
	\eqnd
That is, factorization homology of $M$ with coefficients in the free $E_n$-algebra on one generator is homotopy equivalent to the disjoint union over $l \geq 0$ of the configuration spaces of $l$ disjoint, unordered points in $M$. 
\end{theorem}

\begin{remark}
Theorem~\ref{theorem:conf-as-fact-hom} is a simple case of Proposition~5.5 of~\cite{ayala-francis-topological} and Proposition~4.12 of~\cite{aft-2}.
\end{remark}

\begin{example}
\index{Dold-Thom theorem}
\index{configuration spaces}
Another example illustrating some of the ingredients of factorization homology is Bandklayder's proof~\cite{bandklayder} of the  Dold-Thom theorem~\cite{dold-thom} for manifolds. Recall that the Dold-Thom theorem states the following: Fix a connected, reasonable topological space $X$ along with a basepoint $x_0$ and an abelian group $A$. Then there exist natural isomorphisms
	\eqnn
	\pi_k(\Sym(X;A)) \cong \widetilde{H_k}(X;A),
	\qquad
	k \geq 1,
	\eqnd
between the homotopy groups of the {\em infinite symmetric product of $X$ with labels in $A$} and the reduced homology groups of $X$ with coefficients in $A$. Informally, $\Sym(X;A)$ is a configuration space of disjoint, unordered points in $X$ labeled by elements of $A$---topologized so that any element labeled by the identity $e \in A$ disappear; so that when points collide, their labels add; and so that any point that collides with $x_0$ loses its label, or ``disappears.'' (One may reasonably think of the basepoint $x_0$ as a point at infinity where labeled points go to be forgotten.)
\end{example}

Finally, the more ``commutative'' the coefficient algebra, the less sensitive factorization homology is to the smooth diffeomorphism type of the manifold (and becomes more an invariant of its homotopy type). 

\begin{example}[Usual homology.]
\index{factorization homology!recovering ordinary homology}
See Exercise~\ref{exercise.classical-homology} for how one recovers usual homology by using factorization homology with coefficients in an abelian group (which is a commutative algebra in $\cC^{\tensor}=\chain_{\ZZ}^{\oplus}$). 
\end{example}

\begin{example}[Pirashvili's higher order Hochschild invariants]
\index{Pirashvili}
Fix a base ring $R$. Recall that any commutative $R$-algebra $A$ is an $\EE_n$ algebra in $\chain_R^{\tensor^{\LL}_R}$ for all $n$ (Example~\ref{example.commutative-is-En}). Hence one may compute factorization homology on a framed manifold of any dimension. 

Such invariants were studied under the name of {\em Higher order Hochschild homology} by Pirashvili~\cite{pirashvili}---there, rather than take a framed manifold, Pirashvili constructed invariants associated to any simplicial set\footnote{E.g., what we have called a ``combinatorially defined'' space before, see Remark~\ref{remark. simplicial sets are spaces}.}. (Pirashvili also studied the case where one may take a bimodule $M \neq A$ as an additional datum---see Section~\ref{section.labeled-factorization-homology}). The way this intersects with our story is as follows: Given any manifold $X$, equip it with a homotopy equivalence to a simplicial set $Y$ (for example by taking the singular complex of $X$, see Remark~\ref{remark:singular-complex}). Then Pirashvili's work constructs an explicit chain complex out of $Y$ and $A$ that computes factorization homology of $X$ with coefficients in $A$.

These methods can be generalized more generally to cdgas (commutative differential graded algebras) over $R$.\footnote{Note that ordinary commutative algebras are examples of cdgas concentrated in degree 0.} Indeed, let $\cdga_R$ be the $\infty$-category of cdgas as sketched in Example~\ref{example:cdga-oo-category}. Fixing a cdga $A$, and thinking of $A$ as an $\EE_n$-algebra, for any framed $n$-manifold we have the following equivalence:
	\eqnn
	\int_X A
	\simeq
	X\tensor_{\cdga_R} A 
	\qquad
	\in \chain_R.
	\eqnd
Here we have used the tensoring of the $\infty$-category of cdgas over spaces---i.e., $X \tensor_{\cdga_R} A$ is the colimit of $A$ in $\cdga_R$ indexed by the constant functor in the shape of $X$. For example, we have equivalences
	\eqnn
	\hom_{\cdga}(X \tensor_{\cdga} A, B)
	\simeq
	\hom_{\TTop}(X, \hom_{\cdga}(A,B)).
	\eqnd
See for instance Section~3.2 of~\cite{ginot-tradler-zeinalian} and Proposition~5.7 of~\cite{ayala-francis-primer}.
\end{example}

\begin{remark}
The above example hints at how, when $A$ is a commutative algebra, one may in general construct invariants of reasonable topological spaces---say, those that admit a proper embedding into $\RR^N$ for large $N$ (for example, suitably finite CW complexes). One may then take a small open neighborhood of the embedded space, which is a framed $N$-manifold, and compute its factorization homology with coefficients in $A$ considered as an $\EE_N$-algebra. All such invariants will be invariant under homotopy equivalences. 
\end{remark}

\clearpage
\section{Leftovers and elaborations}\label{section. leftovers}

\subsection{What are left Kan extensions?}\label{section: left kan extension}
\index{left Kan extension}
Let me at least say something about left Kan extensions.

Fix two functors $F: \cC \to \cE$ and $a: \cC \to \cD$. One can ask if there is a ``canonical'' way to extend $F$ to a functor emanating from $\cD$:	
	\eqnn
	\xymatrix{
	\cC \ar[r]^F \ar[d]^a & \cE \\
	\cD \ar@{-->}[ur]_{\exists ?}
	}
	\eqnd
Suppose you are given the data of a pair $(G,\eta)$ where $G$ is a functor $\cD \to \cE$, and $\eta: F \to G \circ a$ is a natural transformation. Such a pair is called a {\em left Kan extension} if it is initial with respect to all possible $(G,\eta)$; that is, given any other $(G',\eta')$, one can find a unique (up to contractible choice) natural transformation from $(G,\eta)$ to $(G',\eta')$. (See Section~4.3 of~\cite{lurie-htt} for details.)

The word ``extension'' is slightly misleading, as $G \circ a$ may not agree with $F$. However, if $a$ is a fully faithful inclusion (as in the case $\disk \to \mfld$ of interest to us), $G \circ a$ can be made naturally equivalent to $F$.

A left Kan extension exists if $\cE$ admits certain colimits. In fact, one can prove the following formula for left Kan extensions:
	\eqnn
	G(d) = \colim_{x \in \cC_{/d}} F(x).
	\eqnd
The colimit is diagrammed by the slice category of objects in $\cC$ over $d \in \cD$. In the case of factorization homology, we have
	\eqn\label{eqn:factorizationhomology-formula}
	\int_X A
	=
	\colim_{(\disk_{n})_{/X}} F.
	\eqnd
Concretely, the colimit is indexed by collections of (disjoint unions of) disks equipped with embeddings into the manifold $X$, and we evaluate the symmetric monoidal functor $F$ on these disks. In this way, factorization homology may be interpreted as ``the most efficient invariant glued out of ways to embed disks into $X$.'' (We have omitted mention of framings and orientations from the notation.)

\begin{remark}
While the idea of a left Kan extension has been cast aside here as an afterthought, we would like to remark that our ability to use left Kan extensions for $\infty$-categories depends on a robust enough theory of $\infty$-categories and colimits within an $\infty$-category, just as the classical notion depends on a good development of the language of categories and colimits. Though it seems we can guarantee that a left Kan extension induces a continuous map of morphism spaces ``for free;'' this lunch indeed is not free, and it is made possible by machinery developed by Joyal~\cite{joyal} and Lurie~\cite{lurie-htt}. 
\end{remark}

\subsection{What's up with sifted colimits?}\label{section: sifted colimits}
\index{sifted!colimit}
\index{sifted!diagram}
Recall that colimits are a way to ``glue together'' objects in a category. Likewise, colimits can be articulated in $\infty$-categories, though a given $\infty$-category may not always have the colimits you want (just as in ordinary category theory).

\begin{remark}
We have not defined the notion of colimits in an arbitrary $\infty$-category, but colimits satisfy a universal property analogous to the classical notion: Colimits are initial among objects receiving a map from a diagram. See 1.2.13 of~\cite{lurie-htt}.
\end{remark}

\begin{remark}[$\infty$-categorical colimits are homotopy colimits]
Let us also make a comment for readers familiar with the notion of homotopy colimits. The notion of a colimit in an $\infty$-category agrees with the notion of a homotopy colimit whenever the $\infty$-category arises from a setting in which homotopy colimits make sense. (For example, as the nerve of a model category.) See~4.2.4 of~\cite{lurie-htt} for details.
\end{remark}

We often classify colimits by the shape of the diagram encoding the gluing. A particular class of diagrams is given by the sifted diagrams, and a sifted colimit is a colimit glued out of a sifted diagram.\footnote{A diagram (i.e., an $\infty$-category) $\cD$ is called sifted if it is non-empty and if the diagonal inclusion $\cD \to \cD \times \cD$ is left final (i.e., {\em cofinal} in some works). Typical examples include filtered diagrams and $\Delta^{\op}$.}

It turns out that for any $X \in \mfld$---framed or not---the colimit in~\eqref{eqn:factorizationhomology-formula} is a sifted colimit, so the left Kan extension is guaranteed to exist if $\cC$ has sifted colimits.

Moreover, one can also extend $\int A$ to be a {\em symmetric monoidal} functor if $\tensor$ commutes with sifted colimits in each variable.

Finally, let us say that the tensor product
	\eqnn
	M \tensor_R N
	\eqnd
of two modules over an algebra can be presented as a bar construction, and the bar construction is indexed by a simplicial diagram, which in particular is a sifted diagram. That is to say, the bar construction (the tensor product) always exists if $\cC$ admits sifted colimits. 

This explains why we assume that $\cC$ admits sifted colimits and why we assume that its symmetric monoidal structure preserves sifted colimits in each variable.

\begin{example}
Let $\cC^\tensor = \chain_{\kk}^{\oplus}$. Note that the symmetric monoidal structure here is the direct sum of cochain complexes; i.e., the coproduct. $\oplus$ does preserve sifted colimits in each variable, but it does not preserve all colimits in each variable. (For example, it doesn't preserve itself in each variable!) See also Exercise~\ref{exercise.classical-homology}.
\end{example}

\subsection{How many excisive theories are there?}
Let $\cC^{\tensor}$ be a symmetric monoidal $\infty$-category admitting sifted colimits, and for which $\tensor$ preserves sifted colimits in each variable. Fix an $\EE_n$-algebra $A$.

We have seen that factorization homology results in a functor $\int A: \mfld_{n,\fr}^{\coprod} \to \cC^{\tensor}$ which is symmetric monoidal and $\tensor$-excisive (Remark~\ref{remark. factorization homology is symmetric monoidal} and Theorem~\ref{theorem: excision}). One could ask the following question: Are there other functors other than factorization homology that could satisfy these properties?

The answer is no. More precisely, we have the following:

\begin{theorem}[Theorem~3.24 of~\cite{ayala-francis-topological} and Theorem~2.45 of~\cite{aft-2}]\label{theorem.excisives-are-fact-hom}
Restriction to the full subcategory $\disk_{n,\fr} \subset \mfld_{n,\fr}$ of disks defines an equivalence of $\infty$-categories 
	\eqnn
	\{\text{
	$\tensor$-excisive, symmetric monoidal $F: \mfld_{n,\fr}^{\coprod} \to \cC^{\tensor}$
	}\}
	\to
	\text{$\disk_{n,\fr}$-algebras in $\cC^{\tensor}$}
	\eqnd
with an inverse functor implemented by factorization homology. 
\index{$\disk_{n,\fr}$-algebra}
\index{excision!$\tensor$-excision}
\end{theorem}

This means that any symmetric monoidal, $\tensor$-excisive invariant arises as factorization homology of some $\EE_n$-algebra. The proof itself is not too difficult once the machinery is set up: One uses the assumption of $\tensor$-excision to show by induction that the values of an excisive theory on handles are completely determined by their values on disks, then one uses handle decompositions to show the same for arbitrary smooth manifolds. The proofs of the theorems in~\cite{ayala-francis-topological} and~\cite{aft-2} cited above are slightly more complicated because~\cite{ayala-francis-topological} addresses the case of topological manifolds (where handlebody decompositions are not guaranteed in dimension 4) and because~\cite{aft-2} addresses a stratified generalization for which one must rely on (analogous, but involved to set up and verify) decomposition results for stratified spaces as developed in~\cite{aft-1}. 

\begin{remark}
Theorem~\ref{theorem.excisives-are-fact-hom} holds true also in the generality of arbitrary $G$-structures (Remark~\ref{remark:disk-G-algebras}): Any symmetric monoidal, $\tensor$-excisive functor for $G$-structured manifolds arises as factorization homology of a $\disk_{n,G}$-algebra. 
See again Theorem~3.24 of~\cite{ayala-francis-topological} and Theorem~2.45 of~\cite{aft-2}.
\end{remark}

\begin{remark}\label{remark.open-exhaustion-colimits}
If one wants to consider manifolds that are larger than our Convention~\ref{convention:manifolds-small}, the above classification theorem must be modified slightly. Every symmetric monoidal, $\tensor$-excisive functor that {\em preserves sequential colimits of countable open exhaustions} arises as factorization homology of some algebra. (See for example Definition~2.37 of~\cite{aft-2}.)
\end{remark}

\subsection{Locally constant factorization algebras}
\index{factorization algebras}
The origins of factorization homology are rooted in the work on chiral algebras of Beilinson and Drinfeld~\cite{beilinson-drinfeld}. Factorization algebras are another perspective on how to take local algebraic structures and form global invariants; see the work of Costello-Gwilliam~\cite{costello-gwilliam}. 

I am often asked about the equivalence between locally constant factorization algebras on $\RR^n$, and $\EE_n$-algebras. I refer the reader to Section~2.4 of~\cite{aft-2} for one formulation, where we exhibit the $\infty$-category of disk embeddings as a localization of a discrete version. 

Another equivalence proven using the formulation of cosheaves on the Ran space of $\RR^n$ can be found in~\cite{higher-algebra}. This approach works for not-necessarily-unital $\EE_n$-algebras.  

\subsection{How good a manifold invariant is factorization homology?}
This is also a natural question. Roughly speaking, it seems to be ``about as good as the homotopy type of configuration spaces.'' 
\index{configuration spaces}
It is unknown how good a manifold invariant the homotopy type of configuration spaces are; for some time they were conjectured to be only sensitive to the homotopy type of a manifold, but now they are known to distinguish homotopy-equivalent (but non-diffeomorphic) manifolds of the same dimension~\cite{longoni-salvatore}.

To see why configuration spaces enter the picture, recall that the free $E_n$-algebra (in topological spaces) generated by a single element is the configuration space of unordered points in $\RR^n$ (Example~\ref{example:free-conf}). By the free-forget adjunction, any $E_n$-algebra may be resolved by algebra maps from this configuration space, and thus factorization homology with coefficients in an arbitrary $E_n$-algebra receives maps from (a diagram made up of) the factorization homology with coefficients in the free algebra. On the other hand, factorization homology of a manifold with coefficients in the free algebra is homotopy equivalent to the configuration space of the manifold (Example~\ref{theorem:conf-as-fact-hom}).

\subsection{Are we stuck with algebras only?}\label{section.labeled-factorization-homology}
There is a variant of factorization homology where one can obtain invariants of not only algebras, but algebras equipped with the data of bimodules, and also of higher categories. This is developed in~\cite{aft-1,aft-2} (in the case of algebras and bimodules) and in~\cite{ayala-francis-rozenblyum} (for higher categories). For example, by articulating what we mean by a category of 1-manifolds equipped with marked points, factorization homology over a circle with marked points recovers Hochschild homology of an algebra (assigned to open 1-dimensional disks) with coefficients in the desired bimodules (assigned to marked points).

\clearpage
\section{Exercises}

\begin{exercise}\label{exercise. framing in 1 dim is orientation}
When $n=1$, show that a framing on a 1-manifold is the same thing as an orientation; moreover, show that the space of framed embeddings (i.e., embeddings equipped with framing compatibility) is equivalent to the space of orientation-preserving embeddings. You will need to consult Section~\ref{section.framings} for this.
\end{exercise}

\begin{exercise}
Exhibit an equivalence of $\infty$-categories between $\mfld_{1,\fr}$ and $\mfld_{1,\orr}$.
\end{exercise}

\begin{exercise}\label{exercise.E2-in-vect-and-set-is-commutative}
Let $\cC^{\otimes}$ be $\Vect_{\kk}^{\tensor_{\kk}}$. Show that for $n\geq 2$, an $\EE_n$-algebra in $\Vect_{\kk}^{\tensor_{\kk}}$  is a commutative algebra over $\kk$. Likewise, show that an $\EE_n$-algebra in the category of sets $\Set^\times$ is a commutative monoid.
\end{exercise}

\begin{exercise}\label{exercise. associative associative monoid is commutative}
Show that an associative monoid in the category of associative monoids is simply a commutative monoid. (We are testing Dunn additivity for $n=2$ when $\cC^{\tensor} = \Set^\times$, the category of sets under direct product.)

Likewise, show that an associative algebra in the category of associative $\kk$-algebras is a commutative $\kk$-algebra.
Compare with Exercise~\ref{exercise.E2-in-vect-and-set-is-commutative}.
\end{exercise}

\begin{exercise}
Take $\cC^\otimes=\Cat^\times$. Show that an $\EE_2$-algebra in $\Cat$ is a braided monoidal category~\cite{joyal-street}. (You can see where the word ``braided'' comes from in this picture: What does a movie of moving 2-dimensional disks around each other look like?)

Show that a $\disk_{2,\orr}$-algebra is a balanced monoidal category. (In particular, ribbon categories are examples of $\disk_{2,\orr}$-algebras.) 

What about $\EE_n$-algebras in $\Cat$ for $n\geq 3$?
\end{exercise}

\begin{exercise}
Spell out $\EE_n$-algebra structures for Examples~\ref{example:loop-spaces} and~\ref{example:free-conf}.
\end{exercise}

\begin{exercise}
Let $\cC^{\tensor} = \Vect_{\kk}^{\tensor_\kk}$ be the ordinary category of vector spaces and fix a unital associative $\kk$-algebra $A$. Using the general excision theorem, Theorem~\ref{theorem: excision} of this chapter, recover the excision theorem for the circle from the previous chapter (Theorem~\ref{theorem: circle excision}).

Be very careful about orientations of circles; the point is to understand why the tensor product contains an $A^{\op}$ factor.
\end{exercise}

\begin{exercise}
How many framings does $S^1 \times \RR$ admit? 

(This may be difficult.) Which of these allow for $S^1 \times \RR$ to be an $\EE_1$-algebra in $\mfld_{2,\fr}^{\coprod}$?
\end{exercise}

\begin{exercise}
How many framings does $S^1 \times S^1$ admit? 
\end{exercise}

\begin{exercise}
(Open problem.) Can you classify the $\EE_n$-algebras in $\mfld_{n,\fr}^{\coprod}$?
\end{exercise}

\begin{exercise}
\index{excision!for usual homology}
Let $\cC^{\tensor} = \chain^{\oplus}$ be the $\infty$-category of cochain complexes, but with direct sum as the symmetric monoidal structure. Show that any cochain complex $V$ admits an $\EE_n$-algebra structure by defining $V \oplus V \to V$ to be the addition map. Show further that this structure is unique.

Note that I have left out the subscript $\kk$; this exercise is valid for cochain complexes over an arbitrary base ring, including $\ZZ$.
\end{exercise}

\begin{exercise}\label{exercise.classical-homology}
\index{factorization homology!recovering ordinary homology}
Let $\cC^{\tensor}=\Chain^\oplus$ and fix $A$ any abelian group (considered as a cochain complex in degree 0). By the previous exercise, addition endows $A$ with an $\EE_n$-algebra structure for any $n$.

First, show that $\oplus$ preserves sifted colimits in each variable.

Then, show that for any framed manifold $X$, there is a quasi-isomorphism of cochain complexes
\eqnn
	\int_X A \simeq C_*(X;A)
\eqnd
where $C_*(X;A)$ is the cochain complex of singular chains. (Hint: Show that
$\otimes$-excision for $\oplus$ gives rise to the Mayer-Vietoris sequence; this is a consequence of the usual excision theorem for singular homology.) 

This exercise shows that factorization homology generalizes
ordinary homology, hence the word ``homology.'' Your proof also shows why the $\tensor$-excision theorem has the word ``excision'' in it.

Can you recover the homology of any compact manifold using factorization homology? How about for any finite CW complex? 
\end{exercise}

\chapter{Topological field theories and the cobordism hypothesis} 

In this chapter we introduce the notion of a fully extended topological field theory. We will sketch a definition of fully dualizable object in a symmetric monoidal $(\infty,n)$-category, and sketch also why the framed point is a fully dualizable object in the $(\infty,n)$-category of framed cobordisms, stating also the Baez-Dolan cobordism hypothesis. We will then end with a statement of Scheimbauer's result that factorization homology defines a fully extended topological field theory.

\clearpage
\section{Review of last lecture}
\begin{definition}
Let $\cC^\otimes$ be a symmetric monoidal $\infty$-category, 
e.g. $\Vect^\otimes$, $\Chain^\otimes$ or $\Cat^\times$. An $\EE_n$-algebra in $\cC$
is a symmetric monoidal functor $\disk_{n,\fr}^{\coprod} \to \cC^\otimes$. 
\end{definition}
We give some examples from the previous lectures in the following table:
\begin{center}
\begin{tabular}{|c|c|c|c|}
\hline 
 & $\Vect^\otimes$ & $\Cat^\times$ & $\Chain^\otimes$  \\ 
\hline 
$\EE_1$ & Unital associative $\K$-algebras & Monoidal categories & $A_\infty$-algebras   \\ 
\hline 
$\EE_2$ & Unital commutative $\K$-algebras & Braided monoidal categories &  $\EE_2$-algebras  \\ 
\hline 
$\EE_3$ & $``$ & Symmetric monoidal categories &  $\EE_3$-algebras \\ 
\hline 
 $\dots$ & $``$ & $``$ & $\dots$ \\ 
\hline 
\end{tabular} 
\end{center}
\index{$\EE_n$-algebra!examples}

What we see is that the $\Vect_{\kk}$ could not see the difference between $\EE_n$-algebras for $n \geq 2$. This is because the morphism spaces of $\Vect_{\kk}$ are  discrete, hence there is ``no room'' for the interesting homotopies to show up. Likewise, $\Cat^{\times}$ sees the difference between $\EE_1$-algebras and $\EE_2$-algebras, but because its only ``higher'' morphisms are given by natural isomorphisms, there is no room to see the higher-dimensional homotopies that we need to detect $\EE_3$-structures. That is, any $\EE_n$-algebra in categories for $n \geq 3$ is a symmetric monoidal category.

Finally, it turns out that we do not have other names for an $\EE_n$-algebra in cochain complexes other than ``$\EE_n$-algebra.''

We also stated:

\begin{definition}
\index{factorization homology}
Let $\cC^\otimes$  be a symmetric monoidal $\infty$-category which has all sifted
colimits and let $A$ be an $\EE_n$-algebra in $\cC^\otimes$. 
Factorization homology is the left Kan extension
\begin{equation}\nonumber
\begin{tikzcd}
\disk_{n,\fr} \ar[r,"A"] \ar[d, hookrightarrow]& \cC \\ 
\mfld_{n,\fr} \ar[ru, dotted, "\int A" , swap]& 
\end{tikzcd} \ \ .
\end{equation}
We write $\int_X A$ for factorization homology evaluated on a framed manifold $X$. 
\end{definition}

\begin{remark}
Let us collect some remarks.

\begin{itemize}
    \item Factorization homology can also be defined for manifolds whose tangent bundles are equipped with a $G$-reduction; the necessary algebraic input there is---informally---an $\EE_n$-algebra with $G$-action.

	\item Factorization homology is also functorial in the $A$ variable. That is, given a map of $\EE_n$-algebras $A \to B$, we also have induced maps $\int_X A \to \int_X B$ for any framed manifold $X$. In this way, you can think of factorization homology $\int_X A$ as not only an invariant of the framed manifold $X$, but also 
an invariant of the $\EE_n$-algebra $A$. For example, if we take $\cC^{\tensor} = \chain_k^{\tensor}$, and if we know that $\int_X A$ is not quasi-isomorphic to $\int_X B$, then we can conclude that $A$ and $B$ are not equivalent $\EE_n$-algebras. 
\end{itemize}
\end{remark}

The tool which makes factorization homology computable is tensor excision:

\begin{theorem}[$\otimes$-excision]
\index{excision!$\tensor$-excision}
Assume that $\cC^\otimes $ is a symmetric monoidal $\infty$-category with all
sifted colimits, and assume that $\otimes$ preserves sifted colimits in each variable. Fix a framed manifold $X$ and a decomposition $X=X_0 \bigcup_{W\times \R} X_1$.

Then $\int_X A$ can be computed as the relative tensor product 
	\eqnn
	\int_X A = \int_{X_0} A \bigotimes_{\int_{W\times \R}A} \int_{X_1} A \ \ . 
	\eqnd
\end{theorem}

\begin{remark}
What $\tensor$-excision really demonstrates is that factorization homology is computable if you can determine the $E_1$-algebra associated to $W \times \RR$. The easiest class of algebras with which one can do this are, informally, those algebras constructed in a way expressible using Dunn additivity (Example~\ref{example.dunn-additivity}). In contrast, formality of the $E_n$-operad in characteristic zero~\cite{tamarkin-2003-deligne-conjecture-published, fresse-willwacher} tells us that $E_n$-algebras can be constructed out of commutative cdgas with a degree $(1-n)$ bracket. This description leaves opaque any compatibility with Dunn additivity, and as such, it is not at all trivial to apply $\tensor$-excision to such examples.
\end{remark}

\clearpage
\section{Cobordisms and higher categories}
\index{cobordism}
\index{cobordism!oriented}
We switch topics a bit to work toward the notion of a topological field theory.

$\otimes$-excision shows that factorization homology is a `local-to-global' 
invariant. Thus the invariant can be computed from a decomposition of $X$, and importantly, the invariant is 
insensitive to the way in which we decompose $X$. 
Our next goal is to capture this geometric property more categorically. 

\begin{warning-numbered}
A lot what we say in this section regarding the (higher) category of cobordisms is informal. For a slightly more rigorous treatment, see Section~\ref{section. cobordism category} below.
\end{warning-numbered}

\begin{definition}[Informal]
\index{$\Cob_{n,n-1}$}
Let
	\eqnn
	\Cob_{n,n-1}
	\eqnd
be the category whose objects are compact $n-1$ dimensional, oriented manifolds. 
Given $W_0$ and $W_1$ objects of $\Cob_{n,n-1}$,
an element of the set $\hom(W_0,W_1)$ is a compact $n$-dimensional 
manifold $X$ together with an oriented
identification 
	\eqnn
	\partial X \cong W_0 \coprod W_1^{\op}
	\eqnd
where $W_1^{\op}$ denotes $W_1$ equipped with the opposite orientation. For technical reasons, we will demand that this identification extends to define a collar of $\partial X$, 
	\eqnn
	\partial X \times [0,1)
	\cong (W_0 \coprod W_1^{\op}) \times [0,1)
	\eqnd
where we also identify $\partial X \times [0,1)$ with a small neighborhood of $\del X \subset X$.

Composition is given by gluing together manifolds along their boundary:

    \small
    \begin{center}
    \begin{overpic}[scale=0.28]{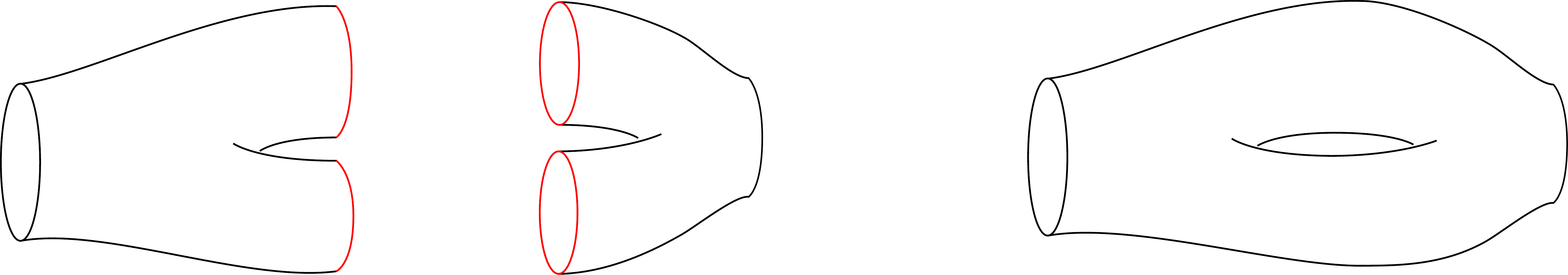}
    \put(27,7){$\circ$} 
    \put(55,7){\huge$=$}
    \put(-3,0){$W_2$}
    \put(23,0){$W_1$}
    \put(31,0){$W_1$}
    \put(48,0){$W_0$}
    \put(63,0){$W_2$}
    \put(100,0){$W_0$}
    \put(82,11){$ X' \circ X$}
    \put(43,8){$  X$}
    \put(10,8){$ X' $}
    \end{overpic}
    \end{center}
    \normalsize

Finally, we note that the identity morphism of any $W$ is given by the cylinder $W \times [0,1]$.
\end{definition} 

\begin{remark}
Note that in the figure, we have read the cobordism ``from right to left.'' This is to make the usual notation for composition more easily compatible with our pictures.
\end{remark}

\begin{remark}
Usually categories are named after their objects. The category $\Cob_{n,n-1}$ is an exception, named after its morphisms. 
\end{remark}

\begin{definition}
The pair 
	\eqnn
	(X,\partial X \cong W_0 \coprod W_1^{op})
	\eqnd
is called a {\em cobordism} from $W_0$ to $W_1$.
\end{definition}

The definition of $\Cob_{n,n-1}$ achieves our goal of articulating local-to-global invariants in the following way: Every functor 
	\eqn\label{eqn. Z n dim not extended}
	Z: \Cob_{n,n-1} \to \cC
	\eqnd
satisfies a decomposition property. 

To see this, fix an element 
	\eqnn
	X\in \hom_{\Cob_{n,n-1}}(\emptyset,\emptyset),
	\eqnd
i.e., a smooth oriented compact manifold $X$ without boundary. The evaluation of 
$Z$ on $X$ is an element 
	\eqnn
	Z(X)\in \hom_\cC(Z(\emptyset),Z(\emptyset)).
	\eqnd
Let $X=X_0\coprod_{W}X_1$ be a decomposition of $X$ into two cobordisms
$X_0\colon \emptyset \longrightarrow W$ and $X_1\colon W \longrightarrow \emptyset$.
The functoriality of $Z$ implies $Z(X)=Z(X_1)\circ Z(X_0)$, i.e. we can compute 
$Z(X)$ from the decomposition of $X$. Moreover, every decomposition of $X$ gives the same result, namely $Z(X)$.

\begin{definition}
\index{TFT}
Let $\cC$ be a symmetric monoidal $\infty$-category.
An $n$-dimensional (oriented) topological field theory, or {\em TFT}, is a symmetric monoidal functor
	\eqnn
	Z: \Cob_{n,n-1}^{\coprod} \to \cC^\tensor.
	\eqnd
\end{definition}

\begin{remark}
Classically, for instance in ideas of Atiyah and Segal, one would take $\cC^{\tensor}$ to be the category of vector spaces with symmetric monoidal structure given by $\tensor_{\kk}$.
\end{remark}

This is tantalizing. So we must naturally be led to ask:
Can we classify such functors?

The answer for high values of $n$ is a resounding no. For example, to classify such functors, it may help to have a good handle on the objects of $\Cob_{n,n-1}$; but classifying all closed $(n-1)$-dimensional manifolds is near impossible. We cannot do it when $(n-1)$=4. 

\begin{remark}
Of course, one need not classify all objects to define a functor. And the difficulty of classification of manifolds is one reason one would {\em seek} a functor as above. The real reason I am moaning on about this is to get to the fully extended cobordism category below.

Let me emphasize that the difficulty of classifying manifolds is not a convincing reason for one to abandon the search for functors out of $\cob_{n-1,n}$ for values such as $n=2,3,$ or $4$. For toy examples in representation theory, the reader can look up Dijkgraaf-Witten theory.  Seiberg-Witten invariants or Chern-Simons theory come also ``close'' to defining functors out of $\Cob_{n,n-1}$ in our sense, though I won't go into it here.
\end{remark}

But what if we can further decompose $n-1$ dimensional manifolds? For example, what if instead of classifying $(n-1)$-dimensional manifolds, we allow our ``category'' to decompose these further into $(n-2)$-dimensional manifolds, and so forth and so forth? Then to understand our objects, we would need only understand 0-dimensional manifolds. This we can do.

I put ``category'' in quotes above because it turns out the natural algebraic structure to look for is that of an  {\em $n$-category}, of which a category is the case of $n=1$. Let me ease us into this notion by discussing the example of $n=2$.

\begin{definition}[Informal]
We denote by $\Cob_{0,1,2}$, or 
	\eqnn
	\Cob_2
	\eqnd
for short, the ``category'' whose data are given by
\begin{itemize}
\item Objects: Oriented 0-dimensional compact manifolds. (Which is to say, an object is a possibly empty collection of points, each point equipped with a plus or a minus.)
\item Given two objects $W_0, W_1$, we declare
	\eqnn
	\hom(W_0,W_1)= \{ \text{cobordisms }X\colon W_0\longrightarrow W_1\}
	\eqnd
to be the collection of oriented cobordisms from $W_0$ to $W_1$. Then,
\item Given any two cobordisms $X$ and $Y$ having the same source and target, we declare an element of
	\eqnn
	\hom(X,Y)
	\eqnd
to be a compact 2-dimensional oriented manifold $Q$ with {\em corners}, equipped with an identification
	\eqnn
	\partial Q \cong \left(X\coprod Y^{\op}\right) \bigcup \left(W_0\times [0,1] \coprod W_1^{\op} \times [0,1]  \right).
	\eqnd
Informally, $Q$ is a cobordism between the cobordisms $X$ and $Y$. The $\bigcup$ above is a gluing along the subspace
	$
	\left( W_0 \times \{0,1\}\right)
	\coprod 
	\left( W_1 \times \{0,1\}\right).
	$
\end{itemize} 
\end{definition}

\begin{example}
There's lots to unpack here. Let's first begin with an example of a $Q$, the saddle:  
	\begin{center}  
    \begin{overpic}[scale=0.6]{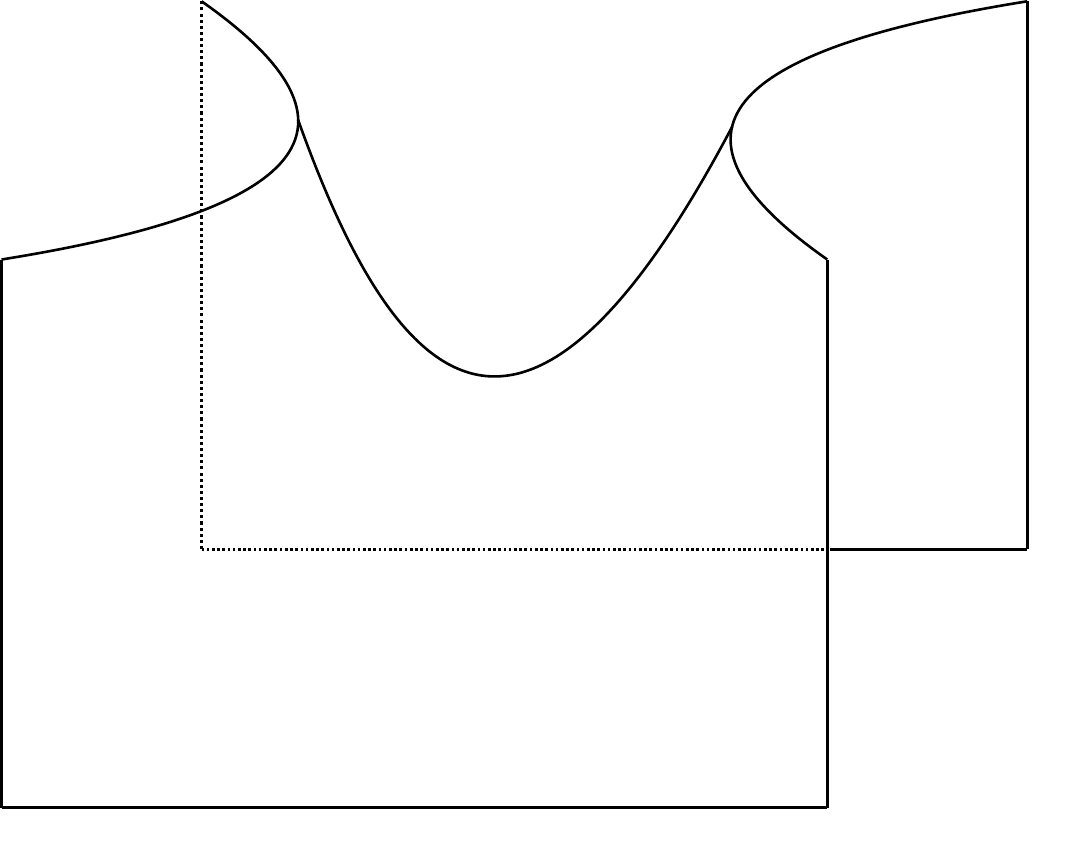}
    \end{overpic}
    \end{center}
In this example, I have taken 
	\eqnn
	W_0 = W_1 = \ast^+ \coprod \ast^-
	\eqnd
to equal two disjoint points; one with positive orientation and one with negative orientation. $X$ is a morphism given by a horseshoe and a co-horseshoe---it is a disconnected, oriented 1-manifold with boundary given by $W_0 \coprod W_1^{\op}$. Finally, I have taken $Y$ to be the product $W_0 \times [0,1]$, the identity morphism of $W_0$.
\end{example}

\begin{remark}[Cobordisms, classically]
Cobordisms have long been a way to prescribe ways to change the differential topology of a smooth manifold. You may have heard smooth topologists talking about attaching handles to change a manifold, and the above cobordism is an example of attaching a 1-handle to change a 2-manifold. Indeed, the (co)-horseshoes are examples of attaching 0- and 1-handles to change 1-manifolds. 
\end{remark}

\begin{remark}[Composing the 2-manifolds]\label{remark. composing 2 manifolds}
We knew from before that we could composes the 1-manifolds, which are cobordisms in the previous sense---we glue $X$'s along a common $W$. We now note that given two $Q$'s (i.e., two 2-dimensional manifolds with corners) we can glue them to each other in possibly two distinct senses: 
	\enum
	\item If the target of $Q$ is a cobordism $Y$, and if the domain of $Q'$ is that same cobordism, we can glue $Q$ and $Q'$ along $Y$. (We can compose along 1-manifolds.)
	\item If $X$ and $Y$ are morphisms from $W_0$ to $W_1$, and if $X',Y'$ are morphisms from $W_1$ to $W_2$, we can glue $Q$ and $Q'$ to each other along $W_1 \times [0,1]$. (We can compose along ``constant'' 0-manifolds.)
	\enumd
\end{remark}
	
All this data, satisfying conditions that are natural but cumbersome to articulate, is some form of a {\em 2-category}. The $W$s are called objects, the $X$'s and $Y$'s are morphisms, or 1-morphisms, and each $Q$ is  a {\em 2-morphism}. 

\begin{remark}[The ``$n$'' in $n$-morphism.]
We referred to $Q$ as a 2-morphism. This ``2'' can stand for many things in your mind: The dimension of the manifold, the ``height'' above the notion of being an object (an object is a 0-morphism and a usual morphism is a 1-morphism, for example), or the number of distinct senses in which you can compose $Q$'s (as indicated in Remark~\ref{remark. composing 2 manifolds}).
\end{remark}

Now that we have somewhat given intuition for $n=2$, let us state:

\begin{definition}[Informal]
\index{$\cob_n$}
We let $\cob_n$ denote the $n$-category whose objects are oriented 0-manifolds, whose morphisms are oriented 1-manifolds with boundary (expressed as cobordisms between 0-dimensional manifolds), and whose $k$-morphisms are $k$-manifolds with corners (expressed as cobordisms between $(k-1)$-dimensional manifolds).
\end{definition}

If we could formalize the notion of $n$-category (I have only sketched an idea here), then it's clear that the following algebraic gadget can encode ``local-to-global'' invariants: Functors
	\eqnn
	\cob_n \to \cC
	\eqnd
of $n$-categories. Whatever we mean by functor, all compositions should be respected; this is the sense in which a functor from $\cob_n$ defines local-to-global invariants.

\begin{example}
Let us be concrete in the case $n=2$. We have defined a 2-category $\cob_2$ whose 2-categorical structure encodes the notion of decomposing 2-dimensional manifolds: Fixing a 2-morphism $Q$ between the empty cobordism and itself (i.e., a compact 2-dimensional manifold), a functor $Z$ assigns a 2-morphism of $\cC$, and this is an invariant of the manifold. We can decompose this 2-dimensional manifold by expressing it as various compositions of other 2-morphisms (i.e., of other 2-dimensional cobordisms), which in turn may be be glued along 1-morphisms that themselves are expressed as compositions of various 1-morphisms. And regardless of how we decompose a 2-manifold---regardless of how we factor the 2-morphism---the original invariant $Z(Q)$ can be recovered by composing along the corresponding factorizations in $\cC$.
\end{example}

And, in the spirit that once we have understood connected manifolds, we have understood the disconnected manifolds, we may as well seek functors that are symmetric monoidal. Finally, in the mean spirit of your homotopy theorist lecturer, we may as well treat everything in sight as an $(\infty,n)$-category. Indeed, $\cob_n$ may be constructed as an $(\infty,n)$-category, whatever that is.\footnote{This is not just for the sake of using the symbol $\infty$. See Section~\ref{section. cobordism category}.} 

\begin{definition}
\index{TFT}
\index{TFT!fully extended}
\index{$\cob_n^{\coprod}$}
\index{$(\Cob_{n}^{\fr})^{\coprod}$}
Let $\cC^{\tensor}$ be a symmetric monoidal $(\infty,n)$-category. A {\em fully extended, oriented $n$-dimensional topological field theory with values in $\cC$}, or {\em $n$-dimensional oriented TFT} for short, is a symmetric monoidal functor
	\eqnn
	Z: \cob_n^{\coprod} \to \cC^{\tensor}.
	\eqnd
One can define a version of $\cob_n$ in which every manifold in sight is appropriately framed. (See Definition~\ref{defn. cobordism framing}.) We denote it by $\Cob_{n}^{\fr}$. A symmetric monoidal functor
	\eqnn
	Z: (\Cob_{n}^{\fr})^{\coprod} \to \cC^{\tensor}
	\eqnd
is called a {\em framed} $n$-dimensional TFT (with values in $\cC$).
\end{definition}

\begin{remark}
The ``fully extended'' refers to the fact that we have ``extended'' a functor in the sense of~\eqref{eqn. Z n dim not extended} to capture manifolds of lowest possible dimension (zero).
\end{remark}

Now, let us ask two natural questions:

\begin{question-numbered}[Question One]\label{question: n-cats and cob_n}
Can we make the definition of $\Cob_n$, and of $(\infty,n)$-categories, precise?
\end{question-numbered}

\begin{question-numbered}[Question Two]\label{question: classify tfts}
Can we classify fully extended topological field theories, i.e. symmetric monoidal functors $Z\colon \Cob_n \longrightarrow \cC^\otimes$ where $\cC$ is a symmetric monoidal $(\infty,n)$-category?
\end{question-numbered}

Both have ``yes'' as an answer. 

\begin{remark}
\index{$(\infty,n)$-category}
We will touch upon Question One (\ref{question: n-cats and cob_n}) more in Section~\ref{section. cobordism category} below. But let us say for now that an $(\infty,n)$-category consists of a collection of objects, and for every pair of objects $X,Y$, the collection of morphisms $\hom(X,Y)$ can be considered an $(\infty,n-1)$-category. This gives some feel of the notion by induction, beginning with the case $n=0$, where an $(\infty,0)$-category can be thought of as the same thing as a topological space. 

Put another way, an $n$-category (or $(\infty,n)$-category) is roughly meant to be a category enriched in $(n-1)$-categories (or in $(\infty,n-1)$-categories). In particular, given a pair of $k$-morphisms $f, g$, the collection $\hom(f,g)$ forms an $(\infty,n-k)$-category.
\end{remark}

We for now focus on Question Two (\ref{question: classify tfts}).

\clearpage
\section{The cobordism hypothesis}
\index{$\cC^{f.d.}$}
\index{fully dualizable}
The answer to Question Two  (\ref{question: classify tfts}) is as follows:

\begin{hypothesis}[Cobordism Hypothesis]
\index{cobordism hypothesis}
Fix a symmetric monoidal $(\infty,n)$-category $\cC^\otimes$. Then there exists an equivalence between 
	\eqnn
	\Fun^\otimes ((\Cob_{n}^{\fr})^{\coprod},\cC^{\tensor})
	\eqnd
(the $\infty$-category of symmetric monoidal functors from $(\Cob_{n}^{\fr})^{\coprod}$ to $\cC^{\tensor}$) and
	\eqnn
	\cC^{f.d.}
	\eqnd
the space of {\em fully dualizable objects} of $\cC^{\tensor}$. This equivalence is given by the evaluation map
	\eqnn
	\Fun^\otimes ((\Cob_{n}^{\fr})^{\coprod},\cC^{\tensor}) \to
	\cC^{f.d.}
	\qquad
	Z \mapsto Z(\ast^+)
	\eqnd
at the positively framed point.
\end{hypothesis}

This was an idea first proposed by Baez and Dolan~\cite{baezdolan}, and for this reason is also referred to as the Baez-Dolan cobordism hypothesis.

\begin{remark}
The reader will note that the above is stated as a hypothesis, and not a theorem; this is because the statement is widely expected to be true. Lurie has outlined a proof in~\cite{lurie-tft} and there are certainly individuals working toward a proof. (One of the announced proof methods does not follow Lurie's outline and instead uses techniques inspired by factorization homology~\cite{ayala-francis-tft}.)
\end{remark}

\begin{remark}\label{remark. TFT maps are equivalences}
We have stated that there is an equivalence of {\em spaces}. Indeed, buried in the statement is the following claim: If we have any symmetric monoidal natural transformation $Z \to Z'$, then this must actually be a natural equivalence.

(Following the philosophy of $\infty$-categories, any higher category in which all morphisms are invertible is equivalent to the $\infty$-category obtained from a space. See Remark~\ref{remark. spaces as groupoids}.)
\end{remark}

So now we must understand the notion of full dualizability. We begin in dimension one.

\clearpage
\section{The cobordism hypothesis in dimension 1, for vector spaces}

Recall that I am using the model of $\infty$-category to speak of $(\infty,1)$-categories (Convention~\ref{convention. oo-cats}). So, at least in dimension 1, there is no mystery about what the target category of a TFT is. 

Under disjoint union the objects of $(\Cob_{1}^{\fr})^{\coprod}$ are generated by 
the empty set $\emptyset$, a positive point $*^+$ and a negative 
point $*^-$. 
We give names to the values of a topological field theory on these $0$-dimensional manifolds 
\begin{align}
    Z\colon  (\Cob_{1}^{\fr})^{\coprod} &\longrightarrow \Vect^{\tensor_{\kk}} \nonumber\\
    *^+ & \longmapsto V \nonumber\\
    *^- & \longmapsto W \nonumber\\
    \emptyset & \longmapsto \kk \nonumber .
\end{align} 
Now, the identity morphism $\ast^+ \times [0,1]$ can be factored as indicated on the lefthand column of the following figure: 

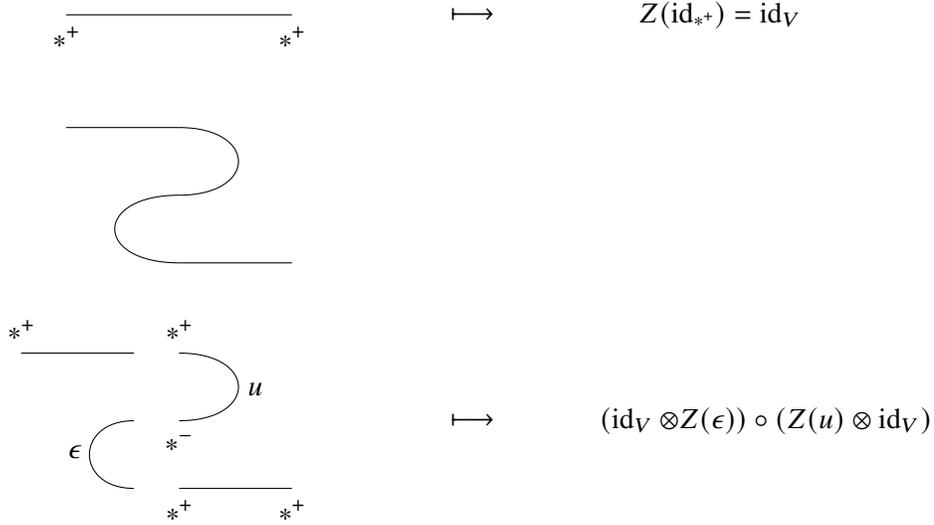
\begin{figure}
\index{Zorro's Lemma}
\caption{	The proof of Zorro's Lemma.}\label{image. zorros lemma} 
    \begin{equation}\nonumber
    \begin{tikzpicture}[scale=1.2, none/.style={circle, inner sep=0pt,minimum size=0mm},
            poin/.style={circle, inner sep=0pt,minimum size=0mm}]
    \tikzstyle{new edge style 0}=[->]
    		\node [style=none] (0) at (-9, 3.25) {};
    		\node [style=none] (1) at (-6.5, 3.25) {};
    		\node [style=none] (2) at (-9, 2) {};
    		\node [style=none] (3) at (-7.75, 2) {};
    		\node [style=none] (4) at (-7.75, 1.25) {};
    		\node [style=none] (5) at (-7.75, 0.5) {};
    		\node [style=none] (6) at (-6.5, 0.5) {};
    		\node [style=none] (7) at (-9.5, -0.5) {};
    		\node [style=none] (8) at (-8.25, -0.5) {};
    		\node [style=none] (9) at (-8.25, -1.25) {};
    		\node [style=none] (10) at (-8.25, -2) {};
    		\node [style=none] (12) at (-7.75, -0.5) {};
    		\node [style=none] (13) at (-7.75, -1.25) {};
    		\node [style=none] (14) at (-7.75, -2) {};
    		\node [style=none] (15) at (-6.5, -2) {};
    		\node [style=none] (16) at (-7.25, 4) {};
    		\node [style=none] (17) at (-8, 4) {};
    		\node [style=none] (18) at (-6.5, 3) {$*^+$};
    		\node [style=none] (19) at (-9, 3) {$*^+$};
    		\node [style=none] (20) at (-6.5, -2.25) {$*^+$};
    		\node [style=none] (21) at (-7.75, -2.25) {$*^+$};
    		\node [style=none] (22) at (-7.75, -1.5) {$*^-$};
    		\node [style=none] (23) at (-7.75, -0.25) {$*^+$};
    		\node [style=none] (24) at (-9.5, -0.25) {$*^+$};
    		\node [style=none] (25) at (-8.9, -1.6) {$\epsilon$};
    		\node [style=none] (26) at (-6.9, -0.9) {$u$};
    		\node [style=none] (27) at (-4.5, 3.25) {};
    		\node [style=none] (28) at (-4.5, 3.25) {$\longmapsto$};
    		\node [style=none] (29) at (-1.75, 3.25) {$Z(\id_{*^+})=\id_V$};
    		\node [style=none] (30) at (-4.5, 1.25) {};
    		\node [style=none] (33) at (-4.5, -1.25) {};
    		\node [style=none] (34) at (-4.5, -1.25) {$\longmapsto$};
    		\node [style=none] (35) at (-1.25, -1.25) {$(\id_V\otimes Z(\epsilon))\circ (Z(u)\otimes \id_V)$};
    		\draw (0.center) to (1.center);
    		\draw (2.center) to (3.center);
    		\draw [in=0, out=0, looseness=3.00] (3.center) to (4.center);
    		\draw [bend right=90, looseness=3.25] (4.center) to (5.center);
    		\draw (5.center) to (6.center);
    		\draw (7.center) to (8.center);
    		\draw [bend right=90, looseness=2.25] (9.center) to (10.center);
    		\draw [in=0, out=0, looseness=3.00] (12.center) to (13.center);
    		\draw (14.center) to (15.center);
    \end{tikzpicture}
    \end{equation}
\end{figure}

\begin{remark}
Let us explain Figure~\ref{image. zorros lemma}. We have drawn a ``horseshoe'' cobordism 
	\eqnn
	\left(\ast^+ \coprod \ast^- \right) \leftarrow \emptyset : u
	\eqnd 
and  a cohorseshoe cobordism 
	\eqnn
	\emptyset \leftarrow \left(\ast^- \coprod \ast^+ \right) : \epsilon,
	\eqnd
where as before, are we reading our cobordisms as propagating from right to left. Thus, the bottom-left composition of the image is read as
	\eqnn
	(\id_{\ast^+} \coprod \epsilon)
	\circ
	(u \coprod \id_{\ast^+}).
	\eqnd
\end{remark}

\begin{notation}
Out of sloth, and to save notational clutter, we will denote $Z(u)$ by $u$ as well. Likewise for $\epsilon$. We hope that from context, it will be clear in which category ($\Cob_n$ or $\cC^{\tensor}$) these morphisms live.
\end{notation}

From the decomposition we can directly deduce the following lemma.

\begin{lemma}[Zorro's Lemma]
\index{Zorro's Lemma}
    Let $Z\colon \Cob_{1}\longrightarrow \cC$ be a topological field theory and let $V = Z(\ast^+)$.
    Then 
    	\eqn\label{Eq: Zorro}
        	\id_V = (\id_V\otimes \epsilon)\circ (u\otimes \id_V).
		\eqnd
\end{lemma}

\begin{proof}
Consider the ``snake-like'' factorization in the bottom-left of Figure~\ref{image. zorros lemma}. This composite is equivalent to the identity cobordism as a cobordism, hence applying $Z$ to the relation
	\eqnn
	\id_{\ast^+}
	=
	(\id_{\ast^+} \coprod \epsilon)
	\circ
	(u \coprod \id_{\ast^+})	
	\eqnd
in $\Cob_1$, the result follows.
\end{proof}

\begin{remark}
Zorro is not a mathematician, but a masked avenger.
\end{remark}

This actually puts a strong restriction on the values that $Z(\ast^+)$ can take:

\begin{lemma}
$V = Z(\ast^+)$ must be finite dimensional.
\end{lemma}

\begin{proof}
As before let $W = Z(\ast^{-})$. We have linear maps 
    \begin{equation}\nonumber
        \epsilon \colon  V\otimes W \longrightarrow \kk,
        \qquad
        v\otimes w \longmapsto \langle v, w \rangle
    \end{equation}
and 
    \begin{equation}\nonumber
        \epsilon\colon \kk \longrightarrow W\otimes V, \qquad
        1  \longmapsto \sum_{i,j} a_{i,j} w_i\otimes v_j \ \ ,
    \end{equation}
where we note that the summation is {\em finite} (by the definition of the tensor product of vector spaces). 
The composition of linear maps on the right side of \eqref{Eq: Zorro} is 
    \begin{equation}\nonumber
        v \longmapsto v\otimes \left( \sum_{i,j}a_{i,j} w_i\otimes v_j \right) \longmapsto \sum_{i,j} a_{i,j}
        \langle v, w_i \rangle v_j \ \ .
    \end{equation}
We get from \eqref{Eq: Zorro} that $V$ must be spanned by finitely many vectors $v_j$. 
\end{proof}

\begin{remark}
This puts us well on the way to verifying the cobordism hypothesis in dimension 1, where there is again no difference between the framing and orientation conditions. Also, it will follow from the definition that a fully dualizable object of $\Vect_{\kk}^{\tensor_\kk}$ will be a finite-dimensional $\kk$-vector space.

What remains to prove is that $Z(\ast^+)$ determines all of $Z$, and that any natural transformation of symmetric monoidal functors $Z \to Z'$ actually induces an {\em isomorphism} of vector spaces $Z(\ast^+) \to Z'(\ast^+)$.  (This is what we mean by ``space'' of objects; an $\infty$-category where all morphisms are invertible).

We will leave this as an exercise to the reader (Exercise~\ref{exercise. cob hyp in dim 1}.)
\end{remark}

This concludes our example for $n=1$ and $\cC^{\tensor} = \Vect_{\kk}^{\tensor_{\kk}}$.

\clearpage
\section{Full dualizability}

From the previous example we can extract the following slogan: full dualizability is a generalization
of ``being finite dimensional" to higher categories. 

\begin{definition}\label{defn. 1-dualizable}
\index{fully dualizable}
Fix a monoidal $\infty$-category $\cC$. An object $V\in \cC$ is {\em dualizable}, or {\em 1-dualizable} if 
\begin{itemize}
\item $V$ admits a left dual. That is, there exists an object $V_L$ and morphisms $u\colon \One_\cC \longrightarrow V\otimes V_L$ 
and $\epsilon \colon V_L\otimes V \longrightarrow \One_\cC$ satisfying the formula
	\eqn\label{eqn. 1-dualizable object}
        	\id_V = (\id_V\otimes \epsilon)\circ (u\otimes \id_V).
	\eqnd
	(The reader should compare this to \eqref{Eq: Zorro}.)
\item $V$ admits a right dual. That is, there exists an object $V_R$ and morphisms $u_R\colon \One_\cC \longrightarrow V_R\otimes V$ 
and $\epsilon_R \colon V\otimes V_R \longrightarrow \One_\cC$ satisfying an appropriate analogue 
of~\eqref{eqn. 1-dualizable object}.
\end{itemize}
\end{definition}

\begin{remark}\label{remark. units are dualizable and left-right duals}
The monoidal unit $\One_\cC$ is always dualizable.
When $\cC$ is symmetric monoidal then $V_L\cong V_R$ and the existence of a left/right dual implies
that $V$ is dualizable. 

Thus, from hereon (because $\cC^{\tensor}$ will always be symmetric monoidal for us) we will make no distinction between having right and left duals.
\end{remark}

\begin{example}
Being dualizable is a delicate balance between the symmetric monoidal structure chosen, and properties of the objects and morphisms in the category. Here is a list of dualizable objects in various categories:

\begin{center}
    \begin{tabular}{|c|c|}
    \hline 
    $\cC^{\tensor}$ & dualizable objects \\ 
    \hline 
    $\Set^\times$ & The terminal set (a point) \\ 
    \hline 
    $\Cat^\times$ & The terminal category (a point) \\ 
    \hline 
    $\Vect^\oplus$ & 0 \\ 
    \hline 
    $\Vect^\tensor$ & Finite-dimensional vector spaces \\ 
    \hline 
    $\ComAlg_\kk^{\tensor_{\kk}}$ & $\kk$ \\
    \hline
    $\AlgBimod$ & all algebras \\ 
    \hline 
    \end{tabular}  
\end{center}

Here $\AlgBimod$ is the $\infty$-category with objects algebras, morphisms 
bimodules and 2-morphisms bimodule equivalences. The composition is given by the 
relative tensor product.
\end{example}

Now we turn our attention to the definition of higher notions of dualizability; this relies on defining when a {\em morphism} (not an object) is dualizable.

\begin{example}\label{example. adjunctions}
As motivation we consider the classical notion of adjunctions. This illustrates when 2 functors are dualizable.

Let $\cD$ and $\cE$ be categories together
with a pair of functors
    \begin{equation}\nonumber
	    L\colon \cD \rightleftarrows \cE \colon R \ \ .
    \end{equation}
Recall that $L$ is left adjoint to $R$ (and $R$ is right adjoint to $L$) if and only if there exists a pair of natural transformations
	\eqnn
	u\colon \id_\cD \longrightarrow R\circ L
	\qquad
	\epsilon \colon L\circ R \longrightarrow \id_\cE
	\eqnd
such that the compositions
	\eqnn
		\xymatrix{
		L \ar[rr]^-{\id_L \circ u}
			&& LRL \ar[rr]^-{ \epsilon \circ \id_L }
			&&L
		}
	\eqnd
and
	\eqnn
		\xymatrix{
		R \ar[rr]^-{u \circ \id_R}
		&& RLR \ar[rr]^-{\id_R \circ \epsilon}
		&& R
		}
	\eqnd
agree with the identity natural transformations (from $L$ to itself, and from $R$ to itself, respectively). 
\end{example}

These are the classic examples of dualizable morphisms in an $(\infty,2)$-category:

\begin{definition}\label{defn. 2-dualizable}
Let $\cC$ be an $(\infty,2)$-category and fix two objects $D$ and $E$. Then a morphism $f: D \to E$ is {\em right dualizable} if there exists another morphism $R: E \to D$, along with two 2-morphisms
	\eqnn
	u: \id_{D} \to Rf,
	\qquad
	\epsilon: fR \to \id_{E}
	\eqnd
satisfying the adjointness conditions:
	\eqnn
	(\epsilon \circ \id_f) \circ (\id_f \circ u) \simeq \id_f,
	\qquad
	(\id_R \circ \epsilon) \circ (u \circ \id_R) \simeq \id_R.
	\eqnd
(That is, there {\em exists} an equivalence between the indicated compositions and the identity morphisms.)
We then call $f$ a {\em left dual} to $R$, and we call $R$ a {\em right dual} to $f$. Likewise, we say that $f$ is left dualizable if it admits a left dual. 

We say $f$ is {\em dualizable}, or {\em admits adjoints}, or is {\em 2-dualizable}, if it admits both left and right duals.
\end{definition}

One may begin to be bothered by the notation, and also bothered by the similarities apparent in the dualizability condition for objects (Definition~\ref{defn. 1-dualizable}), for morphisms (Definition~\ref{defn. 2-dualizable}), and the example of adjunctions in Example~\ref{example. adjunctions}.  Let us clarify the sense in which ``every formula that has appeared is actually the same formula.''

\begin{example}
The notion of an object being dualizable can be phrased as a case of when a morphism is 2-dualizable.

That is, let $\cC$ be a monoidal category. Denote by $B\cC$ the 2-category with one object $\star$ and endomorphisms given by
$\End(\star)=\cC$. The composition in $B\cC$ is given by the tensor product 
$\otimes \colon \cC\times \cC \longrightarrow \cC$.
The 1-morphisms in this 2-category correspond to the objects of $\cC$, and the 2-morphisms in $B\cC$ are the morphisms of $\cC$. 

Then an object of $\cC^{\tensor}$ is dualizable if and only if it is dualizable as a morphism in $B\cC$. 
\end{example}

\begin{example}
Let $\Cattwo$ be the 2-category of categories: Objects are categories, morphisms are functors, and 2-morphisms are natural transformations ({\em not} just natural isomorphisms, as in the case of $\Cat$). 

Then a morphism in $\Cattwo$ is dualizable if and only if it admits both a right and a left adjoints.
\end{example}

We would recommend that the reader only memorize the adjointness relations for functors; this is enough to recover all the relations.

Finally, the notion of 2-dualizability for morphisms is enough to define full dualizability:

\begin{definition}
A $k$-morphism $X\longrightarrow Y$ in an $(\infty,n)$-category is called {\em dualizable}, or {\em $(k+1)$-dualizable},
if it is 2-dualizable as a 1-morphism in $\hom(X,Y)^{(2)}$.  

Here, note that $\hom(X,Y)$ is naturally an $(\infty,n-k)$-category. The notation $\hom(X,Y)^{(2)}$ indicates the  $(\infty,2)$-category obtained by discarding non-invertible morphisms above degree 3.
\end{definition}

\begin{definition}[\cite{lurie-tft}]\label{defn. fully dualizable}
\index{fully dualizable}
An object $X$ inside an symmetric monoidal $(\infty,n)$-category $\cC^\otimes$ is {\em fully dualizable}
if and only if 
\begin{itemize}
\item $X$ is dualizable (this requires the existence of certain 1-morphisms).
\item These 1-morphisms, and their duals, and all their composites, are 2-dualizable (this requires the existence of certain 2-morphisms).
\item These 2-morphisms, and their duals, and all their composites, are in turn 3-dualizable,
\item And this pattern continues until we verify that the required $(n-1)$-morphisms are $n$-dualizable.
\end{itemize}
\end{definition}

\begin{remark}
Inductive definitions can seem opaque at times, but let me emphasize that the hardest part of understanding this definition is getting used to what an $n$-category, or $(\infty,n)$-category, is. 

Once one then remembers the adjointness relations (which are fairly quick to internalize, as it turns out), it is very much possible to verify whether or not an object is fully dualizable. 
\end{remark}

\begin{sadness}
We are already two hours overtime so we won't be able to give great examples in-depth; they'll remain as exercises. 
\end{sadness}

\clearpage
\section{The point is fully dualizable}
\index{fully dualizable}
Now that we have seen the definition of full dualizability (Definition~\ref{defn. fully dualizable}), let us put the cobordism hypothesis to the test in the simplest possible example: Since the identity functor $(\Cob_{n}^{\fr})^{\coprod}  \to (\Cob_{n}^{\fr})^{\coprod} $ is surely a TFT with values in $\cC^{\tensor} = (\Cob_{n}^{\fr})^{\coprod}$, can we verify that $\ast^+$ is a fully dualizable object in $(\Cob_{n}^{\fr})^{\coprod}$?

\begin{warning-numbered}
In what follows, we only illustrate that the point in $\Cob_{n}$ (the oriented case) is dualizable. We must be more careful in the framed case. See Section~\ref{section. framed cobordisms}. We are illustrating only in the oriented case mainly to give a feel for what kinds of pictures one must draw in this business.
\end{warning-numbered}

\begin{example}
Let us first see that $\ast^+ \in \Cob_{1}$ is a dualizable, or 1-dualizable, object. Indeed, the proof of Zorro's Lemma relies on this fact. (See Figure~\ref{image. zorros lemma}.) 

What we discovered there can be summarized as follows using our new vocabulary: The point with opposite orientation $\ast^-$ is both a left and right dual object to $\ast^+$, and the unit and counit maps exhibiting the adjunctions are the (co)horseshoes,
\begin{equation}\nonumber
        \begin{tikzpicture}[scale=0.7, none/.style={circle, inner sep=0pt,minimum size=0mm},
                poin/.style={circle, inner sep=0pt,minimum size=0mm}]
        \tikzstyle{new edge style 0}=[->]
        		\node [style=none] (2) at (-7.5, 2) {$\ast^+$};
        		\node [style=none] (0) at (-7, 2) {};
        		\node [style=none] (M) at (-4, 1) { = $u$};
        		\node [style=none] (1) at (-7, 0) {};
        		\node [style=none] (3) at (-7.5, 0) {$\ast^-$};
        		\draw [bend right=90, looseness=2.50] (1.center) to (0.center);
        \end{tikzpicture}
    \qquad
    \text{and}
    \qquad
        \begin{tikzpicture}[scale=0.7, none/.style={circle, inner sep=0pt,minimum size=0mm},
                poin/.style={circle, inner sep=0pt,minimum size=0mm}]
        \tikzstyle{new edge style 0}=[->]
        		\node [style=none] (0) at (-7, 2) {};
        		\node [style=none] (1) at (-7, 0) {};
        		\node [style=none] (M) at (-5, 1) {$=\epsilon$};
        		\node [style=none] (2) at (-6.5, 2) {$\ast^-$};
        		\node [style=none] (3) at (-6.5, 0) {$\ast^+$};
        		\draw [bend left=90, looseness=2.50] (1.center) to (0.center);
        \end{tikzpicture}.
\end{equation} 
We recall that our conventions are to read cobordisms from right to left, so $u$ is a map from $\emptyset$ to $\ast^+ \coprod \ast^-$ while $\epsilon$ is a map from $\ast^- \coprod \ast^+$ to $\emptyset$. 

\end{example}

\clearpage
\begin{example}\label{example. point is 2-dualizable}
Now let us see that $\ast^+$ is fully dualizable in $\Cob_{2}$. For this, we must verify that the 1-morphisms that arose in verifying that $\ast^+$ is 1-dualizable are 2-dualizable (by Definition~\ref{defn. fully dualizable}), and repeat this process for any dual 1-morphisms that arise, along with all possible composites. 

Let us check only that $u$ is right dualizable. (To finish the proof, one must also verify that $u$ is left dualizable, then verify that $\epsilon$ is 2-dualizable; then that the ``opposite orientation'' cobordisms $u^\vee$ and $\epsilon^\vee$ are also 2-dualizable. In the framed case, ``new'' cobordisms with distinct framings will appear---these all have underlying manifolds given by cohorseshoes and their composites, but with various framings. One must check for all these as well.) 

Let $u^\vee$ be $u$ with the opposite orientation; we read this as a horseshoe-shaped cobordism from $\ast^+ \coprod \ast^-$ to $\emptyset$. We claim this is the right adjoint to $u$. For example, the composite $u^\vee \circ u$ is a circle; and hence a unit for this adjoint pair must be some cobordism from the (identity cobordism of the) empty manifold to the circle. Figure~\ref{image. cap} shows an obvious candidate,
	\begin{figure}[h]
	\caption{The cap}
	\label{image. cap}
    \begin{center}
    \begin{overpic}[scale=0.6]{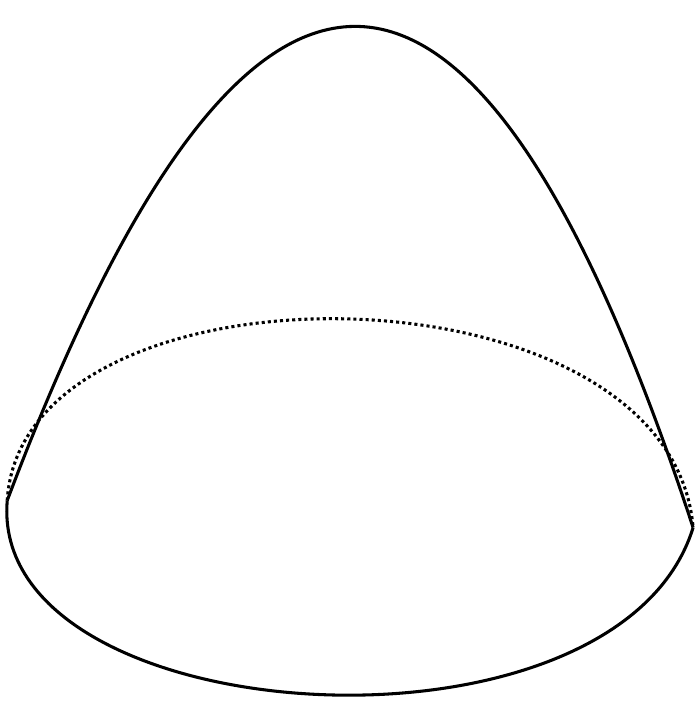}
    \end{overpic}
    \end{center}
    \end{figure}
    \normalsize
which one might call the ``cap.''

	\begin{figure}[h]
	\caption{The saddle}
	\label{image. saddle}
	\begin{center}  
    \begin{overpic}[scale=0.6]{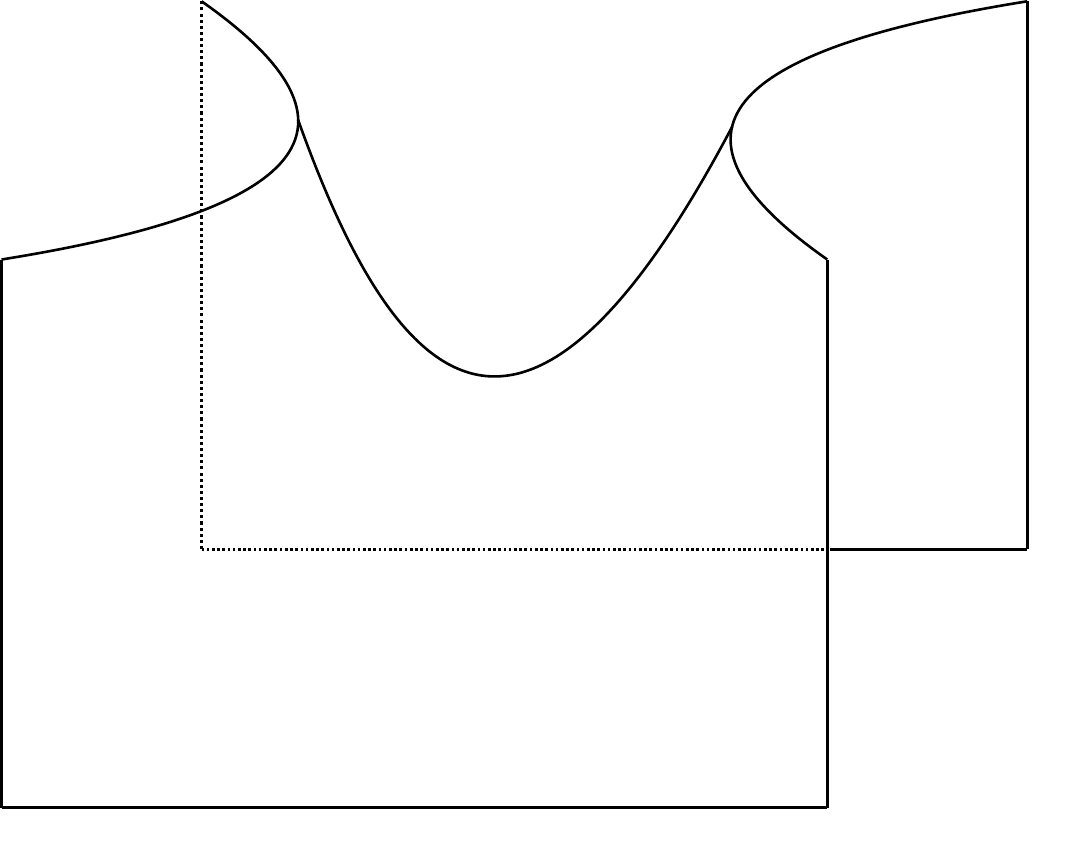}
    \end{overpic}
    \end{center}
    \end{figure}

The other composite $u \circ u^\vee$ is a disjoint union of a horseshoe and a cohorseshoe, giving a morphism from $\ast^+ \coprod \ast^-$ to itself. A counit for this adjoint pair must be a cobordism from this disjoint union to the identity cobordism of $\ast^+ \coprod \ast^-$. We have seen such a cobordism before, given by the saddle, as drawn in Figure~\ref{image. saddle}. 
    
	\begin{figure}
	\caption{The right dualizability of $u$:}
	\label{image. u is dualizable}	
        \begin{center}
        \begin{overpic}[scale=0.6]{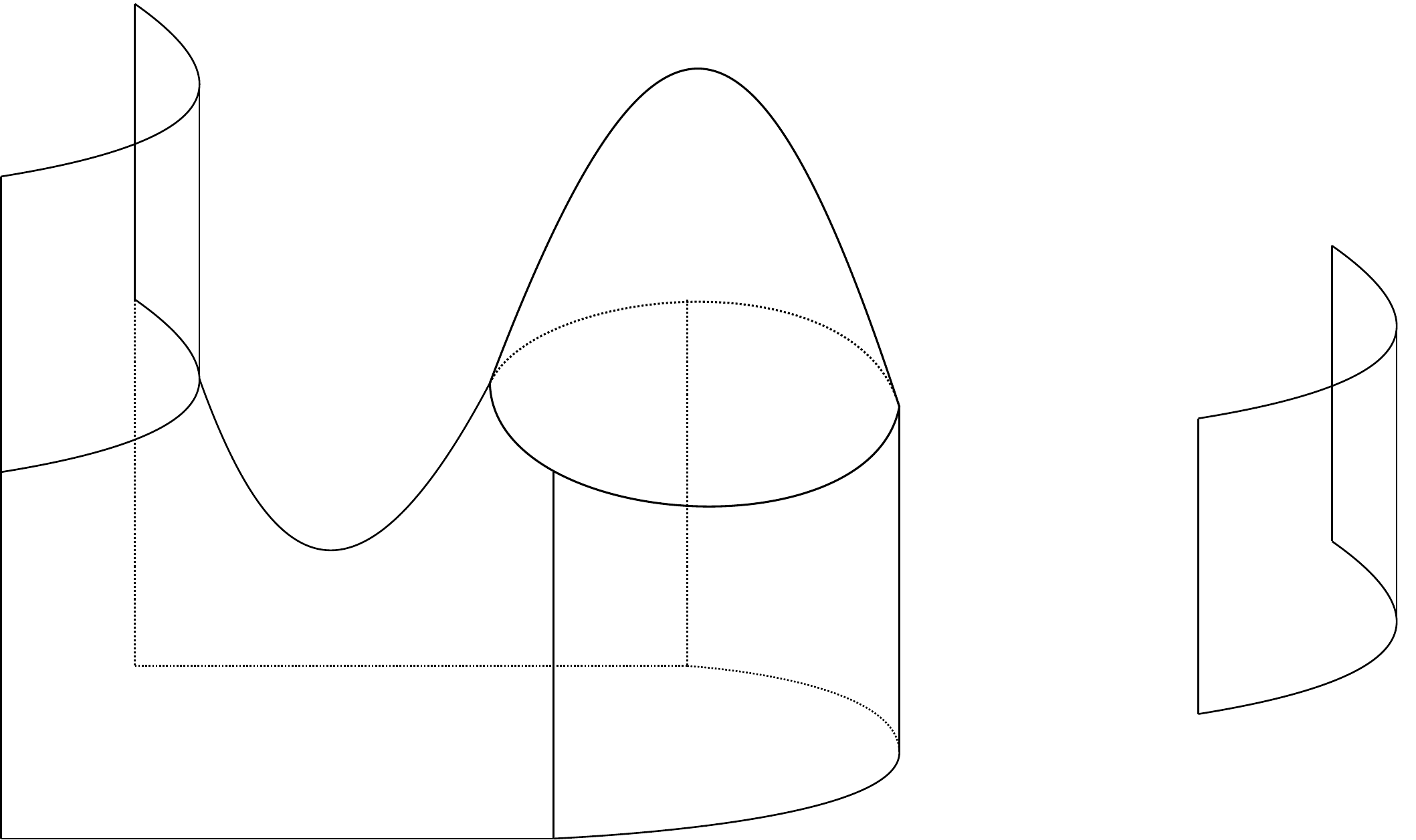}
        \put(73,20){$\cong$}
        \end{overpic}
        \end{center}
        \normalsize
	\end{figure}
    
Figure~\ref{image. u is dualizable} proves that the cap and the saddle indeed exhibit $u^\vee$ as the right adjoint to $u$. (Composition is read from top to bottom.)

\end{example}

\clearpage
\begin{example}
Let us now sketch the dualizability of an arbitrary $k$-morphism in $\Cob_n$. To that end, fix two $(k-1)$-morphisms $V$ and $W$, and let $W \leftarrow V: X$ be a cobordism from $V$ to $W$. We draw $X$ below as a jagged line, to indicate some orientation on $X$. When the jagged curve is flipped upside down, we see $X$ with the opposite orientation.

We claim that Figure~\ref{image. all cobordisms dualizable} illustrates the fact that $X^\vee$ ($X$ with the opposite orientation, read as a cobordism from $W$ to $V$) is the right dual to $X$. 
	\begin{figure}
	\caption{The dualizability of $X$:}
	\label{image. all cobordisms dualizable}
       $\,$ \\ $\,$ \\	
        \begin{center}
        \begin{overpic}[scale=0.6]{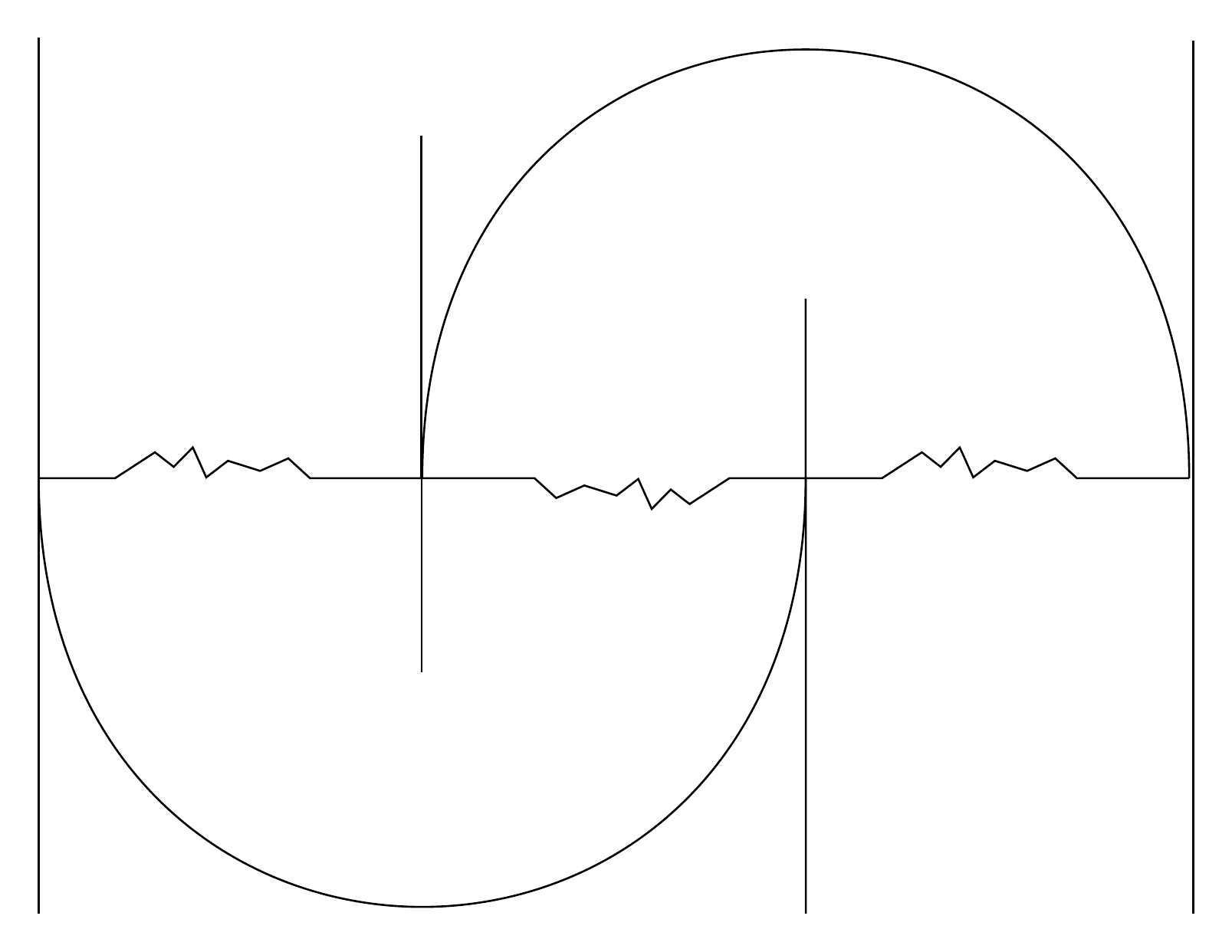}
        \put(15,77){X} 
        \put(40,77){$\coprod$} 
        \put(65,77){$\id_V$}
        \put(17,41){$X$}
        \put(48,40){$X^\vee$}
        \put(80,41){$X$}
        \put(-3,35){$W$}
        \put(30,35){$V$}
        \put(67,35){$W$}
        \put(98,35){$V$}
        \put(30,-5){$\id_W$}
        \put(60,-5){$\coprod$}
        \put(80,-5){$X$}
        \end{overpic}
       $\,$ \\ $\,$ \\
       \end{center}
	\end{figure}
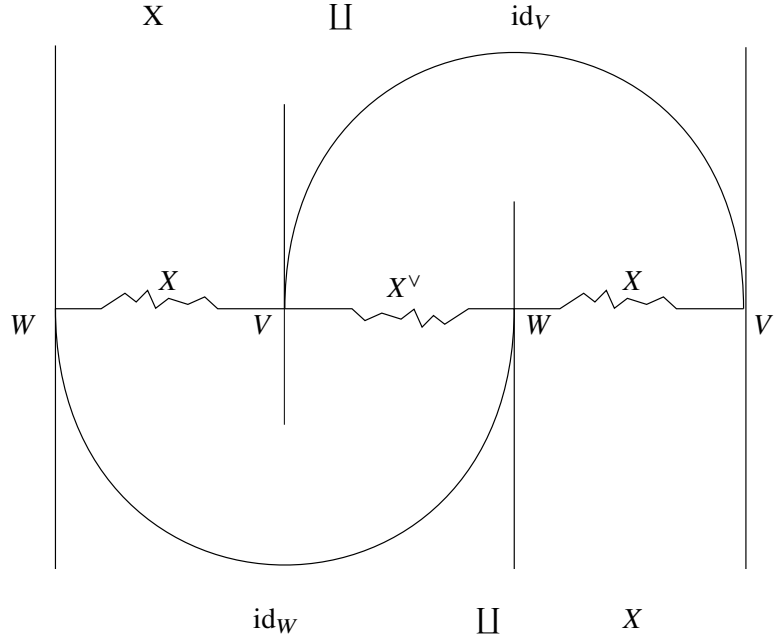        
        
Let us read the image starting in the middle, with the leftmost copy of $X$. We have ``rotated'' $X$ in a semicircle. The equator of this semicircle is hence labeled by the composition $X \circ X^\vee$. Then we have extended the rotation of $X$ by a constant copy of $W$ in the bottom-left region, so that we obtain a cobordism from $X \circ X^\vee$ to $\id_W$. This explains the lower-left region of the image. (There may of course appear to be singularities in the figure; but we can guarantee that the apparent singularities of this process takes place in a collared region of the cobordism $X$, hence this cobordism has no singularities.)

The bottom-right region is simply a copy of $X$ that is extended constantly in the vertical direction of the page. Thus, the bottom region as a whole depicts a cobordism from $X \circ X^\vee \circ X$ to $\id_W \coprod X$.

Likewise, the top-right region is obtained by rotating $X^\vee$ and extended by a constant copy of $V$, and so forth. The top half (above the jagged equator) exhibits a morphism from $X \coprod \id_V$ to $X \circ X^\vee \circ X$.

We then note that this colon-shaped picture may be isotoped to obtain a single, straight copy of the cobordism $X$; this is the sketch of a proof of one adjunction relation. We do note prove the others, leaving the sketches to the reader.
\end{example}

\begin{remark}
Informally, the image above is obtained by drawing the ``direct product'' of Zorro's Lemma with $X$, while filling in the empty regions of Zorro's Lemma with the obvious constant copies of $V$ and $W$.
\end{remark}

\begin{remark}
The reader may benefit from projecting the adjunction picture from Figure~\ref{image. u is dualizable} to $\RR^2$ to visualize the colon-shaped figure in Figure~\ref{image. all cobordisms dualizable}.  For example, the saddle is indeed obtained by extending a rotation of $u \circ u^\vee$.
\end{remark}

\clearpage
\section{Factorization homology as a topological field theory}

Now let us explain how factorization homology provides an interesting example of TFTs. For this, we must describe the target $(\infty,n)$-category. We will not give a rigorous construction of it. (The interested reader may consult~\cite{haugseng}.)

\begin{definition}[Informal.]\label{defn. morita category}
\index{Morita category}
\index{$\morita_{\EE_n}(\cD^{\tensor})$}
Let $\cD^{\tensor}$ denote a symmetric monoidal $\infty$-category. Then the {\em Morita category of $\EE_n$-algebras} in $\cD^{\tensor}$ is an $(\infty,n)$-category 
	\eqnn
	\morita_{\EE_n}(\cD^{\tensor})
	\eqnd
with the following description:
\begin{itemize}
	\item An object is an $\EE_n$-algebra in $\cD^{\tensor}$.
	\item Given two objects $A$ and $A'$, a morphism between them is an $(A,A')$ bimodule $M$, equipped with a compatible $E_{n-1}$-algebra structure.
	\item More generally, given two $k$-morphisms $M$ and $M'$, which are in particular $\EE_{n-k}$-algebras, a $(k+1)$-morphism between them is an $(M,M')$ bimodule equipped with a compatible $\EE_{n-k-1}$-algebra structure.
	\item Finally, given two $(n-1)$-morphisms $M$ and $M'$ (which are, in particular, $\EE_1$-algebras), the space of $n$-morphisms between them is the space of pointed $(M,M')$ bimodules. As a Kan complex, vertices are given by bimodules, edges are given by equivalences of bimodules, triangles are given by homotopy-commuting triangles whose edges are equivalences of bimodules, and so forth. We demand that these bimodules are all {\em pointed}, meaning that they receive maps from the monoidal unit of $\cD^{\tensor}$.
\end{itemize}
\end{definition}

\begin{remark}
Let us leave a cryptic remark. Our definition of $\EE_n$-algebras allowed us to imagine algebras as living in an open disk, say $\RR^n$. In the above, one should think of a 1-morphism $M$ as living on $\RR^n$ equipped with the data of an embedded hyperplane. $M$ itself ``lives on'' the hyperplane, while the two half-spaces on either side of the hyperplane admit embeddings from disks labeled by $A$ (for disks embedded in one half-space), or by $A'$ (for disks embedded in the other half-space). 

Likewise, a $k$-morphism $M$ is an object that one imagines living on a flag of linear subspaces of $\RR^n$, where the flag consists of planes of codimension 1 through $k$ inside $\RR^n$. The $k$-morphism itself lives on the codimension $k$ plane, and this plane cuts the codimension $(k-1)$ plane into two half-spaces. Each of these half-spaces is labeled by the $(k-1)$-morphisms that act in a way rendering $M$ a bimodule over the algebras living on these half-spaces.

For ways to make this intuition precise, we refer the reader to~\cite{gwilliam-scheimbauer} and~\cite{aft-1,aft-2}.
\end{remark}

\begin{remark}
There is another gadget one might call the Morita category of $\EE_n$-algebras, and this gadget is an $(\infty,n+1)$-category. The $n$-morphisms are given by bimodules of $\EE_1$-algebras, and given two such bimodules $M$ and $M'$ over the same pair of algebras, we define $\hom(M,M')$ to be the space of intertwiners (i.e., of bimodule maps). In particular, there are $(n+1)$-morphisms which are not invertible, as not all bimodule maps are equivalences. (Put another way, $\morita_{\EE_n}$ is obtained by throwing out all non-invertible $(n+1)$-morphisms.)

It is a far more subtle business to determine which objects are fully dualizable in this $(\infty,n+1)$-category. But as we will see below, the cobordism hypothesis will tell us that every $\EE_n$-algebra is fully dualizable in the $(\infty,n)$-category $\morita_{\EE_n}$ from Definition~\ref{defn. morita category}.
\end{remark}

Fix a positive integer $n$. We begin to see how factorization homology might define a TFT valued in $\morita_{\EE_n}$ as follows.
\begin{itemize}
	\item When given a point $\ast$, one imagines a fattened neighborhood $\ast \times \RR^n$, and we equip this with a framing. Assign to it an $\EE_n$-algebra $A$.
	\item Given a 1-morphism $X$ in $\Cob_{n}^{\fr}$, we think of it as a framed fattened neighborhood, $X \times \RR^{n-1}$. Because $X$ is a cobordism, it may be collared; but by the usual arguments we saw when discussing factorization homology, this collaring guarantees that the factorization homology of $X$ is a bimodule over its collaring boundary manifolds. So this defines a 1-morphism in $\morita_{\EE_n}$.
	\item More generally, given a $k$-morphism $X$, it is collared by two $(k-1)$-dimensional manifolds. Again thinking of everything as thickened, we view $X \times \RR^{n-k}$ as a framed $n$-manifold, and take its factorization homology. The collaring endows this with a bimodule structure receiving actions from the factorization homology of the two $(k-1)$-dimensional boundary manifolds.
	\item Finally, given a framed $n$-manifold $X$, we obtain a bimodule which is {\em pointed} by virtue of receiving a map from the empty manifold.
\end{itemize} 

The above intuition can be made precise:

\begin{theorem}[Scheimbauer~\cite{scheimbauer}]\label{theorem.facthom-TFT}
Fix an $\EE_n$-algebra $A$ in a symmetric monoidal $(\infty,1)$-category $\cD^\otimes$. 
Factorization homology defines an $n$-dimensional framed topological field theory
        \begin{align}
        (\Cob_{n}^{\fr})^{\coprod} \longrightarrow \morita_{\EE_n}(\cD^\otimes) \nonumber\\
        X^k \mapsto \int_{X^k\times \R^{n-k}} A \ \ . \nonumber
        \end{align} 
\end{theorem}

\begin{remark}
One can also prove that every $\EE_n$-algebra is fully dualizable in $\morita_{\EE_n}$ without utilizing the cobordism hypothesis, see~\cite{gwilliam-scheimbauer}.
\end{remark}

\begin{remark}
In fact, one can define notions of TFT for cobordism categories ``with colors,'' or with defects, in more physical language. An example may be obtained by marking cobordisms with marked points, or by coloring different markings and different regions of a cobordisms with various labels (called colors). One can define a TFT for such cobordism categories using factorization homology as well, by assigning modules to various defects, and various $E_n$-algebras to codimension-zero strata of $n$-manifolds (with one $E_n$-algebra for each color). While there is no written account of this defining a topological field theory, the framework of factorization homology for such manifolds was constructed in~\cite{aft-2}.
\end{remark}

\clearpage
\section{Leftovers and elaborations}

\subsection{The cobordism categories}\label{section. cobordism category}
\index{$\Cob_n$}
To at least give some indication of the answer to Question One (\ref{question: n-cats and cob_n}), we give a slightly more rigorous account of how we think of $\Cob_n$.

\begin{remark}
First, let us say that there does exist a good theory of $(\infty,n)$-categories, thanks to works of Barwick~\cite{barwick}, Rezk~\cite{rezk}, and many others. The most common model of $(\infty,n)$-categories is called an $n$-fold Segal space.
\end{remark}

\begin{remark}
It is not trivial to create a cobordism category fitting into the above framework. Let us reference that the idea of how to define such a thing was given in~\cite{lurie-tft}, and made rigorous in~\cite{calaque-scheimbauer}.
\end{remark}

Without again getting into the technical details, we give some indication of how one can start to construct an $(\infty,n)$-category of cobordisms. We first ignore framings and orientations.

The general idea is that we think of the collection of all 0-morphisms (i.e., the collection of all 0-dimensional manifolds) as the collection of all subsets $W \subset \RR^\infty$ for which $W$ happens to be a compact, smooth 0-manifold. (I.e., happens to be a finite subset.)

Then, the collection of all 1-morphisms is the collection of all subsets $X \subset \RR^\infty \times \RR$ which happen to be smooth 1-manifolds\footnote{We recall that whether a subset of Euclidean space is a manifold or not is a property of the subset, not extra structure.}. Further, we demand that
	\enum
	\item Given the projection map $X \to \RR$, $X$ is ``constant'' outside a compact interval $I =[t_0,t_1]\subset \RR$. This means $X$ is equal to a product $W_0 \times (-\infty, t_0]$ and $W_1 \times [t_1, \infty)$ as a subset of $\RR^\infty \times \RR$. 
	\item Moreover, the map $X \to \RR$ is a submersion at $t_0$ and $t_1$, which guarantees that $W_0$ and $W_1$ are smooth submanifolds of $\RR^\infty$. (This also guarantees the existence of collars.)
	\item We finally demand that the projection $X \to \RR$ is proper, so that in particular the preimage of $[t_0,t_1]$ is a compact subset of $\RR^\infty \times [t_0,t_1]$.
	\enumd

Likewise, we think of the collection of $k$-morphisms as the collection of all subsets $Q \subset \RR^\infty \times \RR^k$ which happens to be a smooth submanifold, and which satisfy analogous properties to the above. Informally, one demands the existence of some $k$-dimensional compact cube in $\RR^k$ above which $Q$ is collared in a standard way that allows us to compose $Q$ along the obvious faces of the cube. Note also that we demand that, if $t_i$ labels the $i$th coordinate of $\RR^k$, the cube's $t_i = \text{constant}$ face is collared by a manifold which is ``constant'' in all directions $t_j$ with $j > i$. (The asymmetry induced by this $j>i$ condition is what allows us to say that $Q$ is a $k$-morphism.)

By doing this for all $0 \leq k \leq n$, we have more or less specified the necessary data. For example, given a $k$-morphism $Q \subset \RR^\infty \times \RR^k$, the restriction of $Q$ to the various faces of the $k$-dimensional cube in $\RR^k$ tells us what the source morphisms and target morphisms of $Q$ are. 

\begin{remark}
We note that the collection of subsets of $\RR^\infty \times \RR^k$ has a topology---it makes sense to say when two compact subsets are nearby (or, when two ``constant outside a compact subset'' are nearby). And the operation of restricting a cobordism to a face of a cube is a continuous operation. This topology is part of the reason we have the fancy symbol ``$\infty$'' in describing $\Cob_n$ as an $(\infty,n)$-category.
\end{remark}

\begin{remark}
Let us at least motivate why one might want to consider such a thing. Consider the collection $\cE nd_n(\emptyset)$ of $n$-morphisms between the empty cobordisms; i.e., the collection of subsets of $\RR^\infty \times \RR^n$ which happen to be compact smooth manifolds of dimension $n$. Choosing a particular $n$-dimensional submanifold $X \subset \RR^\infty \times \RR^n$, we pick out a connected component of $\cE nd_n(\emptyset)$.

We claim that this connected component is interesting: It is homotopy equivalent to $B\diff X$, the classifying space for the space of diffeomorphisms of $X$.

To see this, consider the space of all smooth embeddings $j: X \to \RR^\infty \times \RR^n$. By standard arguments in smooth topology, this space is contractible. On the other hand, clearly the space of diffeomorphisms of $X$ acts on the space of all $j$ by precomposition. This action is free because each $j$ is an injection. On the other hand, if I quotient the collection of embeddings by the collection of reparametrizations, I obtain the collection of all possible images---i.e., the collection of all subsets of $\RR^\infty \times \RR^n$ that happen to be diffeomorphic to $X$. This is exactly the connected component of $\cE nd_n(\emptyset)$ we picked out.

And, of course, the quotient of a contractible space by a free continuous action of a group $G$ is precisely a model for the space $BG$.

In particular, this illustrates how an $n$-dimensional TFT $Z: \Cob_n \to \cC$ assigns to $X$ an object with the action of $\diff(X)$ (or, of the orientation-preserving diffeomorphisms, for example, if we demand all our manifolds be oriented).
\end{remark}

\subsection{Framings}\label{section. framed cobordisms}
\index{framing}
Let us also elaborate on tangential structures. Fix a group homomorphism $G \to GL_n(\RR)$. 

\begin{remark}
The ``$n$'' here is the same $n$ as in $\Cob_n$.
\end{remark}

For a given $k$-dimensional manifold $X$, one has a map $\tau_X : X \to BGL_k(\RR)$ classifying the tangent bundle of $X$. (See Section~\ref{section.framings}.) If $X$ is a $k$-manifold in $\RR^\infty \times \RR^k$, we define a $G$-structure on $X$ to be a homotopy-commutating diagram
	\eqn\label{eqn. n dimensional G structure}
	\xymatrix{
		&& BG \ar[d] \\
	X \ar[r]_-{\tau_X} \ar[urr]
		& BGL_k(\RR) \ar[r]
		& BGL_n(\RR).
	}
	\eqnd
Here, $BGL_k(\RR) \to BGL_n(\RR)$ is the map induced by the usual inclusion $GL_k(\RR) \to GL_n(\RR)$. 

Put a more concrete way, consider the rank $n$ bundle $TX \oplus \underline{\RR^{n-k}}$, where $\underline{\RR^{n-k}}$ is a trivialized bundle on $X$ of rank $n-k$.  Then a $G$-structure is a reduction of structure group of this bundle to $G$.

\begin{warning-numbered}
The $n$ lurking in the background is important. For example, suppose $n > k$.
When $G = \ast$, a $G$-structure on a $k$-manifold $X$ is {\em not} the same thing as a framing on $X$ in the sense of Section~\ref{section.framings}. There, we trivialized  $TX$ itself. Here, we are trivializing a higher-rank bundle, $TX \oplus \underline{\RR^{n-k}}$. Indeed, many $k$-manifolds that do not admit a framing may admit a $G$-structure for $n$ large enough.
\end{warning-numbered}

To heed this warning, we define the following:

\begin{definition}\label{defn.n-dim framing}
Fix $n > k$ and fix a $k$-dimensional manifold $X$. An {\em $n$-dimensional framing} of $X$ is a trivialization of the bundle $TX \oplus \underline{\RR^{n-k}}$. 

More generally, fix a continuous group homomorphism $G \to GL_n(\RR)$. Then data as in~\eqref{eqn. n dimensional G structure}, is called an $n$-dimensional $G$-structure. 
\end{definition}

\begin{remark}
Fix $G = \ast$ to be the trivial group---i.e., the case of a framing.
It is often said that an $n$-dimensional framing could be thought of as a framing on a small trivial neighborhood $X \times \RR^{n-k}$ of $X$. If you prefer to think this way, let me caution that it is most natural to think of an $n$-dimensional framing as a framing on $X \times \RR^{n-k}$ such that the trivialization factors through the collapse map $X \times \RR^{n-k} \to X$. Informally, such a framing is ``constant'' along the $\RR^{n-k}$ directions, and that is the kind of framing that an $n$-dimensional framing captures.
\end{remark}

So let us retroactively set:

\begin{definition}\label{defn. cobordism framing}
Fix $n > k$.
A $k$-morphism of $\Cob_{n}^{\fr}$ is a $k$-dimensional submanifold $X \subset \RR^\infty \times \RR^k$ such that $X$ is equipped with an $n$-dimensional framing.

Note that the boundary $\del X \subset X$ comes equipped with a canonical isomorphism $T(\del X) \oplus \underline{\RR} \to TX|_{\del X}$, and hence we have an isomorphism 
	\eqnn
	T(\del X) \oplus \RR^{n-k + 1}
	\cong
	TX|_{\del X} \oplus \RR^{n-k}
	\eqnd
and this allows us to articulate what the sources and targets---as $n$-dimensionally framed manifolds---of a $k$-morphism are.
\end{definition}

\begin{remark}
This definition has real consequences. For example, we should still prove that the point is fully dualizable in $\Cob_2^{\fr}$. Then what is the right adjoint to $u$ as described in Example~\ref{example. point is 2-dualizable}?

First, we must specify ``which'' framed horseshoe $u$ we mean, as $Tu \oplus \underline{\RR}$ admits several inequivalent framings once we have fixed the 2-dimensional framing data on the boundary points. 

Further, whatever the right adjoint $r$ is, the composition $r \circ u$ should be exhibited as the boundary of a disk (the cap), but this must be done {\em in a framed way}. Because not every 2-dimensional framing of a circle extends to a framing of the tangent bundle of the cap, this puts a restriction on the possible 2-dimensional framing(s) that $r$ can be equipped with.
\end{remark}

\subsection{Structure groups and homotopy fixed points}
Finally, we note that we can always alter an $n$-dimensional framing on a $k$-manifold $X$ by acting on $\underline{\RR^{n}}$ by $GL_n(\RR)$. From this observation, one can in fact exhibit a continuous action of $GL_n(\RR)$ on the category $\Cob_{n}^{\fr}$ itself.

In particular, given a group homomorphism $G \to GL_n$, one can ask for {\em $G$-fixed points} of the induced action on $\Cob_{n}^{\fr}$. 

\begin{warning-numbered}
We warn the reader that here, by a $G$-fixed point we mean a functor from $EG$ to $\Cob_{n}^{\fr}$ extending a functor $G \to \Cob_{n}^{\fr}$ (given by picking out an object). That is, exhibiting a $G$-fixed point in the homotopical world is not finding an object of $\Cob_{n}^{\fr}$, but exhibiting a way in which we can trivialize the $G$-action along a particular orbit of the $G$-action on $\Cob_{n}^{\fr}$.

To emphasize this distinction, what we are calling a fixed point is sometimes called a homotopy fixed point.

We illustrate this point. The reader may assume $G$ is a discrete group, or should otherwise topologize the collection of simplices by the topology of $G$.

$EG$ can be described combinatorially; it has $G$ many vertices, $G \times G$ many edges, $G^l$ many $l$-simplices, and so forth. Informally, one obtains $EG$ by constructing the ``complete Cayley complex'' on the generating set $G$ itself; unlike a Cayley graph which inserts only edges for every $G$-translation, we insert triangles and higher simplices for every sequence of $G$-translations. It follows straightforwardly that $EG$ is contractible and enjoys a free $G$ action. 

If $G$ acts on a category $\cC$, any object $X \in \cC$ determines a functor $G \to \cC$ by acting on $X$. An extension of $G \to \cC$ to $EG$ (i.e., a factorization through $EG$) in particular determines equivalences $X \xrightarrow{\cong} gX$ for any $g \in G$, and these equivalences are compatible with multiplication in $G$ by definition of the higher simplices of $EG$. 

Thus, while $X$ may not {\em equal} $gX$ on the nose, they may abstractly be isomorphic, and the functor from $EG$ specifies these isomorphisms in a way coherent with the $G$-action.
\end{warning-numbered}

As above, fix a continuous homomorphism $G \to GL_n(\RR)$. Since $G$ then acts on $\Cob_n^{\fr}$, we obtain a $G$-action on the space of TFTs; by the Cobordism hypothesis, this induces a $G$-action on the space of fully dualizable objects of the target $\cC^{\tensor}$. 

\begin{example}
When $n=1$ and $\cC^{\tensor} = \Vect^{\tensor_k}$, the $GL_1(\RR) \simeq O(1)$-action sends a finite dimensional vector space $V$ to its dual. Note the necessity of restricting to fully dualizable objects (and isomorphisms between them; not all morphisms). Sending a vector space to its dual is usually a contravariant operation.
\end{example}

This action allow us to identity TFTs with any $G$-structure:

\begin{theorem}[Cobordism hypothesis for $G$-structures, Theorem~2.4.26 of~\cite{lurie-tft}]
\index{cobordism hypothesis!for $G$-structures}
\index{$G$-structures}
Fix a continuous group homomorphism $G \to GL_n(\RR)$ and let $\Cob_{n}^G$ be the $(\infty,n)$-category of cobordism with $n$-dimensional $G$-structures. Then there is an equivalence of $\infty$-categories
	\eqnn
	\Fun^{\tensor}((\Cob_{n}^G)^{\coprod}, \cC^{\tensor})
	\simeq
	(\cC^{f.d.})^{G}
	\eqnd
between the $\infty$-category of fully extended $n$-dimensional TFTs for $G$-structures, and the space of $G$ fixed points of the space $\cC^{f.d.}$.
\end{theorem}

\begin{example}
Let us consider the $(\infty,2)$-category $\cC$ of $\infty$-categories with $\tensor = \times$, and suppose we have a fully dualizable object $D$. (In the $\CC$-linear setting, an example of such a thing is the dg-category $D^bCoh$ of a smooth and proper variety.)

Fixing $D$, we have a map $O(2) \to \cC^{f.d.}$, and in particular a map from $SO(2)$ to $\cC^{f.d.}$. A loop based at the identity in $SO(2) \simeq S^1$ is thus the data of an automorphism of $D$; a generator of $\pi_1(S^1) \cong \ZZ$ gives a distinguished automorphism.

When $\cC$ is the $(\infty,2)$-category of $\CC$-linear $\infty$-categories with $\tensor = \tensor_\CC$, let us fix $D = D^bCoh(Y)$ for some smooth and proper complex variety $Y$. Then this automorphism may be identified with the Serre automorphism.
\end{example}

\clearpage
\section{Exercises}

\begin{exercise}\label{exercise. cob hyp in dim 1}
Show that an object of $\Vect_{\kk}^{\tensor k}$, thought of as a symmetric monoidal $\infty$-category as usual, is fully dualizable if and only if it is a finite-dimensional $\kk$-vector space.

Verify the cobordism hypothesis in this case; be careful in proving that any natural transformation of symmetric monoidal functors $Z \to Z'$ must actually be a natural {\em isomorphism}.
\end{exercise}

\begin{exercise}
In the case $G = O(1)$, consider $\Cob_{1}^G$. (The cobordism category of unoriented 0- and 1-dimensional manifolds.) Verify the $G$-structured version of the cobordism hypothesis.

In this case, show that a $G$-structured TFT is the same thing as a vector space $V$ equipped with a {\em symmetric non-degenerate pairing}. (Hint: Identify $V$ with $W$ using the $O(1)$-fixed point structure.)
\end{exercise}

\begin{exercise}
Verify Remark~\ref{remark. units are dualizable and left-right duals}.
\end{exercise}

\begin{exercise}
Let $\cC^{\tensor}$ be a symmetric monoidal $(\infty,n)$-category (which we have not yet fully defined).  

Convince yourself that any equivalence $k$-morphism (i.e., any equivalence between two $(k-1)$-morphisms) is $(k+1)$-dualizable.

Convince yourself that the unit of $\tensor$ is fully dualizable.
\end{exercise}

\begin{exercise}
Let $\ast^+$ be a point with some fixed framing, considered as an object of $\Cob_{n}^{\fr}$. Prove that the space of self-equivalences of $\ast^+$ in $\Cob_{n}^{\fr}$ is homotopy equivalent to $\Omega GL_n(\RR)$, the based loop space of $GL_n(\RR)$ at the identity matrix. (Note that this is in turn homotopy equivalent to $\Omega SO_n(\RR)$.)

You will want to read Section~\ref{section. framed cobordisms} to think carefully about this exercise.
\end{exercise}

\begin{exercise}
Fix $n=2$ and fix a point $\ast^+ \in \Cob_2^{\fr}$. Draw pictures of the automorphisms of $ \ast^+$; note that the space of automorphisms is given by $\Omega SO(2) \simeq \Omega S^1 \simeq \ZZ$. Can you quantify/describe the sense in which there is an integers' worth of invertible cobordisms from a point to itself?
\end{exercise}

\begin{exercise}
In stating the cobordism hypothesis, we have claimed that the collection of TFTs forms a {\em space}, which informally means that every symmetric monoidal natural transformation $Z \to Z'$ is actually a natural equivalence.  (See Remark~\ref{remark. TFT maps are equivalences}.)

Convince yourself of this fact.
\end{exercise}

\begin{exercise}
Fix $n \geq 1$ and let $\morita_n$ denote the Morita $(\infty,n)$-category (Definition~\ref{defn. morita category}). Let $\ZZ/2\ZZ \cong \pi_0 O(n)$ denote the group of orientations. Convince yourself that the action of $\ZZ/2\ZZ$ from the cobordism hypothesis sends an $\EE_n$-algebra $A$ to its opposite algebra, $A^{\op}$.
\end{exercise}

\begin{exercise}[Morita invariance for compact, top-dimensional manifolds]
Fix a symmetric monoidal $\infty$-category $\cD^{\tensor}$ admitting sifted colimits, for which $\tensor$ preserves sifted colimits in each variable.

Recall Scheimbauer's result (Theorem~\ref{theorem.facthom-TFT}), that factorization homology defines a framed, fully extended $n$-dimensional TFT with values in the Morita $(\infty,n)$-category $\morita_{\EE_n}(\cD^{\tensor})$ of $\EE_n$-algebras  (Definition~\ref{defn. morita category}).

Now assume that two $E_n$-algebras $A$ and $B$ are Morita equivalent. 
\index{Morita equivalence}
That is, there exist two $E_{n-1}$-algebras $M$ and $N$ that are $(A,B)$ and $(B,A)$-bimodules, respectively, equipped with equivalences 
	\eqn\label{eqn:morita-equivalence}
	M \tensor_B N \simeq A,
	\qquad
	N \tensor_A M \simeq B
	\eqnd
{\em in the Morita category}, meaning we are supplied with a sequence of further bimodules (for example, between $M \tensor_B N$ and $A$) exhibiting the Morita equivalences above, all the way until we reach an actual equivalence of objects in $\cD$ as bimodules between $\EE_1$-algebras.

(For instance, when $n=1$, the equivalences in~\eqref{eqn:morita-equivalence} are actual equivalences of objects in $\cD$, while if $n=2$, the equivalences in~\eqref{eqn:morita-equivalence} are equivalences exhibited by invertible bimodules; in particular, these need not imply $M \tensor_B N$ and $A$ are equivalent as objects in $\cD$.)

Show that if two $E_n$-algebras $A$ and $B$ are Morita equivalent, then for any compact framed $n$-manifold $X$, we have an equivalence
	\eqnn
	\int_X A \simeq \int_X B
	\eqnd
as objects in $\cD$.
\end{exercise}

\backmatter

\printindex

\end{document}